\documentclass[12pt]{amsart}
\usepackage{times}

\usepackage{comment}
\usepackage{amssymb}
\usepackage{color}
\reversemarginpar
\begin{document}
\numberwithin{equation}{section}

\def\1#1{\overline{#1}}
\def\2#1{\widetilde{#1}}
\def\3#1{\widehat{#1}}
\def\4#1{\mathbb{#1}}
\def\5#1{\frak{#1}}
\def\6#1{{\mathcal{#1}}}

\newcommand{\de}{\partial}
\newcommand{\R}{\mathbb R}
\newcommand{\al}{\alpha}
\newcommand{\tr}{\widetilde{\rho}}
\newcommand{\tz}{\widetilde{\zeta}}
\newcommand{\tv}{\widetilde{\varphi}}
\newcommand{\tO}{\widetilde{\Omega}}
\newcommand{\hv}{\hat{\varphi}}
\newcommand{\tu}{\tilde{u}}
\newcommand{\usc}{{\sf usc}}
\newcommand{\tF}{\tilde{F}}
\newcommand{\debar}{\overline{\de}}
\newcommand{\Z}{\mathbb Z}
\newcommand{\C}{\mathbb C}
\newcommand{\Po}{\mathbb P}
\newcommand{\zbar}{\overline{z}}
\newcommand{\G}{\mathcal{G}}
\newcommand{\So}{\mathcal{U}}
\newcommand{\Ko}{\mathcal{K}}
\newcommand{\U}{\mathcal{U}}
\newcommand{\B}{\mathbb B}
\newcommand{\oB}{\overline{\mathbb B}}
\newcommand{\Cur}{\mathcal D}
\newcommand{\Dis}{\mathcal Dis}
\newcommand{\Levi}{\mathcal L}
\newcommand{\SP}{\mathcal SP}
\newcommand{\Sp}{\mathcal Q}
\newcommand{\Ma}{\mathcal M}
\newcommand{\Co}{\mathcal C}
\newcommand{\Hol}{{\sf Hol}(\mathbb H, \mathbb C)}
\newcommand{\Aut}{{\sf Aut}(\mathbb D)}
\newcommand{\D}{\mathbb D}
\newcommand{\oD}{\overline{\mathbb D}}
\newcommand{\oX}{\overline{X}}
\newcommand{\loc}{L^1_{\rm{loc}}}
\newcommand{\loci}{L^\infty_{\rm{loc}}}
\newcommand{\la}{\langle}
\newcommand{\ra}{\rangle}
\newcommand{\thh}{\tilde{h}}
\newcommand{\N}{\mathbb N}
\newcommand{\kd}{\kappa_D}
\newcommand{\Ha}{\mathbb H}
\newcommand{\ps}{{\sf Psh}}
\newcommand{\tg}{\widetilde{\gamma}}

\newcommand{\subh}{{\sf subh}}
\newcommand{\harm}{{\sf harm}}
\newcommand{\ph}{{\sf Ph}}
\newcommand{\tl}{\tilde{\lambda}}
\newcommand{\ts}{\tilde{\sigma}}

\renewcommand{\epsilon}{\varepsilon}
\renewcommand{\kappa}{\varkappa}
\newcommand{\lam}{\lambda}

\newcommand{\eps}{\epsilon}
\newcommand{\del}{\delta}
\newcommand{\Om}{\Omega}
\newcommand{\LL}{\mathcal{L}}
\newcommand{\LM}{\mathcal{LM}}
\newcommand{\cl}{\text{cl}}

\newcommand{\mg}{\color{magenta}}

\definecolor{blue-violet}{rgb}{0.54, 0.17, 0.89}

\def\cD{{\mathcal D}}

\def\v{\varphi}
\def\Re{{\sf Re}\,}
\def\Im{{\sf Im}\,}

\def\dist{{\rm dist}}
\def\const{{\rm const}}
\def\rk{{\rm rank\,}}
\def\id{{\sf id}}
\def\aut{{\sf aut}}
\def\Aut{{\sf Aut}}
\def\CR{{\rm CR}}
\def\GL{{\sf GL}}
\def\U{{\sf U}}

\def\la{\langle}
\def\ra{\rangle}

\newtheorem{theorem}{Theorem}[section]
\newtheorem{lemma}[theorem]{Lemma}
\newtheorem{proposition}[theorem]{Proposition}
\newtheorem{corollary}[theorem]{Corollary}

\theoremstyle{definition}
\newtheorem{definition}[theorem]{Definition}
\newtheorem{example}[theorem]{Example}

\theoremstyle{remark}
\newtheorem{remark}[theorem]{Remark}
\numberwithin{equation}{section}

%

\title[Complete frequencies]{Complete frequencies for Koenigs domains}

\author[F. Bracci]{Filippo Bracci}
\author[E. A. Gallardo-Guti\'errez]{Eva A. Gallardo-Guti\'errez}
\author[D. Yakubovich]{Dmitry Yakubovich}

\address{F. Bracci\newline
Dipartimento Di Matematica\newline
Universit\`{a} di Roma \textquotedblleft Tor Vergata\textquotedblright\ \newline
Via Della Ricerca Scientifica 1, 00133 \newline
Roma, Italy}
\email{fbracci@mat.uniroma2.it}
\thanks{The first author (F. Bracci) is partially supported by PRIN (2022) Real and Complex Manifolds: Geometry and holomorphic dynamics Ref:2022AP8HZ9, by GNSAGA of INdAM and  by the MUR Excellence Department Project 2023-2027 MatMod@Tov CUP:E83C23000330006 awarded to the Department of Mathematics, University of Rome Tor Vergata}

\address{E. A. Gallardo-Guti\'errez \newline Departamento de An\'alisis Matem\'atico y\newline
Matem\'atica Aplicada \newline
Facultad de CC. Matem\'aticas,
\newline Universidad Complutense de
Madrid, \newline
 Plaza de Ciencias 3, 28040 Madrid,  Spain
 \newline
and Instituto de Ciencias Matem\'aticas ICMAT(CSIC-UAM-UC3M-UCM),
\newline Madrid,  Spain }
\email{eva.gallardo@mat.ucm.es}

\address{
D. V. Yakubovich\newline
Departamento de Matem\'aticas,\newline  Universidad Aut\'onoma de Madrid,\newline
Cantoblanco, 28049 Madrid, Spain }
\email {dmitry.yakubovich@uam.es}
\thanks{The second author (E. A. Gallardo-Guti\'errez) and the third author (D. Yakubovich) are partially supported by Plan Nacional  I+D grant no. PID2022-137294NB-I00, Spain.
\newline
The authors also acknowledge support from ``Severo Ochoa Programme for Centres of Excellence in R\&D'' (CEX2019-000904-S \& CEX2023-001347) of the
Ministry of Economy and Competitiveness of Spain and by the European Regional Development
Fund and from the Spanish
National Research Council, through the ``Ayuda extraordinaria a
Centros de Excelencia Severo Ochoa'' (20205CEX001).
}


\begin{abstract}
We provide a complete characterization of those non-elliptic semigroups of holomorphic self-maps of the unit disc for which the linear span of eigenvectors of the generator of the corresponding semigroup of composition operators is weak-star dense in $H^\infty$. We also give some necessary and some sufficient conditions for completeness in $H^p$. This problem is equivalent to the completeness of the corresponding exponential functions in
$H^\infty$ (in the weak-star sense) or in $H^p$ of the Koenigs domain of the semigroup. As a tool needed for the results, we introduce and study discontinuities of semigroups of holomorphic self-maps of the unit disc.
\end{abstract}

\maketitle
\tableofcontents

\section{Introduction}

The investigation into the density of a natural class of functions in classical $H^p$-spaces represents a significant problem, akin to a broader exploration of classical approximation theorems such as those by Runge and Mergelyan \cite{FF}. Notably, this inquiry serves as fundamental extensions of classical approximation theorems. Several notable contributions have been made to this field, as evidenced by the works of \cite{Ak0, Ak, Bre, Do, Ru-Sh, Vo1, Vo2}, particularly concerning polynomial approximation. Nikolskii's monograph \cite{Nik} stands out as a valuable source of results in this regard.

One compelling avenue of investigation involves the study of weak-star density of polynomials in $H^\infty(G)$ for simply connected domains $G$, leading to the conceptualization of ``weak-star generators of $H^\infty$". This notion, fully characterized in topological terms by Sarason \cite{Sa, Sa1}, has proven to be instrumental not only in understanding the density of polynomials but also in the broader context of operator theory (see \cite{Bo}). The exploration of this topic is of paramount importance, linking classical approximation theorems to modern operator theory developments.

In this paper, we extend the ongoing investigation by exploring approximation  achieved through the frequencies associated with holomorphic flows (semigroups), providing a complete picture for the weak-star topology of $H^\infty$ and sufficient conditions in $H^p$.

In order to better describe our results, we recall that every continuous semigroup $(\varphi_t)_{t\geq 0}$ of holomorphic self-maps of the unit disc $\D$, or, for short, a semigroup in $\D$, gives rise to a semigroup $\{C_{\varphi_t}\}_{t\geq 0}$ of strongly continuous composition operators in the classical Hardy space of the unit disc $H^p(\D)$, $p\in[1,\infty)$.

In case the semigroup is non-elliptic, the eigenvalues of the generator $\Gamma$ of $\{C_{\varphi_t}\}_{t\geq 0}$ are those $\lambda\in\C$ such that $e^{\lambda h}\in H^p(\D)$, where $h$ is the Koenigs function of $(\varphi_t)_{t\geq 0}$. In other words, if $\Omega=h(\D)$ is the Koenigs domain of $(\varphi_t)_{t\geq 0}$, $\lambda\in\C$ is an eigenvalue of $\Gamma$ if and only if $e^{\lambda z}\in H^p(\Omega)$. Such $\lambda$'s  will be called {\sl frequencies}. We refer the reader to Section~\ref{Sec:Preliminaries} for definitions and details.

Betsakos \cite{Be} conducted a meticulous study, employing harmonic measure theory and considering the nature of $(\varphi_t)_{t\geq 0}$ (whether hyperbolic, parabolic with a positive hyperbolic step, or parabolic with a zero hyperbolic step), aimed to determine the values of $\lambda\in\C$ for which $e^{\lambda z}\in H^p(\Omega)$ and provided detailed estimates.

The aim of this paper is to study {\sl completeness} of such exponentials. More precisely, for $p\in [1,+\infty]$, let
\[
\mathcal E_p(\Omega):=\hbox{span}\{e^{\lambda z}: \lambda\in\C, e^{\lambda z}\in H^p(\Omega)\}.
\]
The precise questions we address here are:

\medskip

\begin{itemize}
\item Given $p\in [1,\infty)$, when is $\mathcal E_p(\Omega)$  dense in $H^p(\Omega)$?
\item For $p=\infty$,  when is $\mathcal E_\infty(\Omega)$  weak-star dense in $H^\infty(\Omega)$?
\end{itemize}
\medskip

In case $\mathcal E_p(\Omega)$ is dense in $H^p(\Omega)$ (or $\mathcal E_\infty(\Omega)$  weak-star dense in $H^\infty(\Omega)$)  we say that the semigroup $(\varphi_t)_{t\geq 0}$ has $p$-complete (or $\infty$-weak-star complete) frequencies.

In this paper we completely characterize when $\mathcal E_\infty(\Omega)$ is weak-star dense in $H^\infty(\Omega)$, in terms of  the topological properties of the Koenigs domain,  its geometrical properties and the dynamical properties of the (non-elliptic) semigroup. We also provide sufficient conditions, and some necessary, for $p$-completeness.

It is here worth emphasizing that when $(\varphi_t)_{t\geq 0}$ is elliptic, similar queries consistently yield negative results. This is due to the existence of a common fixed point for all the functions associated to frequencies of the infinitesimal generator of the composition operator semigroups.

In order to address the previous questions, we start  this paper with a careful study of the set $\mathcal S(h)\subset\partial\D$ of boundary discontinuities of the Koenigs function $h$ of the semigroup $(\varphi_t)_{t\geq 0}$ (see Subsection~\ref{subsect:discont} and Subsection~\ref{Subsect:bdd-disc} for details). Here $\partial\D$ denotes the boundary of $\D$. Except at the Denjoy-Wolff point, the discontinuities of $h$ correspond to the discontinuities of $(\varphi_t)_{t\geq 0}$. Among such discontinuities, we focus on the {\sl unbounded ones}, $\mathcal S_u(h)$ for which (with the exception of the Denjoy-Wolff point)
%
%
the impression of the corresponding prime end contains $\infty$.
The Denjoy-Wolff point can be a discontinuity in the hyperbolic and parabolic of positive hyperbolic step cases. If this is the case, and only in the hyperbolic case, the impression in $\C$ of the corresponding prime end  under $h$, is contained in the union of two lines. We say that the Denjoy-Wolff point is a {\sl simple discontinuity} if the impression does not coincide with the union of the two lines.

Also, we introduce {\sl Cantor combs} of discontinuities, which are discontinuities on $\partial\D$ which form a Cantor set  and whose image under the Koenigs function looks like a comb (see Section~\ref{Sec:char1}).

An important role in our characterization is also played by maximal contact arcs (see Subsection~\ref{Subsec:maximal-arcs}), introduced and studied in \cite{BG}. Roughly speaking, a maximal contact arc is an  open arc on $\partial\D$ such that the semigroup extends holomorphically through it and flows the arc from the starting point to the final point. A maximal contact arc is exceptional if its final point is the Denjoy-Wolff point.

Every Koenigs domain $\Omega$ of a non-elliptic semigroup can be described as $$\{x+iy\in \C: x>\psi(y)\},$$ for some upper semicontinuous function $\psi:I\to [-\infty,+\infty)$ defined on an open interval $I\subseteq(-\infty,+\infty)$. We let $\tilde{\psi}$ to be the upper semicontinuous regularization of the lower semicontinuous regularization of $\psi$ (see Definition~\ref{Def:psi-tilde}). The functions $\psi$ and $\tilde\psi$ encode the geometry of $\Omega$.

With these definitions at hand, we can state our main results. We start with the hyperbolic case:

\begin{theorem}\label{Thm:main-hyp}
Let $(\varphi_t)_{t\geq 0}$ be a hyperbolic semigroup in $\D$ with Denjoy-Wolff point $\tau\in\partial\D$. Let $h$ be the Koenigs  function, $\Omega=h(\D)$ the Koenigs  domain  and let $\psi:I\to \R$ be the defining function of $\Omega$.   The following conditions are equivalent:
\begin{enumerate}
\item $(\varphi_t)_{t\geq 0}$ has $\infty$-weak-star complete frequencies, that is, $\mathcal E_\infty(\Omega)$ is weak-star dense in $H^\infty(\Omega)$,
\item $\hbox{Int}(\overline{\Omega})=\Omega$ and $\C\setminus \overline{\Omega}$ has at most two connected components,
\item $\psi(y)=\tilde{\psi}(y)$ for every $y\in I$ and either $\liminf_{y\to y_0}\psi(y)>-\infty$ for all $y_0\in \overline{I}$ or there exist $a\leq b$ such that $[a,b]\subset \overline{I}$,  $\liminf_{y\to y_0}\psi(y)=-\infty$ for all $y_0\in [a,b]$ and $\liminf_{I\ni y\to y_1}\psi(y)>-\infty$ for all $y_1\in\overline{I}\setminus [a,b]$,
\item the semigroup $(\varphi_t)_{t\geq 0}$ satisfies  the following conditions:
\begin{itemize}
\item[(a)] if $A, B$ are non-exceptional maximal contact arcs for $(\varphi_t)_{t\geq 0}$ with final points $x_1(A)$ and $x_1(B)$ then   $x_1(A)\neq x_1(B)$,
\item[(b)] $\mathcal S(h)$ does not contain any Cantor comb,
\item[(c)] one  of the following (mutually excluding) conditions holds:
\begin{itemize}
\item[(c.1)] $(\varphi_t)_{t\geq 0}$ has at most one boundary fixed point other than $\tau$ and no unbounded discontinuities,
\item [(c.2)] $(\varphi_t)_{t\geq 0}$ has  one and only one boundary fixed point $\sigma\in\partial\D\setminus\{\tau\}$, $\mathcal S_u(h)=\{\xi\}$ with $\xi\neq \tau$, and there exists a maximal contact arc starting from $\sigma$ and ending at $\xi$,
\item [(c.3)] $(\varphi_t)_{t\geq 0}$ has  one and only one boundary fixed point $\sigma\in\partial\D\setminus\{\tau\}$ which is a  boundary regular fixed point, $\mathcal S_u(h)=\{\xi_1, \xi_2\}$ with $\xi_j\neq \tau$, and for $j=1,2$,
there is a maximal contact arc starting from $\sigma$ and ending at $\xi_j$.
\item[(c.4)]  $(\varphi_t)_{t\geq 0}$ has no boundary fixed points other than $\tau$ and $\mathcal S_u(h)=\{\xi\}$, with $\xi\neq \tau$,
\item[(c.5)]  $(\varphi_t)_{t\geq 0}$ has no boundary fixed points other than $\tau$ and $\mathcal S_u(h)=\{\xi_1, \xi_2\}$, with $\xi_j\neq \tau$, $j=1,2$ and there exists a maximal contact arc  starting from $\xi_1$ and ending at $\xi_2$,
\item[(c.6)]  $(\varphi_t)_{t\geq 0}$ has no boundary fixed points other than $\tau$, $\mathcal S_u(h)=\{\tau\}$ and $\tau$ is a simple unbounded discontinuity.
\end{itemize}
\end{itemize}
\end{enumerate}
Moreover, if one---and hence any---of the above conditions
(1)--(4) holds, then $(\varphi_t)_{t\geq 0}$ has $p$-complete frequencies, that is, $\mathcal E_p(\Omega)$ is dense in $H^p(\Omega)$ for all $p\in[1,\infty)$.
\end{theorem}

The result looks rather surprising and somewhat unexpected, because, on one hand, condition (2) says essentially that $\Omega$ is a Carath\'eodory domain in the strip, that is to say, the composition of $h$ with the conformal mapping of the strip onto $\D$ is a  weak-star sequential generator of $H^\infty$ in the sense of Sarason. However, there exist weak-star generators of $H^\infty$ such that their images are not Carath\'eodory domains. On the other hand, one might expect that discontinuities of $h$ represent a serious constraint to completeness, but, apart from those which are forbidden, $h$ can have ``serious'' discontinuities (like a Cantor set which is not a Cantor comb of discontinuities) and still enjoy completeness. Finally, the presence of boundary fixed points is very limited.

For  the case of parabolic semigroups of positive hyperbolic step, we have:

\begin{theorem}\label{Thm:main-para-pos}
Let $(\varphi_t)_{t\geq 0}$ be a parabolic semigroup of positive hyperbolic step  in $\D$ with Denjoy-Wolff point $\tau\in\partial\D$. Let $h$ be the Koenigs  function, $\Omega=h(\D)$ the Koenigs  domain  and let $\psi:I\to \R$ be the defining function of $\Omega$.   Then the following are equivalent:
\begin{enumerate}
\item $(\varphi_t)_{t\geq 0}$ has $\infty$-weak-star complete frequencies, that is, $\mathcal E_\infty(\Omega)$ is weak-star dense in $H^\infty(\Omega)$,
\item $\hbox{Int}(\overline{\Omega})=\Omega$ and $\C\setminus \overline{\Omega}$ is connected,
\item $\psi(y)=\tilde{\psi}(y)$ for every $y\in I$ and  $\liminf_{I\ni y\to y_0}\psi(y)>-\infty$ for every $y_0\in \overline{I}\setminus \overline{I_\infty}$, where $I_\infty$ is the (possibly empty) unbounded connected component of $\psi^{-1}(-\infty)$,
\item the semigroup $(\varphi_t)_{t\geq 0}$ satisfies  the following conditions:
\begin{itemize}
\item[(a)] if $A, B$ are non-exceptional maximal contact arcs for $(\varphi_t)_{t\geq 0}$ with final points $x_1(A)$ and $x_1(B)$ then   $x_1(A)\neq x_1(B)$,
\item[(b)] $\mathcal S(h)$ does not contain any Cantor comb,
\item[(c)] $(\varphi_t)_{t\geq 0}$ has no boundary fixed points other than $\tau$ and $\mathcal S(h)$ contains only bounded discontinuities except, at most,  one unbounded discontinuity $\xi\in\partial\D\setminus\{\tau\}$ which is the final point of a maximal contact arc starting at $\tau$.
\end{itemize}
\end{enumerate}
Moreover, if one---and hence any---of the above conditions (1)--(4) holds, then $(\varphi_t)_{t\geq 0}$ has $p$-complete frequencies, that is, $\mathcal E_p(\Omega)$ is dense in $H^p(\Omega)$ for all $p\in[1,\infty)$.
\end{theorem}

Finally, for parabolic semigroups of zero hyperbolic step, the situation is simpler:

\begin{theorem}\label{Thm:main-para-zero}
Let $(\varphi_t)_{t\geq 0}$ be a parabolic semigroup of zero hyperbolic step  in $\D$ with Denjoy-Wolff point $\tau\in\partial\D$. Let $h$ be the Koenigs  function, $\Omega=h(\D)$ the Koenigs  domain  and let $\psi:I\to \R$ be the defining function of $\Omega$.  If $\Omega$ is not contained in any half-plane, then $\mathcal E_\infty(\Omega)=\C$ and in particular $\mathcal E_\infty(\Omega)$ is not weak-star dense in $H^\infty(\Omega)$.

If there exists a half-plane $H$ such that $\Omega\subset H$, then the following are equivalent:
\begin{enumerate}
\item $(\varphi_t)_{t\geq 0}$ has $\infty$-weak-star complete frequencies, that is, $\mathcal E_\infty(\Omega)$ is weak-star dense in $H^\infty(\Omega)$,
\item $\hbox{Int}(\overline{\Omega})=\Omega$,
\item $\psi(y)=\tilde{\psi}(y)$ for every $y\in I$,
\item the semigroup $(\varphi_t)_{t\geq 0}$ satisfies  the following conditions:
\begin{itemize}
\item[(a)] if $A, B$ are non-exceptional maximal contact arcs for $(\varphi_t)_{t\geq 0}$ with final points $x_1(A)$ and $x_1(B)$ then   $x_1(A)\neq x_1(B)$,
\item[(b)] $\mathcal S(h)$ does not contain any Cantor comb.
\end{itemize}
\end{enumerate}
Moreover, if one---and hence any---of the above conditions (1)--(4) holds, then $(\varphi_t)_{t\geq 0}$ has $p$-complete frequencies, that is, $\mathcal E_p(\Omega)$ is dense in $H^p(\Omega)$ for all $p\in[1,\infty)$.
\end{theorem}

The proof of the previous theorems, which is based on the materials developed in the next sections, goes as follows. By a functional analysis argument, the $\infty$-weak-star completeness of frequencies implies that the same frequencies are also $p$-complete. Thus, we just consider $\infty$-frequencies. We show that for some basic domain like the half-plane and the strip, frequencies are $\infty$-weak-star complete. This is done using suitable Laplace transforms and approximation of rational functions. Then we show that a Koenigs domain $\Omega$ contained in a strip or in a half-plane has $\infty$-weak-star complete frequencies if and only if $H^\infty$ of the strip or the half-plane is weak-star dense in $H^\infty(\Omega)$. We study then such a condition and, using some results of Sarason and properties of starlike at infinity domains, we come up with the characterization (2) in the previous theorems. The rest of the arguments are then devoted to
 state the topological condition on $\Omega$ in terms of dynamical properties of the semigroup, and this brings to the study of discontinuities.

In the final Section~\ref{Sec:p-freq}, we turn our attention to $p$-completeness of frequencies. We do not know whether the  equivalent conditions (2)-(3)-(4) in the previous theorems are also necessary for $p$-completeness of frequencies. This seems to be an interesting problem, related to polynomial approximation in $H^p$ spaces of the so-called ``crescents''.

However, we can prove that certain conditions are necessary for $p$-completeness of frequencies (for instance, we show that (4).(a) is a necessary condition, and also the fact that $p$-frequencies are not contained in an interval  is a necessary condition).

We also provide some positive result for parabolic semigroups of zero hyperbolic step which are not contained in any half-plane (and for which then there are no $\infty$-frequencies, so that the $\infty$-weak-star completeness is not an issue). In such a case, we define {\sl logarithmic starlike at infinity domains}, which, roughly speaking, are logarithmic perturbations of a half-plane. If such a domain has  $p$-frequencies (see \eqref{Eq:full-freq}),
then these frequencies are $p$-complete (see Theorem~ \ref{thm-log-domains}). Using such logarithmic domains as a ``pivot'', we prove:

\begin{theorem}\label{Thm:main-density-in-log}
Let $(\varphi_t)_{t\geq 0}$ be a parabolic semigroup of zero hyperbolic step in $\D$ with Denjoy-Wolff point $\tau\in\partial\D$. Let $h$ be the Koenigs  function, $\Omega=h(\D)$ the Koenigs  domain  and let $\psi:I\to \R$ be the defining function of $\Omega$.  Suppose $\Omega\subseteq D$, where $D$ is the rotation of  logarithmic starlike at infinity domain with $p$-frequencies (that is, satisfying \eqref{Eq:full-freq}). Then conditions (2)-(3) and (4) of Theorem~\ref{Thm:main-para-pos} are equivalent and if one---and hence any---of them holds, then  $(\varphi_t)_{t\geq 0}$ has $p$-complete frequencies, that is, $\mathcal E_p(\Omega)$ is dense in $H^p(\Omega)$ for all
$p\in[1,\infty)$.
 \end{theorem}

It is worth  mentioning that there is some relation between the subject of our paper and an extensive study of the possibility and the uniqueness of the representation of a function $f(z)$, holomorphic in a domain $\Om$, as a series $\sum c_k \exp(\lam_k z)$. See, for instance, the book by Leont'ev  \cite{Leont-book} or its account in \cite{Leont-1988} and Krasichkov-Ternovskii \cite{Kras-Tern-1972}. In the setting of these studies, the domain $\Om$ is convex and the sequence of exponents $\{\lam_k\}$ is fixed. Typically it is a zero sequence of an entire function with some known growth properties. Explicit formulas for coefficients $c_k$ are given. In most of these works, $\Om$ is assumed to be bounded, but in the book \cite{Leont-book}, the case when $\Om$ is unbounded and its boundary contains two half-lines and the case when $\Om=\C$  are also considered (in the setting of the locally uniform convergence). As it is proved in \cite[Sections 7 and 8]{Kras-Tern-1972},  the so-called spectral synthesis for invariant subspaces of the space of holomorphic functions in $\Om$ holds for all unbounded convex domains and does not hold in general for bounded convex domains. Relations with convolution equations for holomorphic functions and differential equations of infinite order can also be found in those works.

 The plan of the paper is the following. In Section~\ref{Sec:Preliminaries} we recall some definitions and prove some preliminary facts about semigroups, semigroups of composition operators and Hardy spaces. In Section~\ref{Sec:freq}, we define frequencies and prove some basic properties of them. In Section~\ref{Sec:approx-basic}, we prove the completeness results for  strips, half-planes and logarithmic starlike at infinity domains. In Section~\ref{Sec:density-weak}, we prove that $\infty$-weak-star completeness of frequencies is equivalent to the topological conditions on $\Omega$ studying the holomorphic hulls of starlike at infinity domains. In Section~\ref{Sec:char1}, we initiate our study of discontinuities of Koenigs functions, introducing the Cantor combs, and characterize  which Koenigs domains $\Omega$ have the property that $\hbox{Int}(\overline{\Omega})=\Omega$. In Section~\ref{Sec:bum}, we turn our attention to boundary fixed points and show that there are natural one-to-one correspondences among boundary regular and super-repulsive fixed points, unbounded discontinuities and certain subsets of the domain of definition of the defining function of the Koenigs' domain. These are the key results to prove the dynamical characterization of completeness of frequencies. In the next Section~\ref{Sec:char2} we  characterize those Koenigs  domains $\Omega$ such that $\C\setminus\overline{\Omega}$ has at most two connected components.  Section~\ref{Sec:proofs} contains the proofs of our main results. Finally, in Section~\ref{Sec:p-freq} we focus on $p$-frequencies.

\medskip

\centerline{\bf Acknowledgements}

\medskip

Part of this work has been done while the first named author was ``Profesor Distinguido'' at the research institute \emph{Instituto de Ciencias Matem\'aticas (ICMAT),
	 Spain}. He wants sincerely to thank ICMAT for the possibility of spending time there and work in a perfect atmosphere.
	 
	 The authors thank the referees for several helpful comments which improve the original manuscript.

\section{Preliminaries}\label{Sec:Preliminaries}

Let $\D:=\{z\in \C: |z|<1\}$ denote the unit disc of the complex plane $\C$. The Riemann sphere will be denoted by $\C_\infty$. A {\sl half-plane} is the open set $\{w\in\C: \Re \langle w, v\rangle<c\}$, where $v\in\C\setminus\{0\}$ and $c\in \R$. The right half-plane is $\C_+ :=\{z\in \C: \Re z>0\}$. A {\sl strip} is the open set $\{w\in\C: a<|\Im w|<b\}$ for some $a, b\in \R$, $a<b$.

If $D\subset \C$ is a domain (open and connected set), its Euclidean closure is denoted by $\overline{D}$ and its boundary by $\partial D$. Likewise, $\overline{D}^\infty$ and $\partial_\infty D$ denote, respectively, the  closure of $D$ and the boundary of $D$ in the Riemann sphere. For any set $D\subset \C_\infty$ we denote by $\hbox{Int}(D)$ its interior.

\subsection{Semigroups in the unit disc, Koenigs  functions and Koenigs  domains}
In this subsection we collect the relevant results regarding semigroups of holomorphic functions. For details we refer the reader to the monographs \cite{BCDbook, Sho}.

A continuous semigroup $(\varphi_t)_{t\geq 0}$ of holomorphic self-maps of $\D$, or, for short, a {\sl semigroup in $\D$}, is a continuous semigroup homomorphism from $[0,+\infty)$ endowed with the Euclidean topology to $\hbox{Hol}(\D,\D)$, the space of holomorphic self-maps of $\D$ endowed with the topology of uniform convergence on compacta. Hence, for every $t\geq 0$,
\begin{enumerate}
\item $\varphi_t:\D\to\D$ is holomorphic,
\item $\varphi_0={\sf id}_\D$,
\item $\varphi_t\circ \varphi_s=\varphi_{t+s}$ for all $s,t\geq 0$, and
\item $\varphi_t$ converges uniformly on compacta of $\D$ to $\varphi_s$ when $t\to s$, $s, t\geq 0$.
\end{enumerate}

A semigroup $(\varphi_t)_{t\geq 0}$ is {\sl elliptic} if there exists a point $z_0\in\D$ such that $\varphi_t(z_0)=z_0$ for some $t>0$. It is {\sl non-elliptic} otherwise.

If $(\varphi_t)_{t\geq 0}$ is non-elliptic, then there exists a unique point $\tau\in\partial\D$ such that $(\varphi_t)_{t\geq 0}$ converges uniformly on compacta to the constant map $z\mapsto\tau$ when $t\to +\infty$. Such a point is called the {\sl Denjoy-Wolff point} of $(\varphi_t)_{t\geq 0}$. Moreover, for all $t\geq 0$, $\varphi_t'(z)$ has non-tangential limit at $\tau$ given by $e^{t\alpha}$ for some $\alpha\leq 0$. If $\alpha<0$, the semigroup is called {\sl hyperbolic} while it is {\sl parabolic} if $\alpha=0$.

A parabolic semigroup $(\varphi_t)_{t\geq 0}$ is {\sl of positive hyperbolic step} if for some---and hence any---$z\in \D$ it holds
$$\lim_{t\to+\infty}k_\D(\varphi_{t}(z), \varphi_{t+1}(z))>0$$
(here $k_\D$ is the hyperbolic distance in $\D$), otherwise it is {\sl of zero parabolic step}.

For every semigroup $(\varphi_t)_{t\geq 0}$ there exists \emph{the infinitesimal generator of $(\varphi_t)_{t\geq 0}$}, namely, the unique holomorphic function $G:\D \to \C$ such that
$$
\frac{\partial \varphi_t(z)}{\partial t}=G(\varphi_t (z)), \qquad t\geq 0,\; z\in \D.
$$

We  recall the following theorem (see, {\sl e.g.}, \cite[Thm. 9.3.5]{BCDbook}):

\begin{theorem}\label{Thm:models}
Let $(\varphi_t)_{t\geq 0}$ be a non-elliptic semigroup in $\D$.  There exists $h:\D\to\C$ univalent such that
\begin{enumerate}
\item $h(\varphi_t(z))=h(z)+t$ for all $t\geq 0$,
\item $\bigcup_{t\geq 0} (h(\D)-t)=U$,
\end{enumerate}
where $U=\C$ or $U=\{z\in\C: \Im z>0\}$, or $U=\{z\in\C: \Im z<0\}$ or $U=\{z\in\C: a<\Im z<b\}$ for some $a<b$.
\end{theorem}

Moreover, \cite[Prop. 9.3.10]{BCDbook}, if $g:\D\to \C$ is any univalent map such that $g(\D)=h(\D)$ and satisfies (1) and (2) then there exists $a\in\C$ such that $g(z)=h(z)+a$ (and, actually, in case $U\neq\C$, $a\in \R$).

The semigroup is  parabolic of zero hyperbolic step  if and only if $U=\C$, it is parabolic of positive hyperbolic step if and only if $U=\{z\in\C: \Im z>0\}$ or $U=\{z\in\C: \Im z<0\}$ and it is  hyperbolic if $U$ is a strip.

We call $h$ the {\sl Koenigs function} of $(\varphi_t)_{t\geq 0}$ and $\Omega:=h(\D)$ the {\sl Koenigs domain} of $(\varphi_t)_{t\geq 0}$.

By definition, the Koenigs  domain $\Omega$ is {\sl starlike at infinity}\footnote{In this paper when we say that a domain is starlike at infinity, we always mean with respect to the real translation $z\mapsto z+t$, $t\geq 0$.}
{\sl i.e.},
\[
\Omega+t\subseteq \Omega, \quad \hbox{ for all } t\geq 0.
\]
There is a simple way to define starlike at infinity  domains. Recall that, if $I\subseteq (-\infty,+\infty)$ is an open interval,  a function
$u:I \to [-\infty,+\infty)$ is {\sl upper semicontinuous} if for every $t_0\in I$
\[
\limsup_{t\to t_0} u(t)\leq u(t_0).
\]
Note that we allow the function $u$ to take value $-\infty$.

\begin{proposition}\label{Prop:definingfnc}
Let $I\subseteq (-\infty, +\infty)$ be an open interval and let $\psi: I\to [-\infty, +\infty)$ be an upper semicontinuous function. Let
\[
\Omega_\psi:=\{z \in \C: \Re z>\psi(\Im z), \Im z\in I\}.
\]
Then $\Omega_\psi$ is a domain starlike at infinity. Conversely, if $\Omega$ is a domain starlike at infinity, then there exist an open interval $I\subseteq (-\infty, +\infty)$ and an upper semicontinuous function $\psi: I\to [-\infty, +\infty)$ such that $\Omega=\Omega_\psi$. Moreover, $\Omega_\psi=\C$ if and only if $I=\R$ and $\psi\equiv -\infty$.
\end{proposition}
\begin{proof}
Since $\psi$ is upper semicontinuous, then $\Omega_\psi$ is a domain and clearly it is starlike at infinity.

Conversely, let $\Omega$ be a domain starlike at infinity.
Then the projection $I$ of $\Om$ onto the $y$-axis is an
open interval
(finite or infinite).
For $y\in I$, put
\[
\psi(y):=\inf \{x: x+iy\in \Omega\}\in [-\infty, +\infty).
\]
Since $\Omega$ is starlike at infinity, it follows easily that $\Omega=\Omega_\psi$. For any $y_0\in I$,
it is easy to see that the point $(\limsup_{y\to y_0}\psi(y), y_0)$
belongs to the closure of $\C\setminus \Om$, whenever
$$\limsup_{y\to y_0}\psi(y)\ne-\infty.$$
It follows that $\limsup_{y\to y_0}\psi(y)\le y_0,$
so that $\psi$ is upper semicontinuous.
\end{proof}

\begin{definition}
If $\Omega$ is the Koenigs  domain of a non-elliptic semigroup in $\D$, the upper semicontinuous function $\psi$ such that $\Omega=\Omega_\psi$ is called {\sl the defining function} of $\Omega$.
\end{definition}

\subsection{Hardy spaces and weak-star topology}

If $D\subsetneq \C$ is a simply connected domain,  for $p\in[1,\infty)$  we let
$H^p(D)$ denote the set of holomorphic functions $f$ in $D$ such that $|f|^p$ has a harmonic majorant in $D$. Given $z_0\in D$, we let $\|f\|_p:=\inf (u(z_0))^{1/p}$, where the infimum is taken among all harmonic majorants $u$ of $|f|^p$. This is indeed a norm that makes $H^p(D)$ into a Banach space.  If $g: D\to G$ is a biholomorphism, then the map $\Phi_g: H^p(D)\to H^p(G)$ defined as $f\mapsto f \circ g^{-1}$ induces a continuous isomorphism.  If $D=\D$, the Poisson representation formula implies that $H^p(\D)$ is the usual Hardy space $H^p$ of holomorphic functions with bounded $p$-mean. See, {\sl e.g.}, \cite{Du, Ru}, for details.

Also, we denote by $H^\infty(D)$  the set of all bounded holomorphic functions, which is a Banach space when endowed with the norm $\|f\|_\infty:=\sup\{|f(z)|, z\in D\}$.

Given $p\in [1,\infty]$, a classical result due to Fatou states that every function $f\in H^p(\D)$ has radial limit $f^\ast$ almost everywhere on $\partial\D$. The map $H^p(\D)\ni f\mapsto f^\ast\in L^p(\partial \D)$  induces an  isomorphism on a subspace  $H^p(\partial \D)$ of $L^p(\partial \D)$ (with respect to the normalized Lebesgue measure)---see, {\sl e.g.}, \cite[Chapter~7]{Du}.

For $f,g$ functions we let, whenever it makes sense,
\[
\langle f, g\rangle:=\frac{1}{2\pi}\int_0^{2\pi}f(e^{i\theta})g(e^{i\theta})d\theta.
\]
Then $\langle\cdot, \cdot\rangle: L^1(\partial\D)\times L^\infty(\partial\D)\to\C$ defines an isometrical isomorphism
\[
\iota_\infty: L^\infty(\partial\D)\to (L^1(\partial\D))^\ast,
\]
 the dual space of $L^1(\partial\D)$.
Thus, one can endow $H^\infty(\D)$ with the topology induced by  the weak-star topology of $(L^1(\partial\D))^\ast$ on $\iota_\infty(H^\infty(\partial\D))$. Such topology is the {\sl weak-star topology} on $H^\infty(\D)$.

For $p\in(1,+\infty)$, let $q\in(1,+\infty)$ be such that $\frac{1}{p}+\frac{1}{q}=1$. The dual pairing
$$\langle \cdot, \cdot \rangle: L^p(\partial\D)\times L^q(\partial \D)\to \C$$
defines an isometrical isomorphism $\iota_p: L^p(\partial\D)\to (L^q(\partial\D))^\ast$. Hence, we can define the weak-star topology of $H^p(\D)$ as the one induced on $\iota_p(H^p(\partial\D))$ by the weak-star topology of $(L^q(\partial\D))^\ast$. The dual pairing $\langle \cdot, \cdot\rangle$ allows one also to define an isomorphism (not isometrical) between $H^p(\partial\D)$ and $(H^q(\partial\D))^\ast$. In particular,  the space $H^p(\D)$ is reflexive, and the weak-star topology in $H^p(\D)$ coincides with the weak topology in $H^p(\D)$.

Using a Riemann map, one can then define the weak-star topology on $H^\infty(D)$ and $H^p(D)$, $p>1$, for every simply connected domain  $D\subsetneq \C$.

It is well-known (see \cite{Sa}) that  a sequence $\{f_n\}$ converges to $f$ in the weak-star topology on 
$H^\infty(D)$ if and only if it converges pointwise  in $D$ to $f$ and the norms of $f_n$ in $H^\infty(D)$
are uniformly bounded.

For the sake of completeness (and partly  because of the lack of direct references), we prove some facts about Hardy spaces that we would need later.

\begin{lemma}\label{Lem:approx-weak-approx-p}
Let $D\subset \C$ be a simply connected domain and let $\mathcal G\subset H^\infty(D)$ be a linear set. If $\mathcal G$ is weak-star dense in $H^\infty(D)$ then $\mathcal G$ is dense in $H^p(D)$ for all $p\in[1,\infty)$.
\end{lemma}

\begin{proof}
Using a Riemann map from $\D$ onto $D$, we can assume $D=\D$ and $\mathcal G\subset H^\infty(\D)$ weak-star dense in $H^\infty(\D)$, and prove that $\mathcal G$ is dense in $H^p(\D)$ for $p\in[1,\infty)$.

Fix $p$, $1<p<\infty$, and let $q$ be such that $\frac{1}{p}+\frac{1}{q}=1$.
Since  the inclusion $(L^1(\partial \D))^\ast\subset (L^q(\partial \D))^\ast$ is weak-star continuous, it follows that the inclusion $H^\infty(\D)\subset H^p(\D)$  is  weak-star continuous as well.

Since polynomials are dense in $H^p(\D)$, it follows that  $H^\infty(\D)$ is dense in $H^p(\D)$ (with respect to the strong topology in $H^p(\D)$), thus, it is dense also in the weak topology of $H^p(\D)$. But $H^p(\D)$ is reflexive, hence $H^\infty(\D)$ is dense in $H^p(\D)$ in the weak-star topology of $H^p(\D)$.

Since $\mathcal G$ is dense in $H^\infty(\D)$ in the weak-star topology, it follows from the previous observations that $\mathcal G$ is  dense in $H^p(\D)$ in the weak-star topology of $H^p(\D)$, thus, $\mathcal G$ is  dense in $H^p(\D)$ in the weak topology of $H^p(\D)$. By Mazur's Theorem (see, {\sl e.g.}, \cite[Corollary 3, Chapter 2]{Diestel}), $\mathcal G$ is dense in $H^p(\D)$ (with respect to the strong topology in $H^p(\D)$).

Therefore, we proved that $\mathcal G$ is dense in $H^p(\D)$ for every $p>1$.

Since $H^p(\D)$ embeds densely and continuously in $H^1(\D)$, it follows immediately that $\mathcal G$ is also dense in $H^1(\D)$.
\end{proof}

\begin{lemma}\label{Lem:approx-Hardy}
Let $G\subset\D$ be a simply connected domain such that $\hbox{Int}(\overline{G})=G$ and $\C\setminus\overline{G}$ is connected. Then $H^\infty(\D)$ is weak-star dense in $H^\infty(G)$ and $H^p(\D)$ is  dense in $H^p(G)$ for all $p\in[1,\infty)$.
\end{lemma}
\begin{proof} Let $\mathcal P$ be the set of all holomorphic polynomials. The conditions that $\hbox{Int}(\overline{G})=G$ and $\C\setminus\overline{G}$ is connected imply (see, {\sl e.g.}, \cite[Prop. 1]{Do}) that $G$ is a Carath\'eodory domain. Hence, \cite[Proposition~2]{Sa} implies that holomorphic polynomials are weak-star dense in $H^\infty(G)$. Since $\mathcal P\subset H^\infty(\D)$, we have that $H^\infty(\D)$ is weak-star dense in $H^\infty(G)$.

The last statement  follows from Lemma~\ref{Lem:approx-weak-approx-p}.
\end{proof}

\subsection{Infinitesimal generators of composition operators semigroups}

Associated to any semigroup  $(\varphi_t)_{t\geq 0}$ in $\D$, there is a  family of composition operators $\{C_{\varphi_t} \}_{t\geq 0}$, defined
on the space of analytic functions on $\mathbb{D}$ by
$$
C_{\varphi_t} f= f\circ \varphi_t.
$$
Obviously, $\{C_{\varphi_t} \}_{t\geq 0}$ has the semigroup property:
\begin{enumerate}
\item[(1)] $C_{\varphi_0}=I$ (the identity operator);
\item[(2)]$C_{\varphi_t} C_{\varphi_s}=C_{\varphi_{t+s}}$ for all $t,s\ge 0$.
\end{enumerate}

Recall that an operator  semigroup $\{T_t\}_{t\geq 0}$ acting on a Banach space $X$ is called \emph{strongly continuous} or \emph{$C_0$-semigroup} if it satisfies
 $$
\lim_{t\to 0^+} T_t f=f
$$
for $f\in X $. A remarkable result of  Berkson and Porta \cite{Ber-Por} states that any semigroup in $\D$ always induces a strongly continuous semigroup on the classical Hardy spaces $H^p(\D)$, $1\le p<\infty$.

Given a $C_0$-semigroup $\{T_t\}_{t\geq0}$ on a Banach space $X $, recall that its generator is the  closed and
densely defined linear operator $\Gamma$ given by
\[
\Gamma f=\lim_{t\to 0^+} \frac {T_t f-f} t
\]
with domain $\cD \Gamma=\{f\in X : \lim_{t\to 0^+} \frac {T_t f-f} t \enspace \text{exists}\}$. It is well-known that the semigroup is
determined uniquely by its generator (see, {\sl e.g.}, \cite{Co-Ma, R-S}).

In the particular instance of $\{C_{\varphi_t} \}_{t\geq 0}$, if $G$ is the infinitesimal generator of $(\varphi_t)_{t\geq 0}$,  the generator $\Gamma$ of the operator semigroup $\{C_{\varphi_t} \}_{t\geq 0}$ in $H^p(\D)$, $p\in[1,\infty)$ is given by
$$
\Gamma f= G' f
$$
for any $f\in \cD \Gamma=\{f\in H^p(\D):\, G' f\in H^p (\D)\}$.

In \cite{Siskakis 85} Siskakis showed that a complex number $\lambda$ is an eigenvalue of $\Gamma$ if and only if $e^{\lambda h} \in H^p(\D)$, where $h$ is the Koenigs function of the semigroup. At this regard, observe that if $(\varphi_ t)_{t\geq 0}$ is a non-elliptic semigroup in $\D$, $h$ its Koenigs  function  and $\Omega:=h(\D)$  the Koenigs domain, for every $\lambda \in \C$,
\[
e^{\lambda h(\varphi_t(z))}=e^{\lambda t} e^{\lambda h(z)}, \qquad (z\in \D),
\]
that is, $e^{\lambda h}$, is a \emph{semi-model} of $(\varphi_t)_{t\geq 0}$.
	By conformal invariance of $H^p$ spaces, we get that
	$\lambda$ is an eigenvalue of $\Gamma$ if and only if $e^{\lambda z} \in H^p(\Omega)$; this is the criterion which we will mostly use.

Recently, Betsakos \cite{Be} has studied the point spectrum of $\Gamma$ in terms of the geometry of the Koenigs domain of the semigroup. As we mentioned in the introduction, our aim is to study the completeness of eigenfunctions of $\Gamma$.

\section{Frequencies of semigroups in  Hardy spaces}\label{Sec:freq}

\begin{definition}
Let $\Omega$ be the Koenigs domain of a semigroup $(\varphi_t)_{t\geq 0}$. For $p\in [1,+\infty]$, we define the set of {\sl $p$-frequencies} of $(\varphi_t)_{t\geq 0}$ to be
\[
\Lambda_p(\Omega):=\{\lambda\in \C: e^{\lambda z}\in H^p(\Omega)\}.
\]
To each $\Lambda_p(\Omega)$ we associate the {\sl linear space of $p$-frequencies} of $(\varphi_t)_{t\geq 0}$ defined as
\[
\mathcal E_p(\Omega):=\hbox{span}\{e^{\lambda z}: \lambda\in \Lambda_p(\Omega)\}.
\]
If $\mathcal E_p(\Omega)$ is dense in $H^p(\Omega)$ for some $p\in[1,\infty)$, we say that $(\varphi_t)_{t\geq 0}$ has {\sl $p$-complete frequencies}.

If $\mathcal E_\infty(\Omega)$ is weak-star dense in $H^\infty(\Omega)$ we say that $(\varphi_t)_{t\geq 0}$ has {\sl $\infty$-weak-star complete frequencies}.
\end{definition}

Frequencies of semigroups have been studied by D. Betsakos~\cite{Be}. We state here his results in this form:

\begin{theorem}[Betsakos]\label{Betsakos}
Let $(\varphi_t)_{t\geq 0}$ be a non-elliptic semigroup of holomorphic self-maps of $\D$, with associated Koenigs domain $\Omega$.
\begin{enumerate}
\item If $(\varphi_t)_{t\geq 0}$ is hyperbolic and $p\in[1,\infty)$ then there exist $c_1\in (0,+\infty]$, $c_2\in (0,+\infty)$  such that
\[
\left\{\lambda\in\C :-\frac{c_1}{p}<\Re\lambda<\frac{c_2}{p}\right\}\subseteq \Lambda_p(\Omega)\subseteq \left\{\lambda\in\C :-\frac{c_1}{p}<\Re\lambda\leq\frac{c_2}{p}\right\}.
\]
In particular, $\{\lam\in \C: \Re \lam=0\}\subseteq \Lambda_p(\Omega)$.
\item If $(\varphi_t)_{t\geq 0}$ is parabolic and $p\in[1,\infty)$ then
\[
\Lambda_p(\Omega)\subseteq \left\{\lambda\in\C :\Re\lambda\leq 0\right\}.
\]
\end{enumerate}
\end{theorem}

We examine some easy properties of frequencies.

\begin{proposition}\label{Prop:easy}
Let $\Omega$ be the Koenigs domain of a semigroup $(\varphi_t)_{t\geq 0}$. Then:
\begin{enumerate}
\item $\Lambda_p(\Omega)$ is a convex set of $\C$ and $0\in\Lambda_p(\Omega)$ for all $p\in [1,+\infty]$.
\item If $p, q\in [1,+\infty)$ then $\Lambda_q(\Omega)=\frac{p}{q}\Lambda_p(\Omega)$.
\item If $1\leq p\leq q\leq +\infty$ then $\Lambda_q(\Omega)\subseteq \Lambda_p(\Omega)$ and hence $\mathcal E_q(\Omega)\subseteq \mathcal E_p(\Omega)$. In particular, $\Lambda_\infty(\Omega)\subseteq \bigcap_{1\le p<\infty}\Lambda_p(\Omega)$ and $\mathcal E_\infty(\Omega)\subseteq \bigcap_{1\le p<\infty}\mathcal E_p(\Omega)$.
\item Let $\lambda\in\C\setminus\{0\}$. Then $e^{\lambda z}\in H^\infty(\Omega)$ if and only if $\Omega\subseteq \{z\in\C: \Re (\lambda z)<M\}$ for some $M\in\R$. In particular, if $\lambda\in \Lambda_\infty(\Omega)$ then $t\lambda\in \Lambda_\infty(\Omega)$ for all $t\geq 0$.
\item If $D$ is the Koenigs domain of another semigroup and $\Omega\subset D$, then $\Lambda_p(D)\subseteq \Lambda_p(\Omega)$ and $\mathcal E_p(D)\subseteq \mathcal E_p(\Omega)$ for all $p\in [1,+\infty]$.
\end{enumerate}
\end{proposition}
\begin{proof}
(1)  It is clear that $0\in \Lambda_p(\Omega)$. Let $\lambda_1, \lambda_2\in \Lambda_p(\Omega)$.

Hence, there exist two harmonic functions $u_1, u_2$ in $\Omega$ such that $|e^{\lambda_j z}|^p\leq u_j(z)$ for $z\in \Omega$ and $j=1,2$.

Let $t\in (0,1)$. We have to prove that $e^{t\lambda_1 z+(1-t)\lambda_2 z}\in H^p(\Omega)$. Simple considerations show that
\[
|e^{t\lambda_1 z+(1-t)\lambda_2 z}|^p=e^{pt\Re (\lambda_1 z)}e^{p(1-t) \Re (\lambda_2 z)}\leq e^{p \Re (\lambda_1 z)}+e^{p\Re (\lambda_2  z)}\leq u_1(z)+u_2(z),
\]
and we are done.

(2) Assume $\lambda\in \Lambda_p(\Omega)$, that is, $e^{\lambda z}\in H^p(\Omega)$, and let $u$ be an harmonic majorant of $|e^{\lambda z}|^p$. Hence,
\[
|e^{\frac{p}{q}\lambda z}|^q=|e^{\lambda z}|^p\leq u(z),
\]
which proves that $\frac{p}{q}\lambda\in \Lambda_q(\Omega)$.

(3) It follows immediately from the inclusion $H^q(\Omega)\subset H^p(\Omega)$.

(4) It is obvious since $|e^{\lambda z}|=e^{\Re (\lambda z)}$.

(5) The restriction map from $D$ to $\Omega$ induces a continuous injective linear map from $H^p(D)$ to $H^p(\Omega)$, hence, if $e^{\lambda z}\in H^p(D)$ then $e^{\lambda z}\in H^p(\Omega)$, and the result follows.
\end{proof}

Frequencies are not conformal invariant. However, they well behave under affine maps, we state here the lemma we need:

\begin{lemma}\label{Lem:under-affine}
Let $T(z)=az+b$ for some $a\in\C\setminus\{0\}$ and $b\in\C$. Let $\Omega$ be a Koenigs domain of a semigroup, and let $D=T^{-1}(\Omega)$. Then for every $p\in [1,+\infty]$,
\[
\Lambda_p(D)=a\Lambda_p(\Omega).
\]
Moreover, $\overline{\mathcal E_p(\Omega)}^{H^p(\Omega)}=H^p(\Omega)$ ({\sl respectively}, $\mathcal E_\infty(\Omega)$ is weak-star dense in $H^\infty(\Omega)$) if and only if $\overline{\mathcal E_p(D)}^{H^p(D)}=H^p(D)$ (respectively, $\mathcal E_\infty(D)$ is weak-star dense in $H^\infty(D)$).
\end{lemma}
\begin{proof}
If $\lambda\in\Lambda_p(\Omega)$ then $e^{\lambda z}\in H^p(\Omega)$. Hence, $e^{\lambda T(z)}\in H^p(D)$, that is, $e^{\lambda b}e^{a\lambda z}\in H^p(D)$ and thus $e^{a\lambda z}\in H^p(D)$. Therefore, $a\lambda\in \Lambda_p(D)$. A similar argument proves that if $\mu\in \Lambda_p(D)$ then $\frac{\mu}{a}\in\Lambda_p(\Omega)$, and the first statement follows.

In order to prove the last assertion, let $1\le p<\infty$ (the argument for $p=\infty$ is similar) and note that $H^p(D)=H^p(\Omega)\circ T$. Hence, $\mathcal E_p(\Omega)$ is dense in $H^p(\Omega)$ if and only if $\mathcal E_p(\Omega)\circ T$ is dense in $H^p(D)$.
Since
\[
\mathcal E_p(\Omega)\circ T=\hbox{span}\{e^{a\lambda z}: \lambda\in \Lambda_p(\Omega)\}=\hbox{span}\{e^{\mu z}: \mu\in \Lambda_p(D)\}=\mathcal E_p(D),
\]
 we have the result.
\end{proof}

\section{ Approximation by exponentials in some starlike at infinity domains}\label{Sec:approx-basic}

The aim of this section is to produce basic examples of starlike at infinity domains where exponentials are dense in $H^p$ or weak-star dense in $H^\infty$.

\subsection{The half-plane and the strip} The aim of this section is to show that linear combinations of exponentials are weak-star dense in half-planes and strips.

We recall the following  result, which is contained  in the proof of  \cite[Thm. 6.1]{AC}.  A
different proof  based on  Laguerre  polynomials can be
found in \cite{GG-Montes} (see also \cite{GG-Montes2})

\begin{theorem}[Ahern-Clark]\label{Thm:Ahern-Clark}
We have
\[
\overline{\hbox{span}\{e^{\lambda \frac{1+z}{1-z}}: \lambda\leq 0\}}^{H^2(\D)}=H^2(\D),
\]
\end{theorem}

Using the Cayley transform,  the previous theorem implies that $\{e^{\lambda z}: \lambda\leq 0\}$ is complete in $H^2(\C_+ )$.

Let
\[
\mathbb S:=\{z\in \C : -\frac{\pi}{2}<\Im z <\frac{\pi}{2}\}.
\]

In \cite[Lemma 1]{C-GG}, the authors proved that $\overline{\hbox{span}\{e^{\lambda z}:  \lambda\in \Lambda_2(\mathbb S)\}}^{H^2(\mathbb S)}=H^2(\mathbb S)$ by means of the Paley-Wiener Theorem. Here we generalize these results:

\begin{theorem}\label{Thm:half-plane-group}
The linear space $\hbox{span}\{e^{\lambda z}: \lambda\leq 0\}$ is weak-star dense in $H^\infty(\C_+)$. Consequently, for every
$p\in[1,\infty)$, $\hbox{span}\{e^{\lambda z}: \lambda\leq 0\}$ is  dense in $H^p(\C_+)$.
\end{theorem}

\begin{theorem}\label{Thm:hyperbolic-group}
The linear space $\hbox{span}\{e^{i\lambda z}: \lambda\in \R\}$ is weak-star dense in $H^\infty(\mathbb S)$. Consequently, for every $p\in [1,\infty)$, $\hbox{span}\{e^{i\lambda z}: \lambda\in \R\}$ is  dense in $H^p(\mathbb S)$.
\end{theorem}

The proofs of the two results are similar, so we only give the proof of the second theorem.

\begin{proof}[Proof of Theorem~\ref{Thm:hyperbolic-group}]
Let  $t\in \R$ and observe that $|e^{itz}|\leq e^{|t|\frac{\pi}{2}}$. Hence, $e^{itz} \in H^{\infty}(\mathbb S)$.

Now, let $t\geq 0$. For $z\in \mathbb S$ and $\beta\in\C, \Re \beta>\frac{\pi}{2}$, let
\[
\Phi_\beta(z):=\int_0^{\infty} e^{itz} e^{-t\beta} dt.
\]
Since, for every $z\in \mathbb S$,
\begin{equation*}
\int_0^{\infty} |e^{itz} e^{-t\beta}| dt\leq \int_0^{\infty}  e^{-t(\Re \beta-\frac{\pi}{2})} dt,\end{equation*}
 the integral converges uniformly for $z\in \mathbb S$, and
\begin{equation*}
\Phi_\beta(z)=\int_0^\infty  e^{it z} e^{-t \beta} dt= \frac{1}{-iz+\beta}.
\end{equation*}
This also proves that $\Phi_\beta\in H^\infty(\mathbb S)$.

Now, for $R>0$ and $\beta\in\C, \Re \beta>\frac{\pi}{2}$, $z\in \mathbb S$, let
\[
\Phi_\beta^R(z):=\int_0^R e^{itz} e^{-t\beta} dt.
\]
By the same token as before, we see that $\{\Phi_\beta^R\}$ converges uniformly in $z\in \mathbb S$ to $\Phi_\beta$ as $R\to +\infty$, hence, $\{\Phi_\beta^R\}$ converges to $\Phi_\beta$ in $H^\infty(\mathbb S)$.

Now, fix $R>0$ and $\beta\in \C$, $\Re \beta>\frac{\pi}{2}$. Hence, $\Phi_\beta^R$ can be seen as the Laplace transform of the complex Borel measure compactly supported in  $[0,+\infty)$ defined by $\mu:=\chi_{[0,R]}e^{-t\beta} dt$.

For $n=1,2,\ldots$, we define $\mu_n$  as
\begin{equation}\label{eq:approx-measure}
\mu_n=\sum_{j=1}^{n}  \mu([\frac {(j-1)R} n,\frac {jR} n))\delta_{\frac{jR}{n}}.
\end{equation}
For each $n$, the measure $\mu_n$ has support on a finite number of points, and  $\{\mu_n\}$ converges to $\mu$ in the sense of measures, for $n\to\infty$.  Let
\[
P_n(z):=\int_0^{+\infty} e^{itz}d\mu_n(t).
\]
By definition, $P_n(z)\in\hbox{span}\{e^{itz}: t\geq 0\}$ for all $n$. Also, since the support of $\mu_n$ is contained in $[0,R]$ for all $n$, there exists $C>0$ such that $\sup_{z\in\mathbb S}|P_n(z)|\leq C$.

Thus, for every $z\in \mathbb S$ fixed,
\[
\lim_{n\to\infty}|P_n(z)-\Phi_\beta^R(z)|=\lim_{n\to\infty}\Big|\int_0^R e^{itz} \,d\mu_n(t)-e^{itz}e^{\beta t} \,dt\Big|=0,
\]
where the last limit is $0$  because the latter integral converges to $0$ as $n\to \infty$ for $z$ fixed since $\{\mu_n\}$ converges to $\mu$ in the sense of measures.  Thus, $\{P_n\}$ converges weak-star in $H^\infty(\mathbb S)$ to $\Phi_\beta^R$.

This proves  that rational functions with simple poles in
$$\{z\in \C: \Re z>\frac{\pi}{2}\}$$
can be approximated in the weak-star topology of $H^\infty(\mathbb S)$ by elements of $\hbox{span}\{e^{itz}, t\geq 0\}$.

A similar argument shows that rational functions with simple poles in
$$\{z\in \C: \Re z<-\frac{\pi}{2}\}$$
can be approximated in the weak-star topology of $H^\infty(\mathbb S)$ by elements of $\hbox{span}\{e^{itz}, t\leq 0\}$.

Therefore, rational functions with simple poles outside $\overline{\mathbb S}$ are contained in the closure of $\hbox{span}\{e^{itz}: t\in \R\}$  in weak-star topology of $H^\infty(\mathbb S)$. Each rational function in $H^\infty(\mathbb S)$ can be uniformly approximated in $\mathbb S$ by rational functions with simple poles. 
Since $\overline{\mathbb S}^\infty$ is a compact finitely connected proper subset of the Riemann sphere and 
 $\hbox{Int}(\overline{\mathbb S}^\infty)=\mathbb S$, rational functions in $H^\infty(\mathbb S)$ 
are weak-star  dense in $H^\infty(\mathbb S)$ by  \cite[Theorem 5.3]{Gam-book} (note that this theorem is formulated in \cite{Gam-book} only for compact subsets of the complex plane, but by applying M\"obius transformations, one can extend it to 
compact proper subsets of the Riemann sphere). 
Thus $\hbox{span}\{e^{i\lambda z}: \lambda\in \R\}$ is weak-star dense in $H^\infty(\mathbb S)$.

The case $p\in[1,\infty)$ follows at once from Lemma~\ref{Lem:approx-weak-approx-p}.
\end{proof}

\subsection{Logarithmic starlike at infinity domains} In this section we consider domains which might not be contained in any half-plane (so that they have no $\infty$-frequencies), for which however there is density in $H^p$ of linear combinations of exponentials.

\begin{definition}\label{Def:log-star-dom}
Let $\psi:\R\to \R$ be a a function such that
\begin{enumerate}
\item $\psi$ is differentiable and there exists $K>0$ such that $|\psi'|\le K$ on $\R$,
\item  there exist $R>0$ and $0<a<1$ such that $\psi(y)\geq -a \log |y|$ for $|y|\geq R$.
\end{enumerate}
We call the corresponding domain
\[
\Omega:=\{x+iy:x>\psi(y)\}
\]
a {\sl logarithmic starlike at infinity domain}.
\end{definition}

It is easy to see that any logarithmic starlike at infinity domain $\Omega$ is a Jordan domain in $\C_\infty$ and that, if $b\in \R$ then $\Omega+b$ is still a logarithmic starlike at infinity domain.

We show now that every  logarithmic starlike at infinity domain is biholomorphic to a bounded Jordan domain via a particular map. For $z\neq 0$, let
\[
\al(z):=\int_{-1}^0 (1+\lam)^2e^{\lam z}\, d\lam
= \frac {-2e^{-z}-2z+z^2+2}{z^3}.
\]
Also, for $b>0$ and $z\in \C$ let
\[
T_b(z)=z+b.
\]

\begin{lemma} \label{lem-alpha}
Let $\Omega$ be a logarithmic starlike at infinity domain. Then there exists $b>0$ such that the function $\alpha \circ T_b$ is a biholomorphism from $\Omega$ onto a bounded Jordan domain $G$, and, in particular it extends as a homeomorphism from $\overline{\Omega}^\infty$ to $\overline{G}$.
\end{lemma}

\begin{proof}
Let $\psi$ be the defining function of $\Omega$. Let $b>0$ and let $\Omega_b:=T_b(\Omega)$. The defining function of $\Omega_b$ is $\psi_b:=\psi+b$. Choose $b>0$ so large that  $\psi_b(y)\geq - a \log |y|$ whenever $|y|\ge 1$ and $\psi_b(y)\geq 1$ if $|y|\le 1$.
Then the closure of $\Om_b$ does not contain $0$.
We observe that
\begin{equation}
\label{z^a}
|e^{-z}|\le |z|^a \quad \text{for any } z=x+iy\in \overline{\Om_b}.
\end{equation}
Indeed,
if $|y|\le 1$, then $x\ge \psi_b(y)\ge 1$, and so
$|e^{-z}|=e^{-x}\le 1\le |z|^a$. If $|y|> 1$, then
$|e^{-z}|=e^{-x}\le e^{-\psi_b(y)}\le  |y|^a \le |z|^a$.

All these conclusions remain true if $b$ is replaced by a even larger constant.

Let $\eta(z)=1/\alpha(z)$. A direct differentiation gives that
\[\eta'(z) =
\frac{z^4+\kappa_1(z)}{z^4+\kappa_2(z)},
\]
where
\[
\kappa_1(z)=-4z^3+6z^2-(2z^3+6z^2)e^{-z}, \quad
z^4+\kappa_2(z)=(z^2-2z+2-2e^{-z})^2.
\]
Choose a small $\del>0$. By \eqref{z^a}, one can find a large $b>0$ such that
for all $z\in\partial \Om_b$, one has
\[
|\kappa_j(z)|\le \del |z|^4, \quad j=1,2.
\]
Hence
\begin{equation}
\label{der}
|\eta'(z)  - 1|
=
\bigg|\frac{1+\kappa_1(z)z^{-4}}{1+\kappa_2(z)z^{-4}}\,-1  \bigg|
\le \frac{2\del}{1-\del}.
\end{equation}
We parametrize $\partial\Omega_b$ by $z(t):=\psi(t)+b+it$, $t\in\R$.
Since $z'(t)=\psi'(t)+i$ and $|\psi'(t)|\le K$, we have
\[
\eps\le \arg z'(t)\le \pi-\eps,
\]
where $\eps=\arg(K+i)\in(0, \pi/2)$.
Take $\del$ so small that \eqref{der} implies
that
\[
|\arg \eta'(z)|< \eps/2
\]
for all $z\in \de \Om_b$. Then we get
\[
\eps/2\le \arg \frac d {dt} \eta(z(t))=
\arg \big(\eta'(z(t))\cdot z'(t)\big) \le \pi-\eps/2
\]
for all $t$. This implies that $\eta(z(t))$ parametrizes a curve without self-intersections. Moreover, $\Im \eta(z(t))>0$ for
$t>0$ and $\Im \eta(z(t))<0$ for $t<0$. Note that $|\eta(z(t))|\to \infty$ as $t\to\pm\infty$. Since $\Om_b$ is a Jordan domain in the Riemann sphere $\C_\infty$, it follows that $\al=1/\eta$ maps the boundary
of this domain onto a positively oriented curve $\de G$, containing $0$, where
$G$ is a Jordan domain.
By the argument principle, $\al$ is univalent on $\Om_b$ and maps $\Om_b$ onto $G$. By Carath\'eodory extension theorem, $\alpha$ extends as a homeomorphism on the boundaries.
\end{proof}

Let $M_{\text{fin}}(\R_+)$ denote the set of compactly supported (finite) complex Borel measures on $[0,+\infty)$ and let
\[
\LM=\{\LL \mu: \; \mu \in M_{\text{fin}}(\R_+)) \}
\]
to be the set of Laplace transforms of these measures. That is,
\[
\LL \mu(z):=\int_0^{+\infty}e^{-sz}d\mu(s).
\]
The set $\LM$ is an algebra, consisting of entire functions. The restrictions of functions in $\LM$ to $\C_+ $ belong to  $H^\infty(\C_+ )$.

Let $\mathcal {A}(\Omega)$ be algebra  of continuous functions on $\overline{\Omega}^\infty$
that are holomorphic on $\Omega$ and have a finite limit at infinity.

We prove the following density result:

\begin{theorem}\label{thm-log-domains}
Let $\Omega$ be a logarithmic starlike at infinity domain. Let $p_0\geq 1$ and suppose that
\begin{equation}\label{Eq:full-freq}
e^{\lam z}\in H^{p_0}(\Om), \quad \forall \lam\le 0.
\end{equation}
 Then, 	
 \begin{enumerate}
	\item $\LM\cap \mathcal{A}(\Omega)$ is dense in $\mathcal{A}(\Omega)$, 	
	\item $e^{\lam z}\in H^{p}(\Om)$ for all $\lambda\leq 0$, and
$\overline{\hbox{span}\{e^{\lambda z}: \lambda\leq 0\}}^{H^p(\Omega)}=H^p(\Om)$ for all $p\in[1,\infty)$.
	\end{enumerate}	
\end{theorem}
\begin{proof}
(1) First observe that it is enough to prove the result for $T_b(\Omega)$ for some $b>0$. Hence, by Lemma~\ref{lem-alpha}, we can assume that $\alpha:\Omega\to G$ is a biholomorphism.  Note that, by definition, $\alpha\in \LM\cap \mathcal{A}(\Omega)$.

Let $f\in \mathcal{A}(\Om)$ and let  $\tilde f= f\circ \al^{-1}$.
Then $\tilde f$ is continuous on $\overline{G}$ and hence, by Mergelyan's approximation theorem there exists a sequence of holomorphic polynomials $\{p_n\}$ such that
$\|\tilde f-p_n\|_{\infty, \overline{G}}\to 0$ as $n\to \infty$. Therefore, $\{p_n\circ \al\}$ approximates $f$ uniformly on $\overline{\Omega}^\infty$.

Since the Laplace transform of the convolution of two measures is the product of the Laplace transforms of these measures, it follows that   $p_n\circ \al\in \LM\cap \mathcal{A}(\Om)$ for all $n$,  and (1) follows.

 (2) By Proposition~\ref{Prop:easy} (2),  $e^{\lam z}\in H^{p}(\Om)$ for all $\lambda\leq 0$ and $p\in[1,\infty)$.

Fix $p\in[1,\infty)$ and let $\mu\in M_{\text{fin}}(\R_+)$, whose support is contained in
an interval $[0,R]$. For every $n=1, 2, \ldots$ we define $\mu_n$ as in \eqref{eq:approx-measure}.

Each $\mu_n$ is a measure  supported on a finite set. Moreover, by definition, $\LL \mu_n$ is a finite combination of exponentials of type $e^{tz}$, $t\leq 0$, hence, by \eqref{Eq:full-freq}, $\LL \mu_n\in H^p(\Omega)$.

Now, we claim that $\{\LL \mu_n\}$ converges to $\LL \mu$ in $H^p(\Omega)$.

To this aim, note that the map $[0,+\infty)\ni t\mapsto e^{tz}\in H^p(\Omega)$ is continuous, thus uniformly continuous on $[0,R]$. Hence, given $\epsilon>0$ there exists $\delta>0$ such that $\|e^{-\lambda z}-e^{-\tilde{\lambda}z}\|_{H^p(\Omega)}\leq \epsilon$ for every $\lambda,\tilde{\lambda}\in [0,R]$ such that $|\lambda-\tilde{\lambda}|<\delta$. Let  $I_j:=[\frac {(j-1)R} n,\frac {jR} n)$, with $n$ sufficiently large so that $\frac{1}{n}<\delta$. Hence,
\begin{equation*}
\begin{split}
\| \LL\mu-\LL \mu_n\|_{H^p(\Omega)}&\leq \sum_{j=1}^{n}\|\int_{I_j} e^{-\lambda z}d\mu(\lambda)-\mu(I_j)e^{-\frac{jR}{n}z}\|_{H^p(\Omega)}\\&=\sum_{j=1}^{n}\|\int_{I_j} (e^{-\lambda z}-e^{-\frac{jR}{n}z})d\mu(\lambda)\|_{H^p(\Omega)}\\&\leq \sum_{j=1}^{n}\int_{I_j} \|e^{-\lambda z}-e^{-\frac{jR}{n}z}\|_{H^p(\Omega)}|d\mu(\lambda)|\leq \epsilon |\mu([0,R])|.
 \end{split}
\end{equation*}

Therefore, each function in $\LM$ belongs to the closure of linear combination of exponentials
in $H^p(\Om)$.

Now, $\mathcal{A}(\Om)$ is dense in
$H^p(\Om)$ and by (1), each function in $\mathcal{A}(\Om)$ can be approximated uniformly on $\overline{\Omega}^\infty$ by functions in  $\LM\cap \mathcal{A}(\Om)$, which belong to the closure of the span of exponentials in
$H^p(\Om)$. Since the topology of $H^p(\Om)$ is weaker than the topology of ${\mathcal A}(\Om)$, assertion (2) follows.
\end{proof}

\begin{definition}
A logarithmic starlike at infinity domain $\Omega$ which satisfies \eqref{Eq:full-freq} is a logarithmic starlike at infinity domain {\sl with $p$-frequencies}.
\end{definition}

Further positive and negative results on the $p$-completeness of frequencies in logarithmic starlike domains at infinity
are given in Theorem~\ref{Thm:density-in-log} and
in Section~\ref{Sec:p-freq} (where, in particular,
Theorem~\ref{Thm:main-density-in-log} is proved).

\section{Weak-star density for Koenigs domains}\label{Sec:density-weak}

The aim of this section is to extend the results of the previous section to more general Koenigs domains.

To this aim, assume that $\Omega, D$ are Koenigs domains of two non-elliptic semigroups. Suppose $\Omega\subseteq D$. Then the inclusion $\iota:D\to\Omega$ induces the restriction  map
$$H^p(D)\to H^p(\Omega)$$ which is a continuous injective, for every $p\in [1,+\infty]$.

The following statement follows at once from the conformal invariance of Hardy space and recalling that $\mathcal E_p(D)\subseteq \mathcal E_p(\Omega)$ for $p\in[1,+\infty]$:

\begin{lemma}\label{Lem:dense-complete}
Let $p\in[1,\infty)$. Let $\Omega, D$ be Koenigs domains of two non-elliptic semigroups. Suppose $\Omega\subseteq D$.
\begin{itemize}
\item Assume $\mathcal E_\infty(D)$ is weak-star dense in $H^\infty(D)$. Then $H^\infty(D)$ is weak-star dense in $H^\infty(\Omega)$ if and only if $\mathcal E_\infty(D)$ is weak-star dense in $H^\infty(\Omega)$. If this is the case, $\mathcal E_\infty(\Omega)$ is also weak-star dense in $H^\infty(\Omega)$.
\item Assume $\mathcal E_p(D)$ is  dense in $H^p(D)$ for some $p\in[1,\infty)$. Then $H^p(D)$ is  dense in $H^p(\Omega)$ if and only if $\mathcal E_p(D)$ is dense in $H^p(\Omega)$. If this is the case, $\mathcal E_p(\Omega)$ is  also dense in $H^p(\Omega)$.
\end{itemize}
\end{lemma}

In our context, the domain $D$ of the previous lemma will be either a strip, or a half-plane or a  logarithmic starlike at infinity domain (for the $p\in[1,\infty)$ case). Thus, we are aimed to understand when $H^\infty(D)$ is weak-star dense in $H^\infty(\Omega)$ in case $D$ is either a strip or a half-plane.

In order to see this, we need the following proposition, which is a simple consequence of the work of Sarason~\cite{Sa, Sa1} on weak-star generators of $H^\infty(\D)$:

\begin{proposition}\label{Prop:Sarason}
Let $\Omega, D$ be simply connected domains in $\C$ and assume $\Omega\subset D$. The following are equivalent:
\begin{enumerate}
\item $H^\infty(D)$ is weak-star dense in $H^\infty(\Omega)$.
\item For every domain $E\subseteq D$ such that $\Omega\subsetneq E$ there exist $f\in H^\infty(E)$ and $z_0\in E\setminus\Om$ such that $|f(z_0)|>\sup_{z\in \Om}|f(z)|$.
\end{enumerate}
\end{proposition}
\begin{proof}
Using a Riemann map of $D$, we can assume that $D=\D$ and $\Omega\subset\D$.
Let $\mathcal P$ denote the set of all holomorphic polynomials. Let $\varphi:\D\to \Omega$ be a Riemann map. Then by \cite[Corollary~2]{Sa}, $\mathcal P\circ \varphi$ is weak-star dense in $H^{\infty}(\D)$ if and only if (2) holds. By the conformal invariance of Hardy spaces, it follows that (2) is equivalent to $\mathcal P$ being weak-star dense in $H^\infty(\Omega)$. Since $\mathcal P\subset H^\infty(\D)$, it follows that (1) is equivalent to (2).
\end{proof}

As a first corollary we have
\begin{corollary}\label{Cor:no-int-diff}
Let $\Omega, D$ be simply connected domains in $\C$ and assume $\Omega\subset D$. Assume that $\hbox{Int}(\overline{D})=D$ and that  $H^\infty(D)$ is weak-star dense in $H^\infty(\Omega)$. Then $\hbox{Int}(\overline{\Omega})=\Omega$.
\end{corollary}
\begin{proof}
Recall that $\hbox{Int}(\overline{\Omega})$ is the union of all  open sets $A\subset \overline{\Omega}$. Since $\overline{\Omega}\subseteq \overline{D}$, it follows that  $\hbox{Int}(\overline{\Omega})\subseteq \hbox{Int}(\overline{D})=D$.

Let $E:=\hbox{Int}(\overline{\Omega})\subseteq D$. Let $f\in H^\infty(E)$ and $z_0\in E$. Then there exists a sequence $\{z_n\}\subset \Omega$ converging to $z_0$. Thus,
\[
|f(z_0)|=\lim_{n\to\infty}|f(z_n)|\leq \sup_{z\in\Omega}|f(z)|.
\]
Since $H^\infty(D)$ is weak-star dense in $H^\infty(\Omega)$, we get  $\hbox{Int}(\overline{\Omega})=\Omega$ by Proposition~\ref{Prop:Sarason}.
\end{proof}

As a matter of notation, we say that a Jordan arc $\Gamma$ with end points $a,b$, $a\neq b$ is a {\sl crosscut} in $\D$ if $\Gamma\setminus\{a,b\}\subset \D$ and $a,b\in\partial\D$. We stress that, by definition, a crosscut has {\sl different} end points.
In what follows, for $y\in\R$, we set
\[
L_y:=\{z\in \C: \Im z=y\}.
\]

We need the following lemma which is a simple consequence of \cite[Proposition~4]{Sa}.

\begin{lemma}\label{Lem:top-cara-Sarason}
Let $G\subset \D$ be a simply connected domain. Suppose that every crosscut of $\D$ intersects $G$. Then for every $f\in H^\infty(\D)$ and every $z\in \D$
\[
|f(z)|\leq \sup_{\zeta\in G}|f(\zeta)|.
\]
\end{lemma}

Now we are in good shape to characterize Koenigs domains contained either in a half-plane or a strip for which linear combinations of exponentials are weak-star dense. We start with domains contained in a half-plane.

\begin{theorem}\label{Thm:density-in-half-plane}
Let $\Omega$ be the Koenigs domain of a non-elliptic semigroup in $\D$. Suppose $\Omega\subseteq H$, where $H$ is a half-plane. The following are equivalent:
\begin{enumerate}
\item $\mathcal E_\infty(H)$ is weak-star dense in $H^\infty(\Omega)$.
\item $\hbox{Int}(\overline{\Omega})=\Omega$ and $\C\setminus\overline\Omega$ is connected.
\end{enumerate}
\end{theorem}
\begin{proof}
We first show that (1) implies (2). By Lemma~\ref{Lem:dense-complete} and  Corollary~\ref{Cor:no-int-diff}, $\hbox{Int}(\overline{\Omega})=\Omega$.

In order to show that $\C\setminus\overline\Omega$ is connected, we first note that by
Lemma~\ref{Lem:dense-complete}  and Proposition~\ref{Prop:Sarason}, (1) is equivalent to

\smallskip
($\star$) For every domain $E\subseteq H$ such that $\Omega\subsetneq E$ there exist $f\in H^\infty(E)$ and $z_0\in E$ such that $|f(z_0)|>\sup_{z\in \Omega}|f(z)|$.
\smallskip

Suppose by contradiction that $\C\setminus\overline{\Omega}$ is not connected.

Note that $\C\setminus\overline{\Omega}$ has a connected component, $U_0$ such that $\C\setminus\overline{H}\subset U_0$.   Let $E:=H\setminus \overline{U_0}$. Since we are assuming that $\C\setminus\overline{\Omega}$ is not connected, it follows that $\Omega\subsetneq E\subset H$.

Now, let $f\in H^\infty(E)$. Let $z_0\in E$. If $z_0\in\overline{\Omega}$ then $|f(z_0)|\leq \sup_{z\in \Om} |f(z)|$.

If $z_0\in E\setminus\overline{\Omega}$, then  $z_0$ belongs to a connected component $U_1$ of $\C\setminus \overline{\Omega}$ different from $U_0$. Since $\Omega$ is starlike at infinity, $z_0-t\not\in \overline{\Omega}$ for all $t\geq 0$. It follows that $z_0-t\in U_1$ for all $t\geq 0$. Hence
$z_0-t\notin U_0$ for all $t\geq 0$.

Since $\C\setminus \overline{H} \subset U_0$, this implies that the boundary of $H$ has to be  parallel to the real axis.
Up to a translation and reflection we can assume that $H=\{w\in\C:\Im w>0\}$.

Thus, $z_0:=x+iy$ for some $y>0$. We claim that there exists $y_0\in [0,y)$ such that, setting $Q:= \{w\in \C: \Im w>y_0\}$, we have
\begin{equation}\label{Eq:L-Q-Omega}
L_{y_0}\subset \overline{\Omega}\cap \overline{Q}.
\end{equation}
Indeed, let $t\geq 0$. Note that if $x-t+ir\not\in\overline{\Omega}$ for all $r\in [0, y]$, then $[-1,y]\ni r\mapsto x-t+ir$ would connect $U_1$ with $U_0$ in $\C\setminus\overline \Omega$, a contradiction. Thus, for every $t\geq 0$ there exists $r(t)\in [0,y]$ such that $x-t+ir\not\in\overline{\Omega}$ for all $r\in (r(t), y]$ and $x-t+ir(t)\in\overline{\Omega}$. Since $\Omega$ is starlike at infinity, then $a-t+ir(t)\in\overline{\Omega}$ for all $a\geq x$. Hence, $[0,+\infty)\ni t\mapsto r(t)$ is non-increasing. Let $y_0:=\lim_{t\to+\infty}r(t)$.  Since $a-t+ir(t)\in\overline{\Omega}$ for all $a\geq x$ and $t\geq 0$, it follows that  $L_{y_0}\subset\overline{\Omega}\cap \overline{Q}$, and ~\eqref{Eq:L-Q-Omega} holds.

It follows from ~\eqref{Eq:L-Q-Omega} that $L_{y_0}\cap U_0=\emptyset$.
Since $\C\setminus \overline{H}\subset U_0$ and $U_0$ is connected,
it follows that $\overline{U_0}\subset \C\setminus Q$. Hence $Q\subset E$.
Note also that $Q\cap \Omega$ is a  domain starlike at infinity, in particular, it is simply connected.

Recall that $f\in H^\infty(E)$, thus $f|_Q\in H^\infty(Q)$. We claim
\begin{equation}\label{Eq:first-sup-iQ}
 \sup_{ \zeta\in Q}|f(\zeta)|= \sup_{\zeta\in\Omega\cap Q}|f(\zeta)|.
\end{equation}
In order to prove \eqref{Eq:first-sup-iQ}, we denote by $T$ the Cayley transform which maps $Q$ onto $\D$ and let $G:=T(Q\cap\Omega)$. We show that every crosscut of $\D$ intersects $G$. By Lemma~\ref{Lem:top-cara-Sarason}, this implies that  $|f(T^{-1}(\zeta))|\leq \sup_{\xi\in G}|f(T^{-1}(\xi))|$ for all $\zeta\in \D$ and thus \eqref{Eq:first-sup-iQ} follows.

Assume by contradiction that there is a crosscut $\tilde{\Gamma}$ in $\D$ which does not intersect $G$. Let $\tilde{a},\tilde{b}\in\partial\D$, $\tilde a\neq \tilde b$ be the end points of $\tilde{\Gamma}$. Then $\tilde{\Gamma}\setminus\{\tilde{a},\tilde{b}\}\subset \D\setminus G$. Hence, $\Gamma:=T^{-1}(\tilde{\Gamma})$ is a Jordan arc in $\C_\infty$ with end points $a:=T^{-1}(\tilde{a})$ and $b:=T^{-1}(\tilde{b})$ such that $a, b\in \partial_\infty Q$, $a\neq b$ and $\Gamma\setminus\{a,b\}\subset Q\setminus\Omega$. In particular, either $a$ or $b$ is different from $\infty$. We assume $a\in L_{y_0}$. Let $u\in \Gamma\setminus\{a,b\}$ and let $\Gamma'$ be the Jordan arc given by the part of $\Gamma$ connecting $a$ and $u$. Since $\Omega$ is starlike at infinity, $\cup_{t\geq 0}(\Gamma'-t)\cap\Omega=\emptyset$. This implies easily that $L_{y_0}\not\subset \overline{\Omega}\cap \overline{Q}$, contradicting \eqref{Eq:L-Q-Omega}. Thus, \eqref{Eq:first-sup-iQ} holds.

Since $z_0\in Q$ by construction,  \eqref{Eq:first-sup-iQ} implies
\[
|f(z_0)|\leq  \sup_{ \zeta\in Q}|f(\zeta)|= \sup_{\zeta\in\Omega\cap Q}|f(\zeta)|\leq
\sup_{\zeta\in\Omega}|f(\zeta)|.
\]
Summing up, we proved that for every $z_0\in E$ and $f\in H^\infty(E)$ we have $|f(z_0)|\leq \sup_{\zeta\in\Omega}|f(\zeta)|$, contradicting ($\star$). Hence $\C\setminus\overline{\Omega}$ is connected. Thus (2) holds.

\medskip

Now, we show that (2) implies (1). By Lemma~\ref{Lem:dense-complete}
and Theorem~\ref{Thm:half-plane-group}, it is enough to show that $H^\infty(H)$ is weak-star dense in $H^\infty(\Omega)$. To this aim, let $T:H\to \D$ be a Riemann map and let $G:=T(\Omega)$. Hence, it is enough to prove that $H^\infty(\D)$ is weak-star dense in $H^\infty(G)$. Since $T$ is a Moebius transformation, hence a homeomorphism of $\C_\infty$, and by hypothesis (2), it follows easily that $\hbox{Int}(\overline{G})=G$ and $\C\setminus \overline{G}$ is connected. Thus the result follows from Lemma~\ref{Lem:approx-Hardy}.
\end{proof}

\begin{remark}
The previous theorem shows that Theorem~\ref{Thm:hyperbolic-group} cannot be improved, in the sense that $\hbox{span}\{e^{it}: t\leq 0\}$ (or $\hbox{span}\{e^{it}: t\geq 0\}$) is not weak-star dense in $H^\infty(\mathbb S)$.
\end{remark}

Now we consider the case when $\Omega$ is contained in a strip.
Notice that the set $\Re\Omega=\{\Re z: z\in \Omega\}$ is not bounded from below if and only if $\overline{\Om}$ contains some horizontal line $L_y$. Indeed,
if $\Re\Omega$ is not bounded from below, then one can find a sequence of points $w_n\in \Om$
such that $\Re w_n\to -\infty$ and $\Im w_n$ tend to a real number $y$. Since $\Om$ is starlike at infinity, $w_n+\R_+\subseteq \Om$ for all $n$, and it follows easily that $L_y$ is contained in $\overline{\Om}$.

We need a preliminary lemma:

\begin{lemma}\label{Lem:meaning-connect-strip}
Let $\Omega$ be the Koenigs domain of a non-elliptic semigroup in $\D$. Suppose $\Omega\subseteq \mathbb S$, where $\mathbb S$ is a strip. Then
\begin{enumerate}
\item  $\C\setminus \overline{\Omega}$ has exactly one connected component
if and only if $\Re\Om$ is bounded from below.
\item
Now assume that $\Re\Om$ is not bounded from below.
Set
\[
y_0=\inf\{y: L_y\subset \overline \Om\},
\quad
y_1=\sup\{y: L_y\subset \overline \Om\}.
\]
Then $\C\setminus\overline{\Om}$ has exactly two connected components if
$\{z\in \C: y_0\le \Im z\le y_1\}\subset \overline \Om$, and has more than two
connected components if
$\{z\in \C: y_0\le \Im z\le y_1\}\not\subset \overline \Om$.
\end{enumerate}
\end{lemma}
\begin{proof}
(1) Since $\Omega$ is starlike at infinity, it follows that if $p,q \in \C\setminus\overline{\Omega}$, then $p-t, q-t\in \C\setminus \overline{\Omega}$ for all $t\geq 0$. Thus, if $\Omega\subset \{z\in\C:\Re z>R\}$,  then $p-t, q-t\in \{z\in\C:\Re z<R\}$ for $t$ large enough, and hence $p, q$ can be joined
in $\C\setminus \overline{\Omega}$
by a Jordan arc, showing that $\C\setminus \overline{\Omega}$ has one connected component.

Conversely, if
$\overline{\Om}$ contains a line $L_y$, then this line separates
$\C\setminus \overline{\Omega}$ into two open nonempty sets, so that
$\C\setminus \overline{\Omega}$ has at least two connected components.

(2) Suppose that $\{z\in \C: y_0\le \Im z\le y_1\}\not\subset \overline \Om$.
Notice that
	the lines $L_{y_0}$, $L_{y_1}$ are contained in $\overline{\Om}$.
	Hence $y_0\neq y_1$.

The set $\C\setminus (L_{y_0}\cup L_{y_1})$
 has three connected components. Clearly, each of these components contains a non-empty part of $\C\setminus \overline{\Omega}$. Therefore
$\C\setminus \overline{\Omega}$ has at least three connected components.

Now consider the remaining case when
$\{z\in \C: y_0\le \Im z\le y_1\}\subset \overline \Om$.
Then $\C\setminus\overline{\Om}$ is a disjoint union
of open sets $A_-:=\{z: \Im z<y_0\}\setminus \overline{\Om}$ and
$A_+:=\{z: \Im z>y_1\}\setminus \overline{\Om}$.
We assert that $A_-$ and $A_+$ are the two connected components of
$\C\setminus\overline{\Om}$. We need only to check that these two sets are connected.
Notice that for any sequence of points $\{w_n\}\subseteq \overline{\Om}$
such that $\lim_{n\to\infty}\Re w_n=-\infty$ and the limit $y:=\lim_{n\to\infty} \Im w_n$ exists, one has
$y_0\le y\le y_1$.

Let $a<b$ be real numbers such that
\[
\mathbb S=\{z\in \C: a<\Im z<b\}.
\]
Fix any point $z_0$ with $\Im z_0< a$, then $z_0\in A_-$.
Take any $z_1\in A_-$ with $a\le \Im z_0 <y_0$.
We assert that there exists $t<\min(\Re z_0, \Re z_1)$ such that the interval
$[t+i\Im z_0, t+i\Im z_1]$ is contained in $\C\setminus \overline{\Om}$.
Indeed, if it were not the case, then
one could find a sequence of points $t_n+is_n\in \overline{\Om}$ such
that $t_n\to-\infty$ and $\Im z_0 \le s_n\le \Im z_1$ for all $n$.
As in the argument preceding the formulation of Lemma, this would imply
that there is a line $L_y\subset\overline{\Om}$ with
$\Im z_0 \le y\le \Im z_1<y_0$. This contradicts the definition of $y_0$.

If $[t+i\Im z_0, t+i\Im z_1]$ is contained in $\C\setminus \overline{\Om}$,
then the union of three segments $[z_0, t+i\Im z_0]$,
$[t+i\Im z_0, t+i\Im z_1]$, $[t+i\Im z_1, z_1]$ connects $z_0$ with
$z_1$ inside $\C\setminus \overline{\Om}$. It follows that $A_-$ is connected.

A similar argument shows that $A_+$ is connected.
\end{proof}

\begin{theorem}\label{Thm:density-in-strip}
Let $\Omega$ be the Koenigs domain of a non-elliptic semigroup in $\D$. Suppose $\Omega\subseteq \mathbb S$, where $\mathbb S$ is a strip. The following are equivalent:
\begin{enumerate}
\item $\mathcal E_\infty(\mathbb S)$ is weak-star dense in $H^\infty(\Omega)$.
\item $\hbox{Int}(\overline{\Omega})=\Omega$ and $\C\setminus\overline{\Omega}$ has at most two connected components.
\end{enumerate}
\end{theorem}

\begin{proof}
	Let $a<b$ be real numbers such that $\mathbb S=\{w\in\C: a<\Im w<b\}$.
	
	Suppose (1) holds. By Lemma~\ref{Lem:dense-complete} and by Corollary~\ref{Cor:no-int-diff}, $\hbox{Int}(\overline{\Omega})=\Omega$.
	
	Also, by Lemma~\ref{Lem:dense-complete} and Proposition~\ref{Prop:Sarason}, (1) is equivalent to
	
	\smallskip
	($\star\star$) For every domain $E\subseteq \mathbb S$ such that $\Omega\subsetneq E$ there exist $f\in H^\infty(E)$ and $z_0\in E$ such that $|f(z_0)|>\sup_{z\in \Omega}|f(z)|$.
	\smallskip
	
	Assume now by contradiction that $\C\setminus\overline{\Omega}$ has more than two connected components.
	Define $y_0, y_1$ as in Lemma~\ref{Lem:meaning-connect-strip}, and
    let
	\[
	Q:=\{w\in \C: y_0 < \Im w < y_1\}.
	\]
	By Lemma~\ref{Lem:meaning-connect-strip},
	 $L_{y_0}, L_{y_1}\subset\overline{\Omega}$ and
	 $Q\not\subset \overline{\Omega}$.
	Let
	\[
	E:=\Omega\cup Q.	
	\]
	It is easy to see that, by construction, $E$ is a domain starlike at infinity, $\Omega\subsetneq E\subseteq \mathbb S$.
	
	Now let $f\in H^\infty(E)$,
	$M:=\sup_{\zeta\in\Omega}|f(\zeta)|$ and $z\in E$. If $z\in\overline{\Omega}\cap E$ then clearly $|f(z)|\leq M$. If $z\in E\setminus\overline{\Omega}$, then  by construction, $z\in Q$.
	
	Let  $g$ be a Riemann map from $Q$ onto $\D$ and let $G:=g(Q\cap\Omega)$. Note that $Q\cap \Omega$ is a (simply connected) starlike at infinity domain, so $G$ is simply connected. We claim that every crosscut of $\D$ intersects $G$. Once the claim is proved, by Lemma~\ref{Lem:top-cara-Sarason}, this implies that  $|f(g^{-1}(\zeta))|\leq \sup_{\xi\in \D}|f(g^{-1}(\xi))|=\sup_{\xi\in G}|f(g^{-1}(\xi))|$ for all $\zeta\in \D$ and thus
	\[
	|f(z)|\leq \sup_{\zeta\in Q}|f(\zeta)|=\sup_{\zeta\in Q\cap\Omega}|f(\zeta)|\leq M.
	\]
	Hence, $|f(z)|\leq \sup_{\zeta\in \Omega}|f(\zeta)|$ for every $z\in E$,  contradicting  ($\star\star$).
	
	So we have to prove that every crosscut of $\D$ intersects $G$. Assume by contradiction that there is a crosscut $\tilde{\Gamma}$ in $\D$ which does not intersect $G$. Let $\tilde{a},\tilde{b}\in\partial\D$, $\tilde a\neq \tilde b$ be the end points of $\tilde{\Gamma}$. Then $\tilde{\Gamma}\setminus\{\tilde{a},\tilde{b}\}\subset \D\setminus G$.
	
	Now,
	by applying translations and dilations, we can assume that
	$y_0=-\pi/2, y_1=\pi/2$.
	Then we can take $g(w)=\frac{e^w-1}{e^w+1}$. In particular, $\lim_{\Re w\to -\infty}g(w)=-1$, $\lim_{\Re w\to +\infty}g(w)=1$ and $g$ extends continuously and injectively  on $\partial  Q$, with values different from $\pm 1$.
	
	Hence, $\Gamma:=g^{-1}(\tilde{\Gamma})$ is a Jordan arc in $\C_\infty$ with end points $a:=g^{-1}(\tilde{a})$ and $b:=g^{-1}(\tilde{b})$ such that $a, b\in \partial_\infty Q$,  and $\Gamma\setminus\{a,b\}\subset Q\setminus\Omega$. Moreover, if either $a\in\partial Q$ or $b\in \partial Q$ then $a\neq b$.
	
	In case $a\in\partial Q$ or $b\in\partial Q$, we get a contradiction arguing as in Theorem~\ref{Thm:density-in-half-plane}.
	
	In case $a=b=\infty$, let $\gamma:\R\to\C$ be a continuous injective parametrization of $\Gamma\cap \C$. Hence $y_0<\Im \gamma(t)<y_1$ for all $t\in \R$ and $\lim_{t\to\pm\infty}\gamma(t)=\infty$ in $\C_\infty$. Now, the curve $\R\ni t\mapsto g(\gamma(t))$ parametrizes $\tilde{\Gamma}\cap \D$, and, by the previous discussion on $g$,  we can assume $\lim_{t\to+\infty}g(\gamma(t))=1$ and $\lim_{t\to-\infty}g(\gamma(t))=-1$. It follows  that $\lim_{t\to+\infty}\Re\gamma(t)=+\infty$ and $\lim_{t\to-\infty}\Re\gamma(t)=-\infty$. Then $\Gamma\cap\C$ is a Jordan arc outside $\Omega$ which divides $\C$ into two connected components, one, say $V_1$ containing $L_{y_1}$, and the other, say $V_0$, containing $L_{y_0}$. Then $\Omega$ is contained either in $V_0$ or in $V_1$. If $\Omega\subset V_0$ then clearly $L_{y_1}\not\subset\overline{\Omega}$, while, if $\Omega\subset V_1$ then  $L_{y_0}\not\subset\overline{\Omega}$. In both cases, we have a contradiction and the claim is proved.
	
	Assume (2) holds. Let $g:\mathbb S\to \D$ be a Riemann map. As before, we assume that $y_0=-\pi/2, y_1=\pi/2$ and $g(w)=\frac{e^w-1}{e^w+1}$.
	
	Let $G:=g(\Omega)$. We claim that  $\C\setminus \overline{G}$ is connected and $\hbox{Int}(\overline{G})=G$. As soon as the claim is proved, arguing as in the proof of ``(2) implies (1)'' of Theorem~\ref{Thm:density-in-half-plane}, one can show that (1) holds.
	
	Let us prove the claim. It is clear that $\hbox{Int}(\overline{G})=G$ because $g:\mathbb S\to\D$ is a homeomorphism and $\hbox{Int}(\overline{\Omega})=\Omega$.
	
	Next, note that $\C\setminus\overline{\Omega}=\C_\infty\setminus\overline{\Omega}^\infty$.
	If $\C\setminus \overline{\Omega}$---and hence $\C_\infty\setminus \overline{\Omega}^\infty$---has one connected component, then by Lemma~\ref{Lem:meaning-connect-strip}, $\Omega\subset H:=\{z\in\C:\Re z>R\}$, for some $R\in \R$. Since $H\cap \mathbb S$ is a Jordan domain in $\C_\infty$, it follows by the Jordan-Schoenflies theorem that $g|_{H\cap \mathbb S}$ extends  to a homeomorphism $\tilde{g}:\C_\infty\to\C_\infty$. Taking into account that $G=g(\Omega)=\tilde{g}(\Omega)$, we have that $\C\setminus\overline{G}$ is connected.
	
	If $\C\setminus \overline{\Omega}$ has two connected components, by Lemma~\ref{Lem:meaning-connect-strip}, it follows that if $p\in (\C\setminus \overline{\Omega})\cap \mathbb S$ then there exists  a Jordan arc $\Gamma$ in $\C$ with one end point $p$ and the other end  point $\xi\in\partial\mathbb S$ such that $\Gamma\setminus\{\xi\}\subset\mathbb S$ and
	$\Gamma\cap \overline{\Omega}=\emptyset$. Therefore, $g(\Gamma)$ is a Jordan arc in $\D\setminus \overline{G}$ which connects $g(p)$ to a point of $\partial \D$, from which it follows immediately that $\C\setminus \overline{G}$ is connected.
\end{proof}

The next corollary complements the previous theorems:

\begin{corollary}\label{Cor:complement}
Let $\Omega$ be the Koenigs domain of a non-elliptic semigroup in $\D$.
\begin{enumerate}
\item Suppose there is a strip $\mathbb S$ such that
$\Omega\subseteq \mathbb S$.  Then the following are equivalent:
\begin{itemize}
\item[(1.a)] $\mathcal E_\infty(\mathbb S)$ is weak-star dense in $H^\infty(\Omega)$.
\item[(1.b)] $\mathcal E_\infty(\Omega)$ is weak-star dense in $H^\infty(\Omega)$.
\end{itemize}
\item Suppose there is a half-plane $H$ such that
$\Omega\subseteq H$ and $\Omega$ is not contained in any  strip.  Then the following are equivalent:
\begin{itemize}
\item[(2.a)] $\mathcal E_\infty(H)$ is weak-star dense in $H^\infty(\Omega)$.
\item[(2.b)] $\mathcal E_\infty(\Omega)$ is weak-star dense in $H^\infty(\Omega)$.
\end{itemize}
\end{enumerate}
\end{corollary}

\begin{proof} Since $\mathcal E_\infty(\mathbb S)\subseteq \mathcal E_\infty(\Omega)$, it is clear that (1.a) implies (1.b) ({\sl respectively}, (2.a) implies (2.b)), so we are left to show that (1.b) implies (1.a) ({\sl respect.}, (2.b) implies (2.a)).

Let $D=\mathbb S$ or $D=H$, according to hypothesis (1) or (2).

If $\Lambda_\infty(\Omega)=\Lambda_\infty(D)$, there is nothing to prove. Thus,  we assume

\smallskip
\noindent $\Lambda_\infty(D)\subsetneq \Lambda_\infty(\Omega)$ and $\mathcal E_\infty(\Omega)$ is  weak-star dense in $H^\infty(\Omega)$.
 \smallskip

Let $\mathcal H$ be the family of all half-planes which contain $\Omega$. Note that each $V\in \mathcal H$ is starlike at infinity and convex. Then let $Q$ be the intersection of all $V\in \mathcal H$. It is clear that $Q$ is a convex domain, starlike at infinity and $\Omega\subseteq Q\subseteq D$.

By Proposition~\ref{Prop:easy}.(4), $\Lambda_\infty(Q)=\Lambda_\infty(\Omega)$. Thus  $\mathcal E_\infty(Q)=\mathcal E_\infty(\Omega)$, and, since we are assuming that $\mathcal E_\infty(\Omega)$ is weak-star dense in $H^\infty(\Omega)$,
Lemma~\ref{Lem:dense-complete} implies that $H^\infty(Q)$ is weak-star dense in $H^\infty(\Omega)$.

Moreover, since $Q$ is convex, then $\C\setminus\overline Q$ is either  connected or has two connected components, and the latter possibility occurs if and only if $Q$ is a strip.

Therefore, if $\C\setminus\overline Q$ is connected, by either Theorem~\ref{Thm:density-in-strip} or Theorem~\ref{Thm:density-in-half-plane} (depending whether we are in case (1) or (2)), we have that $\mathcal E_\infty(D)$ is weak-star dense in $H^\infty(Q)$, and since the latter is weak-star dense in $H^\infty(\Omega)$, we have that (1.b) implies (1.a) ({\sl respect.}, (2.b) implies (2.a)).

If $\C\setminus\overline Q$ has two connected components, then $Q$ is a strip. Hence, we are in case (1), and arguing as before using Theorem~\ref{Thm:density-in-strip} we are done.
\end{proof}

In case $\Omega$ is not contained in a half-plane, then $\Omega$ has no $\infty$-frequencies, but it might have $p$-frequencies. For this case, we have a sufficient condition for completeness:

\begin{theorem}\label{Thm:density-in-log}
Let $\Omega$ be the Koenigs domain of a non-elliptic semigroup in $\D$. Suppose $\Omega\subseteq D$, where $D$ is the rotation of  logarithmic starlike at infinity domain with $p$-frequencies (that is, \eqref{Eq:full-freq} is satisfied). If $\hbox{Int}(\overline{\Omega})=\Omega$ and $\C\setminus \overline{\Omega}$ is connected, then $\mathcal E_p(D)$---and hence $\mathcal E_p(\Omega)$---is dense in $H^p(\Omega)$ for all $p\in[1,\infty)$.
\end{theorem}
\begin{proof}
Let $h:\D\to D$ be a Riemann map and let $G:=h^{-1}(\Omega)$. By Lemma~\ref{Lem:dense-complete} it is enough to prove that $H^p(\D)$ is dense in $H^p(G)$.

Since $D$ is  the rotation of  logarithmic starlike at infinity domain, then $D$ is a Jordan domain in $\C_\infty$. Therefore, by the Jordan-Schoenflies theorem, $h$ extends as a homeomorphism $\tilde{h}:\C_\infty\to \C_\infty$. It is then clear that  $\hbox{Int}(\overline{\Omega})=\Omega$  if and only if $\hbox{Int}(\overline{G})=G$ and $\C\setminus \overline{\Omega}$ is connected if and only if $\C\setminus \overline{G}$ is connected. Thus the result follows from Lemma~\ref{Lem:approx-Hardy}.
\end{proof}

\section{Characterization of Koenigs  domains $\Omega$ such that $\hbox{Int}(\overline{\Omega})=\Omega$}\label{Sec:char1}

The aim of this section is to give a characterization of Koenigs  domains of non-elliptic semigroups for which the interior of the closure in the Riemann sphere coincide with the domain. To this aim, we need some definitions and preliminary results.

\subsection{Prime ends and impressions}

Let $h$ be the Koenigs function of non-elliptic semigroup in $\D$ and let $\Omega=h(\D)$.

Let $\widehat{\Omega}$ denote the Carath\'eodory compactification of $\Omega$, endowed with the Carath\'eodory topology, and let $\partial_C\Omega$ be the set of prime ends of $\Omega$.  Moreover, let $\hat{h}:\widehat{\D}\to\widehat{\Omega}$ be the homeomorphism induced by $h$ and, for $\sigma\in\partial\D$, let $\underline{\sigma}$ be the prime end defined by $\sigma$---more precisely, the identity map ${\sf id}_\D:\D\to\D$ extends as a homeomorphism ${\Phi_\D}$ from $\overline{\D}$ onto $\widehat{\D}$, and we let $\underline{\sigma}:={\Phi_\D}(\sigma)$.  If $(C_n)$, $n\geq 0$, $n\in \N$, is a null-chain, every $C_n$ divides $\Omega$ in two simply connected domains.  The interior part $V_n$ of $C_n$, $n\geq 1$, is the connected component which does not contain $C_0$.  The impression of the prime end $[(C_n)]$ represented by $(C_n)$ is then
\[
I([(C_n)])=\bigcap_{n\geq 1}\overline{V_n}^\infty\subset \partial_\infty\Omega.
\]
See, {\sl e.g.} \cite[Chapter 4]{BCDbook} for details about Carath\'eodory's prime ends and Carath\'eodory's topology.

Recall (see, {\sl e.g.} \cite[Corollary 11.1.7]{BCDbook}) that, if $h$ is the Koenigs function of a semigroup, then for every $\sigma\in\partial \D$, there exists
\[
h^\ast(\sigma):=\angle\lim_{z\to \sigma}h(z).
\]
Therefore (see, {\sl e.g.}, \cite[Prop. 4.4.4, Thm. 4.4.9]{BCDbook}), if $\sigma\in\partial\D$, the {\sl principal part} of $\hat{h}(\underline{\sigma})$ (that is, the set of points which are accumulation points of any null chain representing $\hat{h}(\underline{\sigma})$) is $\{h^\ast(\sigma)\}$, while
\begin{equation}\label{Eq:impress-cluster}
I(\hat{h}(\underline{\sigma}))=\{p\in\partial_\infty\Omega: \exists \{z_n\}\subset\D, \lim_{n\to\infty}z_n=\sigma, \lim_{n\to\infty}h(z_n)=p\}.
\end{equation}
In particular, $h$ is continuous at $\sigma\in\partial\D$ if and only if the impression of $I(\hat{h}(\underline{\sigma}))$ is a point.

\subsection{Maximal contact arcs}\label{Subsec:maximal-arcs} We are going to recall briefly the theory of maximal contact arcs for semigroups of the unit disc (we refer the reader to \cite[Cap 14.2]{BCDbook} for details).

If $(\varphi_t)_{t\geq 0}$ is a semigroup in $\D$, then for every $\sigma\in\partial \D$ and for every $t\geq 0$, there exists the non-tangential limit
\[
\varphi^\ast_t(\sigma):=\angle\lim_{z\to\sigma}\varphi_t(z).
\]
 If for some $t_0>0$ (and hence any $0\leq t\leq t_0$) $\varphi^\ast_{t_0}(\sigma)\in\partial \D$, then the curve $[0,t_0)\ni t\mapsto \varphi^\ast_t(\sigma)$ is a continuous injective curve on $\partial \D$, with the natural orientation.

\begin{definition}
Let $(\varphi_t)_{t\geq 0}$ be a semigroup in $\D$.  A {\sl  contact arc} is an open arc $A\subset\partial\D$ such that for every $\sigma\in A$ there exists $t_\sigma>0$ such that $\varphi^\ast_{t_\sigma}(\sigma)\in\partial \D$.

A contact arc is {\sl maximal} if there exists no contact arc $B$ such that $A\subsetneq B$.
\end{definition}

It is clear that every contact arc is contained in a maximal contact arc. Let $A$ be a maximal contact arc  for  a semigroup $(\varphi_t)_{t\geq 0}$. As we remarked above, $A$ has a natural orientation given by flowing $A$ with $\varphi^\ast_t$ and thus we can define  the {\sl starting point} $x_0(A)$ of $A$ to be the end point of $A$ where the flow starts and the {\sl final point} $x_1(A)$ of $A$ to be the end point of $A$ where the flow ends.

A maximal contact arc $A$ is {\sl exceptional} if $x_1(A)$ is the Denjoy-Wolff point of the semigroup.

Let $(\varphi_t)_{t\geq 0}$ be a semigroup  with Koenigs function $h$ and $\Omega=h(\D)$. If $A$ is a maximal contact arc for $(\varphi_t)_{t\geq 0}$, with initial point $x_0(A)$ and final point $x_1(A)$ then for $j=0,1$ we have (see \cite[Proposition~14.2.13]{BCDbook})
\begin{equation}\label{Eq:where-go-int-fin-max-arc}
\lim_{A\ni\xi\to x_j(A)}h(\xi)=h^\ast(x_j(A)).
\end{equation}
Moreover, either $x_1(A)$ is the Denjoy-Wolff point of the semigroup, or $h^\ast(x_1(A))\in\C$ and $h^\ast(x_1(A))+t\in \Omega$ for all $t>0$ (see  \cite[Proposition~14.2.6]{BCDbook}).

For our aims, we also need  the following lemma (which follows easily from  \cite[Thm. 14.2.10]{BCDbook}):

\begin{lemma}\label{Lem:cara-contact}
Let $(\varphi_t)_{t\geq 0}$ be a non-elliptic semigroup with Denjoy-Wolff point $\tau\in\partial\D$,    Koenigs function $h$ and Koenigs domain $\Omega$. An open arc $A\subset \partial \D$ is a  contact arc if and only if there exist $y_0\in \R$, $-\infty\leq a<b\leq +\infty$ such that
\[
h^\ast(A)=\{w\in \C: \Im w=y_0, a<\Re w<b\}.
\]
If this is the case, $h$ extends holomorphically through $A$. Moreover,  there exist two maximal contact arcs $A, B$ with final point $\sigma=x_1(A)=x_1(B)\in\partial\D\setminus\{\tau\}$ if and only if $h^\ast(\sigma)\in \C$ and there exists  $r>0$ such that
\[
 D(h^\ast(\sigma),r)
\setminus \{z: \Re z\leq \Re h^\ast(\sigma), \Im z=\Im h^\ast(\sigma)\}= \Omega\cap D(h^\ast(\sigma),r).
\]
Also, $h$ extends holomorphically through $A\cup B$ and continuously through $\sigma$.
\end{lemma}

From this lemma it follows immediately:

\begin{corollary}\label{Cor:arco-en-def}
Let $(\varphi_t)_{t\geq 0}$ be a non-elliptic semigroup with Denjoy-Wolff point $\tau\in\partial\D$ and  Koenigs  domain $\Omega$ and let $\psi:I\to [-\infty,+\infty)$  be the upper semicontinuous function defining $\Omega$, that is, $\Omega=\{z\in\C: \Re z>\psi(\Im z), \Im z\in I\}$. Then there exist two maximal contact arcs $A, B$ with final point $\sigma=x_1(A)=x_1(B)\in\partial\D\setminus\{\tau\}$ if and only if there exist $q\in \R$, an open interval $J\subset I$ and $c_0\in J$ such that $\psi(c_0)>q$ and $\psi(c)\leq q$ for all $c\in J\setminus\{c_0\}$.
\end{corollary}

\subsection{Discontinuities of Koenigs  functions}\label{subsect:discont}

\begin{definition}
Let $h:\D\to \C$ be the Koenigs  function of a non-elliptic semigroup. The set of  points of continuity of $h$ is
\[
\mathcal C(h):=\{\sigma\in \partial \D: \lim_{z\to \sigma}h(z) \ \hbox{exists in $\C_\infty$}\}.
\]
The set of points of discontinuity is
\[
\mathcal S(h):=\partial \D\setminus \mathcal C(h).
\]
\end{definition}
By Collingwood maximality theorem (see, {\sl e.g.}, \cite[Prop. 2.23]{Pom}) and since the Koenigs  function $h$ has non-tangential limit everywhere on $\partial \D$, the set $\mathcal S(h)$ is a set of Baire first category, that is, it is the countable union of nowhere dense subsets of $\partial \D$. In particular $\hbox{Int}(\mathcal S(h))=\emptyset$.

Let $(\varphi_t)_{t\geq 0}$ be a semigroup in $\D$ with Denjoy-Wolff point $\tau\in\partial \D$. If $\sigma\in \partial \D\setminus\{\tau\}$ then $\sigma\in \mathcal S(h)$ if and only if  $\varphi_{t_0}$ does not extend continuously at $\sigma$ for some $t_0>0$. This is the case if and only if $\varphi_{t}$ does not extend continuously at $\sigma$ for all  $t>0$ (see, \cite[Proposition~11.3.2]{BCDbook}).

For $-\infty\leq a<b<+\infty$ and $y\in \R$ we let
\begin{equation}\label{Eq:L-con-subindex}
L_{a,b;y}:=\{w\in\C: a\leq \Re w\leq b, \Im w=y\}.
\end{equation}

\begin{lemma}\label{Lem:discont-Koenigs-prel}
	Let $h:\D\to \C$ be the Koenigs  function of a non-elliptic semigroup with Denjoy-Wolff point $\tau\in\partial\D$ and let $\psi:I\to [-\infty,\infty)$ be the defining function of $\Om=h(\D)$.
	If $\sigma\in \mathcal S(h)\setminus\{\tau\}$ then  $b:=\Re h^\ast(\sigma)\in \R$ and $y=\Im  h^\ast(\sigma)\in\overline{I}$ and there  exists $a\in [-\infty,b)$ such that $I(\hat{h}(\underline{\sigma}))=L_{a,b;y}$. Moreover, one of the two following two cases (not mutually disjoint) are possible:
	\begin{enumerate}
		\item $(y-\eps,y)\subset I$ for some $\eps>0$, $a=\liminf_{t\to y^-} \psi(t)$
		and $b=\limsup_{t\to y^-} \psi(t)$;
		
		\item $(y,y+\eps)\subset I$ for some $\eps>0$, $a=\liminf_{t\to y^+} \psi(t)$
		and $b=\limsup_{t\to y^+} \psi(t)$.
	\end{enumerate}
	
	Conversely, if there exist $y\in \overline{I}$, $b\in (-\infty, +\infty)$,  and $a\in  [-\infty,b)$ such that (1) or (2) (or both) hold, then there exists $\sigma\in \mathcal S(h)\setminus\{\tau\}$ such that $h^\ast(\sigma)=b+iy$.
\end{lemma}

 The previous lemma is a straightforward reinterpretation, using the defining function of the Koenigs domain, of the next lemma. The proof of the latter immediately ensues from the depiction of prime ends impressions in starlike at infinity domains (see \cite[Thm. 11.1.4]{BCDbook} and its proof):

\begin{lemma}\label{Lem:discont-Koenigs}
	Let $h:\D\to \C$ be the Koenigs  function of a non-elliptic semigroup with Denjoy-Wolff point $\tau\in\partial\D$.  Let $\sigma\in\partial\D\setminus\{\tau\}$. Let $p:=h^\ast(\sigma)$, and, if $p\in\C$, set $b:=\Re h^\ast(\sigma)$ and $y=\Im  h^\ast(\sigma)$.
	
	Then $\sigma\in \mathcal S(h)$ if and only if $p\in\C$ and there exists $a\in [-\infty,b)$ such that the impression  of $\hat{h}(\underline{\sigma})$ is $I(\hat{h}(\underline{\sigma}))=L_{a,b;y}$.
	
	If this is true,
	the two following cases (not mutually disjoint) are possible:
	\begin{enumerate}
		\item The following conditions hold:
		\begin{itemize}
			\item[(1.a)] there exists  a sequence $\{z^-_n\}\subset \C$ converging to $p$ such that
			$\{\Im z^-_n\}$ is  strictly increasing, $L_{-\infty,\Re z^-_n;\Im z^-_n}\subset \C\setminus\Omega$, and $z_n^-+t\in\Omega$ for all $t>0$,
			\item[(1.b)]  for any $x>b$ there exists $\delta>0$ such that $\{w\in \C: y-\delta <\Im w<y, \Re w>x\}\subset\Omega$,
			\item[(1.c)]     $\{w\in \C: y-\epsilon<\Im w<y, \Re w<\al\}\cap \Omega\neq\emptyset$ for any $\alpha>a$ and   $\epsilon>0$,
			\item[(1.d)]  if $a\neq-\infty$,
			for any $\beta<a$ there exists $\delta>0$ such that
			$\{w\in \C: y-\delta <\Im w<y, \Re w<\beta\}\cap\Omega=\emptyset$.
		\end{itemize}
		
		\item The following conditions hold:
		\begin{itemize}
			\item[(2.a)] there exists  a sequence $\{z^+_n\}\subset \C$ converging to $p$ such that
			$\{\Im z^+_n\}$ is  strictly decreasing, $L_{-\infty,\Re z^+_n;\Im z^+_n}\subset \C\setminus\Omega$, and $z_n^++t\in\Omega$ for all $t>0$,
			\item[(2.b)] for any $x>b$ there exists $\delta>0$ such that $\{w\in \C: y <\Im w<y+\delta, \Re w>x\}\subset\Omega$,
			\item[(2.c)]    $\{w\in \C: y<\Im w<y+\epsilon, \Re w<\al\}\cap \Omega\neq\emptyset$ for any $\alpha>a$ and   $\epsilon>0$,
			\item[(2.d)] if $a\neq-\infty$,
			for any $\beta<a$ there exists $\delta>0$ such that
			$\{w\in \C: y <\Im w<y+\delta, \Re w<\beta\}\cap\Omega=\emptyset$.
		\end{itemize}
	\end{enumerate}
	Conversely,  if there exist  $p\in\partial\Omega$, $-\infty\leq a<b<+\infty$ and $y\in \R$ which verify (1) or (2) (or both), then there exists $\sigma\in\mathcal S(h)\setminus\{\tau\}$ such that $h^\ast(\sigma)=p$ and $L_{a,b;y}\subseteq I(\hat{h}(\underline{\sigma}))$.
\end{lemma}

\begin{remark}
Although we are not going to use the following, for the sake of precision, the complete ``converse'' conclusion of Lemma~\ref{Lem:discont-Koenigs} is the following. If there exist  $p\in\partial\Omega$, $-\infty\leq a<b<+\infty$ and $y\in \R$ which verify (1) but not (2) (or (2) but not (1)) then  there exists $\sigma\in\mathcal S(h)\setminus\{\tau\}$ such that $h^\ast(\sigma)=p$ and $L_{a,b;y}=I(\hat{h}(\underline{\sigma}))$. On the other hand, if  there exist  $p\in\partial\Omega$, $-\infty<b<+\infty$, $y\in \R$ and $a_+,a_- \in [-\infty, b)$ which verify both (1) with $a=a_-$ and  (2) with  $a=a_+$, we have:
\begin{itemize}
\item if $p+t\in\Omega$ for all $t>0$ then there exists $\sigma\in\mathcal S(h)\setminus\{\tau\}$ such that $h^\ast(\sigma)=p$ and $L_{\min\{a_-, a_+\},b;y}= I(\hat{h}(\underline{\sigma}))$,
\item  if $p+t_0\not\in\Omega$ for some $t_0>0$ then there exist $\sigma_-, \sigma_+\in\mathcal S(h)\setminus\{\tau\}$, $\sigma_-\neq\sigma_+$ such that $h^\ast(\sigma_-)=h^\ast(\sigma_+)=p$ and $L_{a_-,b;y}= I(\hat{h}(\underline{\sigma_-}))$, $L_{a_+,b;y}= I(\hat{h}(\underline{\sigma_+}))$.
\end{itemize}
\end{remark}

\begin{definition}
	Let $h:\D\to \C$ be the Koenigs  function of a non-elliptic semigroup with Denjoy-Wolff point $\tau\in\partial\D$.  Let $\sigma\in S(h)\setminus\{\tau\}$.  We say that $\sigma\in\mathcal S^-(h)$ if
    $I(\hat{h}(\underline{\sigma}))=L_{a,b;y}$ for some  $a,b\in   [-\infty, \infty)$,  where
    $a<b$ and (1) of Lemma~\ref{Lem:discont-Koenigs-prel} holds.
    While, we say that $\sigma\in\mathcal S^+(h)$ if $I(\hat{h}(\underline{\sigma}))=L_{a,b;y}$ for some
	$a,b\in [-\infty, \infty)$, 	where
	$a<b$ and (2) of Lemma~\ref{Lem:discont-Koenigs-prel} holds.
\end{definition}

	Note that $\mathcal S(h)\setminus\{\tau\}=\mathcal S^+(h)\cup \mathcal S^-(h)$.  In general, $\mathcal S^-(h)\cap\mathcal S^+(h)\neq\emptyset$.

\begin{remark}
	Let $\sigma\in\partial\D\setminus\{\tau\}$.
	Notice that the condition (1) in Lemma~\ref{Lem:discont-Koenigs-prel}
	is equivalent to the condition (1) in   Lemma~\ref{Lem:discont-Koenigs}; the conditions (2)
	in these two lemmas are also mutually equivalent. Moreover,
	$\sigma\in\mathcal S^-(h)$ if and only if for some $a<b$, $a,b\in [-\infty,\infty)$,
	(1.a), (1.b) and (1.c) of Lemma~\ref{Lem:discont-Koenigs} hold. While,  $\sigma\in\mathcal S^+(h)$ if and only if
	for some $a<b$, (2.a), (2.b) and (2.c)  of Lemma~\ref{Lem:discont-Koenigs} hold. 	Note that here it is not necessary to require (1.d) or (2.d).
\end{remark}

\begin{proposition}\label{Lem:two-disc}
Let $h$ be the Koenigs  function of a non-elliptic semigroup  $(\varphi_t)_{t\geq 0}$ in $\D$ with Denjoy-Wolff point $\tau$ and let $\Omega=h(\D)$ be the Koenigs  domain. Let $\psi:I\to\R$ be the defining function of $\Omega$. The fibres of the map
	\[
	\mathcal S(h)\setminus\{\tau\}\ni \sigma\mapsto \Im h^\ast(\sigma)\in \overline{I}
	\]
	have at most two points.
Moreover  if $\sigma_1,\sigma_2\in \mathcal S(h)\setminus\{\tau\}$, $\sigma_1\neq\sigma_2$ are such that $\Im h^\ast(\sigma_1)=\Im h^\ast(\sigma_2)$ then $\Im h^\ast(\sigma_1)\in I$ and,
\begin{enumerate}
	\item if $h^\ast(\sigma_1)=h^\ast(\sigma_2)$, then
\begin{itemize}
	\item  $\sigma_1, \sigma_2 \in \mathcal  S^+(h)\cap \mathcal  S^-(h)$
	\item  there exists $t>0$ such that $h^\ast(\sigma_1)+t \in \partial\Omega$,
	\item $(\varphi_t)_{t\geq 0}$ has two  non-exceptional maximal contact arcs  starting, respectively, at  $\sigma_1$ and $\sigma_2$ and having the same final point (different from $\sigma_1,\sigma_2$).
\end{itemize}
	\item If $h^\ast(\sigma_1)\neq h^\ast(\sigma_2)$, then
$\sigma_1\in \mathcal  S^\pm (h)\setminus \mathcal  S^\mp (h)$, $\sigma_2\in \mathcal  S^\mp (h)\setminus \mathcal  S^\pm (h)$.
\end{enumerate}\end{proposition}

\begin{proof}
Let $\sigma_1\in \mathcal S(h)\setminus\{\tau\}$. We can assume that $\sigma_1\in \mathcal S^-(h)$ (the case $\sigma_1\in \mathcal S^+(h)$ is similar). Let $x_0+iy_0:= h^\ast(\sigma_1)\in \C$. Hence, by Lemma~\ref{Lem:discont-Koenigs-prel}, there  exists $a_0\in [-\infty,x_0)$ such that  $(y_0-\eps,y_0)\subset I$ for some $\eps>0$, $a_0=\liminf_{t\to y_0^-} \psi(t)$ and $x_0=\limsup_{t\to y_0^-} \psi(t)$.

Assume now $\sigma_2\in \mathcal S^-(h)$, $\sigma_2\neq \sigma_1$, and $\Im h^\ast(\sigma_2)=y_0$. It follows again by Lemma~\ref{Lem:discont-Koenigs-prel} that  $\Re h^\ast(\sigma_2)=\limsup_{t\to y_0^-} \psi(t)=x_0$. That is, $h^\ast(\sigma_1)=h^\ast(\sigma_2)$.

Consider the continuous injective curve $\gamma:(0,1)\to\D$ defined by $\gamma(t)=(\sigma_2-\sigma_1)t+\sigma_1$. The closure of the image of $\gamma$ is the segment $T$ joining $\sigma_1$ with $\sigma_2$. Hence, $\Gamma:=h(T)$ is a Jordan curve, $p:=h^\ast(\sigma_2)=h^\ast(\sigma_1)\in\Gamma$ and  $\Gamma\setminus\{p\}\subset \Omega$. Let $U$ be the bounded connected component of $\C\setminus\Gamma$ and let $V=h^{-1}(U)$. Hence, $\partial V\cap\partial \D$ is a closed arc $A$ (of positive measure) in $\partial \D$ joining $\sigma_1$ with $\sigma_2$. Since $h$ is proper, it follows that $h^\ast(\zeta)\in\partial\Omega$ for all $\zeta\in A$. Taking into account that $\Gamma$ is a continuous curve and $\Omega$ is starlike at infinity, it follows easily that
\[
h^\ast(A)\subset \{w\in \C: \Im w=y_0, \Re w\geq x_0\}.
\]
If it were $h^\ast(A)=\{p\}$, then $h$ is constant. In particular, this means that $U$ cannot be contained in either $\{w\in \C: \Im w>y_0\}$ or $\{w\in \C: \Im w<y_0\}$, and thus $y_0\in I$.

Therefore, taking again into account that $\Omega$ is starlike at infinity and $\Gamma$ is continuous, there exists $t_0>0$ such that $x_0+t_0+\delta+iy_0\in\Omega$ for all $\delta>0$,
\begin{equation}\label{Eq:creo-arco-doppio}
\{w\in \C: x_0\leq \Re w\leq x_0+t_0, \Im w=y_0\}=h^\ast(A),
\end{equation}
and for every $t>0$ there exists $\epsilon>0$ such that
\begin{equation}\label{Eq:creo-arco-doppio2}
\{w\in\C: \Re w>x_0+t, 0<|\Im w-y_0|<\epsilon\}\subset \Omega.
\end{equation}

Now, we can assume that $\Im h(\gamma(t))<y_0$ for $t$ close to $0$ and $\Im h(\gamma(t))>y_0$ for $t$ close to $1$. Let $(C_n)$, $n\geq 0$, be a null chain in $\Omega$ representing $\hat{h}(\underline{\sigma_2})$ and let $V_n$ be the interior of $C_n$, $n\geq 1$.

For $n$ sufficiently large, $V_n\subset \{w\in\C: \Re w< x_0+\frac{t_0}{2}\}$. Moreover, taking into account that $\{w\in \C: \Im w=y_0, \Re w\leq x_0+t_0\}\cap\Omega=\emptyset$,
for $n$ sufficiently large we have that $V_n$ is either contained in $\{w\in\C: \Im w>y_0\}$ or $\{w\in\C: \Im w<y_0\}$.

Since $\lim_{t\to 1^-}\gamma(t)=\sigma_2$, it follows that $h(\gamma(t))$ is eventually contained in $V_n$ for all $n\geq 1$ and the condition $\Im h(\gamma(t))>y_0$ for $t$ close to $1$ implies then $V_n\subset\{w\in\C: \Im w>y_0\}$ for $n$ sufficiently large.

Since by hypothesis $\sigma_2\in\mathcal S^-(h)\subset \mathcal S(h)$, we have that $\bigcap_n \overline{V_n}^\infty$ is not reduced to a single point, and thus (2.a), (2.b), (2.c) of Lemma~\ref{Lem:discont-Koenigs} are satisfied, hence, $\sigma_1,\sigma_2\in \mathcal S^+(h)$.

Now we claim:

\smallskip

$(\star)$ if $\eta:[0,1)\to \Omega$ is a continuous curve such that $\lim_{t\to1^-}\eta(t)=x_0+iy_0$ then either $\lim_{t\to1^-}h^{-1}(\eta(t))=\sigma_1$ or $\lim_{t\to1^-}h^{-1}(\eta(t))=\sigma_2$. As a consequence, if $\sigma\in\partial\D$ is such that $h^\ast(\sigma)=x_0+iy_0$ then either $\sigma=\sigma_1$ or $\sigma=\sigma_2$.

\smallskip

In order to prove the claim $(\star)$, we construct two (non equivalent) null-chains as follows. Let $n_0\in\N$ be such that $\frac{1}{n_0}<t_0$. For $n\geq 1$ let 
\[
\Theta^+_n:=\{z\in\C: |z-(x_0+iy_0)|=\frac{1}{n}, \Im z\geq y_0\},
\]
and
\[
 \Theta^-_n:=\{z\in\C: |z-(x_0+iy_0)|=\frac{1}{n}, \Im z\leq y_0\}.
\]
By \eqref{Eq:creo-arco-doppio} and \eqref{Eq:creo-arco-doppio2}, for all $n\geq n_0$, the set 
$\Theta^+_n\cap \Omega$ is a union of open arcs. One of them has $x_0+\frac{1}{n}+iy_0$ as an endpoint. We denote by 
$C^+_n$ the closure of such an open arc.  

Since $\sigma_2\in \mathcal S^+(h)$, by Lemma~\ref{Lem:discont-Koenigs}.(2), the final points of the arc $C^+_n$ are given by $x_0+\frac{1}{n}+iy_0$ and some $w^+_n\in\partial\Omega$ with $\Im w^+_n>y_0$  and $\lim_{n\to\infty} w_n^+=x_0+iy_0$. By construction, $\{C^+_n\}_{n\geq n_0}$ is a null-chain in $\Omega$. Let
\[
R^+_n:=\{w\in\C:  y_0<\Im w<\Im w^+_n, \Re w< \Re w^+_n\},
\]
and 
\[
W_n^+:=\{w\in\C: |w-(x_0+iy_0)|<\frac{1}{n}, \Im w>y_0, \Re w\geq \Re w_n^+\}.
\]
Since $\Omega$ is starlike at infinity, the interior part $V^+_n$ of $C^+_n$, $n>n_0$, is given by 
\begin{equation}\label{Eq:innerV_n+}
V^+_n=\Omega\cap(R^+_n\cup W^+_n). 
\end{equation}
Similarly, one can define a null-chain $\{C^-_n\}_{n\geq n_0}$ such that, for every $n\geq n_0$, $C^-_n$ is a closed arc in $\Theta^-_n$ with final points $x_0+\frac{1}{n}+iy_0$ and some $w^-_n\in\partial\Omega$ with $\Im w^-_n<y_0$ and $\lim_{n\to\infty} w_n^-=x_0+iy_0$. Setting 
\[
R^-_n:=\{w\in\C:  \Im w^-_n<\Im w<y_0, \Re w< \Re w^-_n\},
\]
and 
\[
W_n^-:=\{w\in\C: |w-(x_0+iy_0)|<\frac{1}{n}, \Im w<y_0, \Re w\geq \Re w_n^-\},
\]
the interior part $V^-_n$ of $C^-_n$ is given by $V^-_n=\Omega\cap(R^-_n\cup W^-_n)$, $n>n_0$.
Note that, in particular, $\{C^+_n\}_{n\geq n_0}$ and $\{C^-_n\}_{n\geq n_0}$ are not equivalent. Let $\zeta_1\in\partial\D$ be such that $\{C_n^-\}$ represents the prime end $\hat{h}(\underline{\zeta_1})$ and let $\zeta_2\in\partial\D$ be such that $\{C_n^+\}$ represents the prime end $\hat{h}(\underline{\zeta_2})$. 

Now, if $\eta:[0,1)\to \Omega$ is a continuous curve such that $\lim_{t\to1^-}\eta(t)=x_0+iy_0$, by \eqref{Eq:creo-arco-doppio} it follows that either $\Im \eta(t)>y_0$ or $\Im \eta(t)<y_0$ for $t$ close enough to $1$. Suppose that $\Im \eta(t)>y_0$ for every $t> t_0$, for some fixed $t_0\in[0,1)$. Fix $n>n_0$. Since  $\overline{R^+_n\cup W^+_n}$ is a neighborhood  of $x_0+iy_0$ in $\{w\in\C: \Im w\geq y_0\}$, it follows that $\eta(t)\in  R^+_n\cup W^+_n$ eventually. But $\eta(t)\in\Omega$ for all $t\in[0,1)$, hence, $\eta(t)\in V^+_n$ eventually, by \eqref{Eq:innerV_n+}. This implies that  $\eta(t)$ converges to $\hat{h}(\underline{\zeta_2})$ in the Carath\'eodory topology of $\Omega$ as $t\to 1$. Therefore, $\lim_{t\to 1^-}h^{-1}(\eta(t))=\zeta_2$. But, if $\gamma:(0,1)\to \D$ is the continuous curve defined above such that $\lim_{t\to 1^-}\gamma(t)=\sigma_2$, we have that $h(\gamma(t))$ is a continuous curve in $\Omega$ converging to $x_0+iy_0$ as $t\to 1^-$ such that $\Im h(\gamma(t))>y_0$ for $t$ close to $1$. Hence, $\zeta_2=\sigma_2$. A similar argument shows that if $\Im \eta(t)<y_0$ for every $t> t_0$, for some fixed $t_0\in[0,1)$ then $h^{-1}(\eta(t))$ converges to $\sigma_1$ as $t\to 1$. This proves the first part of $(\star)$. As for the last part, assume $h^\ast(\sigma)=x_0+iy_0$. Hence, the continuous curve $[0,1)\ni t\mapsto h(t\sigma)$ converges to $x_0+iy_0$ for $t\to 1^-$. Thus $t\sigma=h^{-1}(h(t\sigma))$ converges to either $\sigma_1$ or $\sigma_2$ as $t\to 1^-$, that is, either $\sigma=\sigma_1$ or $\sigma=\sigma_2$.

Now we show that $\sigma_1$ and $\sigma_2$ are the starting points of two non-exceptional maximal contact arcs $A_1$ and $A_2$ with the same final point.

To this aim, \eqref{Eq:creo-arco-doppio} and \eqref{Eq:creo-arco-doppio2} and Schwarz reflection principle imply that $h^{-1}$ restricted to $\Omega\cap\{w\in\C: \Im w>y_0\}$ extends holomorphically through $B:=\{w\in \C: x_0< \Re w< x_0+t_0, \Im w=y_0\}=h^\ast(A)$. Hence, there exists an open arc $A_1\subset A$ such that $h$ extends holomorphically through $A_1$ and $h(A_1)=B$. Similarly, considering $\Omega\cap\{w\in\C: \Im w<y_0\}$, one can show that there exists an open arc $A_2\subset A$ such that $h$ extends holomorphically through $A_2$ and $h(A_2)=B$. By Lemma~\ref{Lem:cara-contact}, $A_1, A_2$ are (different) contact arcs  with the same final point different from the Denjoy-Wolff point. 

Since $\sigma_2\in \mathcal S^+(h)\cap \mathcal S^-(h)$, it follows from Lemma~\ref{Lem:discont-Koenigs}  that for every $x<x_0$ the point $x+iy_0$ is not accessible from $\Omega$ by any injective continuous curve, and thus cannot belong to the image via $h$ of any point of $A_1, A_2$. On the other hand, if $x_0<x<x_0+t_0$, then $x+t+iy_0\in B$ for all $t\in [0,x_0+t_0-x]$. Therefore, the image of the starting point of $A_j$ under $h^\ast$, $j=1,2$, is either $x_0+iy_0$ or $x_0+t_0+iy_0$. Since the flow is given by $z\mapsto z+t$, $t\geq 0$, it follows  that the starting point of $A_j$ has image $x_0+iy_0$ via $h^\ast$, $j=1,2$. Hence, by \eqref{Eq:where-go-int-fin-max-arc} and $(\star)$ we have that the starting points of $A_1$ and $A_2$ are $\sigma_1$ and $\sigma_2$.

By $(\star)$ we also see that there cannot be three different points $\sigma_1, \sigma_2, q\in\mathcal S^-(h)$ such that $\Im h^\ast(\sigma_1)=\Im h^\ast(\sigma_2)=\Im h^\ast(q)$. Indeed, as we already noticed at the beginning of the proof, this would imply that $h^\ast(\sigma_1)= h^\ast(\sigma_2)= h^\ast(q)$, contradicting $(\star)$.

Summing up, we have proved that for every $y\in \overline{I}$, the pre-image of $y$ under the map  $\mathcal S^-(h)\ni \sigma\mapsto \Im h^\ast(\sigma)$ has at most two points. Moreover, if  $\sigma_1, \sigma_2\in \mathcal S^-(h)$, $\sigma_1\neq\sigma_2$ are such that $\Im h^\ast(\sigma_1)=\Im h^\ast(\sigma_2)=y$ then $y\in I$, $h^\ast(\sigma_1)=h^\ast(\sigma_2)$, $\sigma_1, \sigma_2\in \mathcal S^+(h)$ and there exist two maximal contact arcs starting, respectively, from $\sigma_1$ and $\sigma_2$ and ending at the same final point. A similar conclusion holds for $\mathcal S^+(h)$.

Since $\mathcal S(h)\setminus\{\tau\}=\mathcal S^-(h)\cup \mathcal S^+(h)$, it follows that the fibres of $\mathcal S(h)\setminus\{\tau\}\ni \sigma\mapsto \Im h^\ast(\sigma)\in \overline{I}$ have at most two points. 
Moreover, if a fiber of this map has two points, then either one point is in  $\mathcal S^+(h)\setminus \mathcal S^-(h)$ and the other in $\mathcal S^-(h)\setminus \mathcal S^+(h)$ or they are both in $\mathcal S^-(h)\cap \mathcal S^+(h)$.

In order to end the proof, we have to show that if, say, $\sigma_1\in \mathcal S^+(h)\setminus \mathcal S^-(h)$ and  $\sigma_2 \in \mathcal S^-(h)\setminus \mathcal S^+(h)$ are such that $y:=\Im h^\ast(\sigma_1)=\Im h^\ast(\sigma_2)$, then $y\in I$ and $h^\ast(\sigma_1)\neq h^\ast(\sigma_2)$.
But this follows immediately from the very definitions of  $\mathcal S^-(h)$ and $\mathcal S^+(h)$.
 \end{proof}

Now we examine certain special discontinuities. To this aim, we need a definition:
\begin{definition}
Let $H=\R$ or $H=\partial \D$. A subset $A\subset H$ is a {\sl Cantor set} if every point of $A$ is an accumulation point of $A$ and $A$ has empty interior.
\end{definition}

Note that, by definition, we are not assuming that a Cantor set is closed.

Now, for discontinuities of a Koenigs  function we give a couple of definitions that depend on their impressions.

Given $z_0\in\C, \delta>0, \epsilon>0$ an {\sl open box} is
\[
B(z_0,\epsilon, \delta):=\{z\in \C: |\Re z-\Re z_0|<\epsilon, |\Im z-\Im z_0|<\delta\}.
\]

\begin{definition}
Let $h$ be a Koenigs  function of a non-elliptic semigroup and let $\Omega=h(\D)$.

A set $\{\sigma_\alpha\}_{\alpha\in \mathcal C}\subset \mathcal S(h)$ is a {\sl Cantor comb of discontinuities} if it forms a Cantor set on $\partial \D$ and if there exists an open box $B\subset \overline{\Omega}$ such that $h^\ast(\sigma_\alpha)\in B$ for all $\alpha\in \mathcal  C$ and
\begin{equation}\label{Eq:Comb-bajo}
\begin{split}
\Omega\cap B&=B\setminus \cup_{\alpha\in\mathcal  C} I(\hat{h}(\underline{\sigma_\alpha}))\\ &=B\setminus \cup_{\alpha\in \mathcal C}\{z\in \C: \Im z=\Im h^\ast(\sigma_\alpha), \Re z\leq \Re h^\ast(\sigma_\alpha)\}.
\end{split}
\end{equation}
\end{definition}

We are now going to give some better description of Cantor combs of discontinuities, also in terms of the defining function of the Koenigs  domain (see Proposition~\ref{Prop:definingfnc}).

As a matter of notation, if $x,y\in \R$, we let
\[
L_{x,y}:=\{z\in \C: \Re z\leq x, \Im z=y\}.
\]

\begin{proposition}\label{Prop:CantorComb}
Let $(\varphi_t)_{t\geq 0}$ be a non-elliptic semigroup, with Denjoy-Wolff point $\tau\in\partial\D$, $h:\D\to\C$  its Koenigs  function and $\Omega:=h(\D)$. Let  $\psi:I\to [-\infty,+\infty)$ be the defining  function of $\Omega$.  Then the following are equivalent:
\begin{enumerate}
\item  There exists $\{\sigma_\alpha\}_{\alpha\in C_1}\subset \mathcal S(h)$  a Cantor comb of discontinuities.
\item There exist an open box $B\subset \overline{\Omega}$, a Cantor set $\mathcal C_2\subset\R$ and a set of real numbers $\{x_c\}_{c\in \mathcal C_2}$ such that  $x_c+ic\in  B$ for all $c\in \mathcal C_2$,
\begin{equation}\label{Eq:Cantor-mas-fino}
\Omega\cap B=B\setminus \cup_{c\in \mathcal C_2} L_{x_c,c},
\end{equation}
and $\limsup_{\mathcal C_2\ni c\to c_0}x_{c}=x_{c_0}$ for every $c_0\in \mathcal C_2$.
\item There exist an interval $J\subset I$, a Cantor set $\mathcal C_3\subset J$ and $q \in \R$,  such that
$\psi(c)> q$ for all $c\in \mathcal C_3$ and $\psi(y)\leq q$ for all $y\in J\setminus \mathcal C_3$ and $\limsup_{\mathcal C_3\ni c\to c_0}\psi(c)=\psi({c_0})$ for every $c_0\in \mathcal C_3$.
\end{enumerate}
\end{proposition}

\begin{proof}  (1) implies (2).
Let $\mathcal C_2:=\{\Im h^\ast(\sigma_\alpha): \alpha\in \mathcal C_1\}$. By definition of Cantor comb of discontinuities, \eqref{Eq:Cantor-mas-fino} holds.

Now we show that $\mathcal C_2$ is  a Cantor set.
Since by definition of Cantor comb of discontinuities, for all $\alpha\in \mathcal C_1$ and $t>0$ we have $h^\ast(\sigma_\alpha)+t\in\Omega$, it follows by
Proposition~\ref{Lem:two-disc}
that
\[
x_c:=\max\{\Re h^\ast(\sigma_\alpha): \alpha\in \mathcal C_1, \Im h^\ast(\sigma_\alpha)=c\}
\]
is well defined. Note that, since the impressions of prime ends are contained in $\partial\Omega$, we have $L_{x_c,c}\cap B\subset \partial \Omega$.

First we show that the interior of $\mathcal C_2$ is empty. If $R_0=\inf\{x: x+iy\in B\}$, by definition of Cantor comb, we have $x_c>R_0$ for all $c\in \mathcal C_2$. Let $n=1, 2\ldots$ and set
\[
\mathcal K(n):=\{c\in \mathcal C_2: x_c>R_0+\frac{1}{n}\}.
\]
Clearly  $\mathcal C_2=\cup_n \mathcal K(n)$. We claim that $\hbox{Int}(\overline{\mathcal K(n)})=\emptyset$ for all $n$, from which it follows from Baire category theorem that $\mathcal C_2$ has empty interior.

Assume there exists $n$ and an open set $A$ such that  $A\subset \overline{\mathcal K(n)}$. Thus, any point  $a\in A$ is a limit of points in $\mathcal K(n)$. Since for every point $c\in K(n)$
	one has $\{w\in\C:\Im w=c, R_0\leq\Re w\leq R_0+\frac{1}{n}\}\subset\partial\Omega$ and $\partial \Omega$ is closed, it follows that also $\{w\in\C:\Im w=a, R_0\leq \Re w\leq R_0+\frac{1}{n}\}\subset\partial\Omega$.
Therefore, the open set $\{w\in\C:\Im w\in A, R_0<\Re w< R_0+\frac{1}{n}\}\subset\partial\Omega$, which is a contradiction.

Since $\partial \Omega$ is closed, we have $\limsup_{\mathcal C_2\ni \tilde c\to c} x_{\tilde c}\leq x_{c}$ for every $c\in\mathcal C_2$.

Therefore, in order to prove (2), we are left to show that every point in $\mathcal C_2$ is an accumulation point in $\mathcal C_2$ and that for every $c\in \mathcal  C_2$ there exists a sequence $\{c_n\}\subset \mathcal C_2$ such that $\lim_{n\to \infty}x_{c_n}=x_{c}$. To this aim, given  $c\in \mathcal C_2$, let $\alpha\in \mathcal C_1$ be such that $h^\ast(\sigma_\alpha)=x_c+ic$. By Lemma~\ref{Lem:discont-Koenigs},  we can assume that there exists a sequence $\{z_n^-\}\subset\partial\Omega$ such that $\{\Im z_n^-\}$ is strictly increasing, $z_n^-+t\in\Omega$ for $t>0$ and $\{z_n^-\}$ converges to $x_c+ic$ (the case of a sequence  $\{z_n^+\}\subset\partial\Omega$ such that $\{\Im z_n^+\}$ is strictly increasing is similar). For $n$ sufficiently large, $\{z_n^-\}$ is contained in the open box $B$ given by the definition of Cantor comb of discontinuities. Therefore, since $z_n^-+t\in \Omega$ for $t>0$, from \eqref{Eq:Comb-bajo} we obtain that $z_n^-=h^\ast(\sigma_{\alpha_n})$ for some $\alpha_n\in \mathcal C_1$. Letting $x_{c_n}+ic_n=h^\ast(\sigma_{\alpha_n})$, it is clear that $\{c_n\}$ accumulates to $c$ and that $\lim_{n\to \infty}x_{c_n}=x_{c}$.

(2) implies (3).
 Let $z_0\in\C$,
$\epsilon>0, \delta>0$ be such that $B=B(z_0,\epsilon,\delta)$. Let $J:=(\Im z_0-\delta, \Im z_0+\delta)$. Note that, since $x_c+ic\in B$ for all $c\in \mathcal C_2$, it follows that $\mathcal  C_2\subset J$. By the very definition of $\psi$ and \eqref{Eq:Cantor-mas-fino}, we have $\psi(c)=x_{c}$ for all $c\in \mathcal C_2$. Take $c_0\in \mathcal C_2$ and let
$\Re z_0-\epsilon<q<x_{c_0}$.
Note that $q+ic_0\in B$. Then define $\mathcal C_3:=\{c\in \mathcal C_2: x_{c}>q\}$. For every $c\in \mathcal C_2$ there exists a sequence $\{c_n\}$
of points of $C_2$ such that $c_n\ne c$ for all $n$ and $\lim_{n\to\infty}c_n=c$. It follows that $\mathcal C_3$ is a Cantor set. By construction, $\psi(c)>q$ for all $c\in C_3$ and $\psi(y)\leq q$ for all $q\in J\setminus \mathcal C_3$.

(3) implies (1). Let $c_0\in \mathcal  C_3$ be a point in the interior of $J$. Since $\psi$ is upper semicontinuous, for every $\epsilon>0$ there exists $\delta>0$ such that $\psi(t)<\psi(c_0)+\epsilon$ for every $t$ such that $|t-c_0|<\delta$. Fix $\epsilon>0$ such that $\psi(c_0)-2\epsilon>q$ and let $\delta>0$ as before. We can assume, up to taking $\delta$ smaller, that  $[c_0-\delta,c_0+\delta]\subset J$.  Let
$\mathcal C_1:=\mathcal C_3\cap (c_0-\delta,c_0+\delta)$.
Note that $\mathcal C_1$ is a Cantor set.
Let $B:=\{z\in\C: |\Re z-\psi(c_0)|<2\epsilon, |\Im z-c_0|<\delta\}$. By construction and hypothesis on $\psi$, we have
\begin{equation}\label{Eq:box-from-psi}
B\cap \Omega=B\setminus \cup_{c\in \mathcal C_1} L_{\psi(c),c},
\end{equation}
and $\psi(c)+ic\in B$ for all $c\in \mathcal C_1$.

Therefore, for every $c\in \mathcal C_1$  the Jordan arc $\gamma_c:(0,+\infty)\to\C$ defined as $\gamma_c(t):=\psi(c)+t+ic$ is contained in $\Omega$.
Hence,  (see, {\sl e.g.}, \cite[Prop. 3.3.3]{BCDbook}), for every $c\in \mathcal C_1$  there exists $\sigma_c\in\partial\D$ such that $\sigma_c=\lim_{t\to 0^+}h^{-1}(\gamma_c(t))$ and
\[
\psi(c)+ic=h^\ast(\sigma_c).
\]
By  Lemma~\ref{Lem:discont-Koenigs}, it follows that $L_{\psi(c),c}\cap B$ contains the impression of the prime-end $\hat{h}(\underline{\sigma_c})$.
Hence, $\sigma_c\in \mathcal S(h)$. By \eqref{Eq:box-from-psi}, we are left to prove that $\{\sigma_c\}_{c\in\mathcal C_1}$ forms indeed a Cantor set in $\partial \D$.

Fix $c\in \mathcal C_1$. Let $\{c_m\}\subset\mathcal C_1$ be a sequence converging to $c$ such that $\psi(c_m)\to \psi(c)$.  We prove that $\{\sigma_{c_m}\}$ converges, up to subsequences,  to $\sigma_c$.

Indeed, assume by contradiction that there exists $C>0$ such that $|\sigma_{c_m}-\sigma_c|\geq C$ for all $m$. Let $\{r_n\}$ be a sequence of positive numbers strictly decreasing to $0$. Since $\psi$ is upper semicontinuous, for every $n$ there exists $\delta_n>0$ such that $\psi(y)<\psi(c)+r_n$ for all $y\in\R$ such that $|y-c|<\delta_n$.  Thus, if $c_m$ is such that $|c_m-c|<\delta_n$, the  Jordan arc
\[
\Gamma_n:=([\psi(c_m), \psi(c)+r_n]+ic_m) \cup (\psi(c)+r_n+i[c_m,c])\cup ([\psi(c), \psi(c)+r_n]+ic)
\]
connects $h^\ast(\sigma_c)$ to $h^\ast(\sigma_{c_m})$ and $\Gamma_n\setminus\{h^\ast(\sigma_c), h^\ast(\sigma_{c_m})\}\subset\Omega$.

It follows that $h^{-1}(\Gamma_n)$ is a Jordan arc  connecting $\sigma_c$ and $\sigma_{c_m}$ and $A_n:=h^{-1}(\Gamma_n)\setminus\{\sigma_c, \sigma_{c_m}\}\subset\D$. By hypothesis, the diameter of $A_n$ is bounded from below by $C$ (since $|\sigma_{c_m}-\sigma_c|\geq C$ for all $m$). However, $\{h(A_n)\}$ converges to $h^\ast(\sigma_c)$---in the sense that for every $w_n\in h(A_n)$, $\lim_{n\to\infty}w_n=h^\ast(\sigma_c)$. Thus, $\{A_n\}$ is a sequence of Koebe arcs for $h$, against the ``no Koebe arcs Theorem'' (see, {\sl e.g.}, \cite[Theorem~3.2.4]{BCDbook}.

Thus, $\{\sigma_{c_m}\}$ converges, up to subsequences,  to $\sigma_c$---actually, using the same token for every subsequence, one can prove that $\{\sigma_{c_m}\}$ converges  to $\sigma_c$, although this is not needed here. By the arbitrariness of $c$, it follows that  $\{\sigma_{c}\}$ is a Cantor set, and we are done.
\end{proof}


\begin{example}
We construct a non-elliptic semigroup such that its Koenigs  function has  a Cantor comb of discontinuities. To this aim, let $\mathcal C\subset \R$ be a closed Cantor set. Let $\psi:\R\to [0,1]$ be defined as $\psi(c)=1$ for $c\in \mathcal C$ and $\psi(y)=0$ for $y\in \R\setminus\mathcal C$. By construction $\psi$ is upper semicontinuous. Then
$\Omega_\psi$ is the Koenigs  domain of a parabolic semigroup in $\D$. By Proposition~\ref{Prop:CantorComb} such a semigroup has a Cantor comb of discountinuities.
\end{example}

\subsection{Characterization}

Now we characterize those Koenigs  domain $\Omega$ such that $\hbox{Int}(\overline{\Omega})=\Omega$. To this aim we give the following:

\begin{definition}\label{Def:psi-tilde}
	Let $I\subseteq \R$ be an open interval and let $\psi:I\to [-\infty,+\infty)$ be an upper semicontinuous function. The lower semicontinuous regularization $\psi_\ast:  I\to [-\infty,+\infty)$ of $\psi$ is defined, for $x\in  I$, as
	\[
	\psi_\ast(x):=\liminf_{t\to x}\psi(t).
	\]
	Then we let $\tilde{\psi}:I\to [-\infty,+\infty)$ be the upper semicontinuous regularization of $\psi_\ast$, defined for $x\in I$ as
	\[
	\tilde{\psi}(x):=\limsup_{t\to x}\psi_\ast(t).
	\]
\end{definition}

Note that, since $\psi$ is upper semicontinuous, it follows that $\psi_\ast(x)\leq \psi(x)$---and  hence $\tilde{\psi}(x)\leq \psi(x)$---for all $x\in I$.

\begin{theorem}\label{Thm:Cantor-no-approx}
Let $(\varphi_t)_{t\geq 0}$ be a non-elliptic semigroup, with Denjoy-Wolff point $\tau\in\partial\D$, $h:\D\to\C$  its Koenigs  function and $\Omega:=h(\D)$. Let $\psi:I\to [-\infty,+\infty)$ be the defining function of $\Omega$. Then the following are equivalent:
\begin{enumerate}
\item  If $A, B$ are non-exceptional maximal contact arcs for $(\varphi_t)_{t\geq 0}$ with final points $x_1(A)$ and $x_1(B)$ then   $x_1(A)\neq x_1(B)$, and  $\mathcal S(h)$ does not contain any Cantor comb.
\item $\psi(y)=\tilde{\psi}(y)$ for every $y\in I$.
\item $\hbox{Int}(\overline{\Omega})=\Omega$.
\end{enumerate}
\end{theorem}
\begin{proof}
First, we extend $\psi_\ast$ to $\overline I$ by setting $\psi_\ast(a):=\liminf_{I\ni y \to a}\psi(y)$ for $a\in\partial I$. Then it is easy to see that
\[
\overline \Om=\overline \Om_\psi=\{x+iy: y\in \overline I, \,
x\ge \psi_\ast(y)\},
\]	
where, if $y\in\partial I$ and $\psi_\ast(y)=+\infty$, the inequality $x \ge +\infty$ means that $x+iy
\not\in\overline{\Omega}$ for any $x\in\R$. Using the previous equality, one can easily check that
\[
\hbox{Int}(\overline \Om)=\{x+iy: y\in I, \,
x > \tilde\psi(y)\}.
\]
 Hence (2) is equivalent to (3).

Now let us see that (1) implies (2). Arguing by contradiction, we assume that there exists $y_0\in I$ such that $\psi(y_0)>\tilde{\psi}(y_0)$. Let $q\in \R$ be such that $\psi(y_0)>q>\tilde{\psi}(y_0)$.  Now, for every $\epsilon>0$ such that $(y_0-\epsilon, y_0+\epsilon)\subset I$  let
\[
C_\epsilon(q):=\{y\in (y_0-\epsilon, y_0+\epsilon): \psi(y)>q\}.
\]
If for every $\epsilon>0$ the set $C_\epsilon(q)$ contains some open interval $(a_\epsilon, b_\epsilon)$, it would follow that $\psi_\ast(a_\epsilon)>q$, and hence, since $a_\epsilon\to y_0$ as $\epsilon\to 0$, we would have $\tilde{\psi}(y_0)\geq q$. Hence, there exists $\epsilon_0>0$ such that $\hbox{Int}(C_{\epsilon}(q))=\emptyset$ for all $0<\epsilon\leq \epsilon_0$.

If there exists $0<\epsilon\leq \epsilon_0$ such that $C_\epsilon(q)=\{y_0\}$, then  by Corollary~\ref{Cor:arco-en-def} the semigroup $(\varphi_t)_{t\geq 0}$ has two non-exceptional maximal contact arcs $A, B$  with final points $x_1(A)=x_2(B)$, contradicting (1).

Therefore, we can assume that for every $\epsilon>0$ the set $C_\epsilon(q)$ contains infinitely many points. If some point in $C_\epsilon(q)$ is isolated, repeating the previous argument at such a point, we would get again two   non-exceptional maximal contact arcs $A, B$  with final points $x_1(A)=x_2(B)$, contradicting (1). Therefore, $C_\epsilon(q)$ is a Cantor set in $(y_0-\epsilon, y_0+\epsilon)$ for every $0<\epsilon\leq \epsilon_0$.

Let $\{q_n\}$ be a increasing sequence converging to $\psi(y_0)$ and let $\{\epsilon_n\}$ be a decreasing sequence of positive numbers converging to $0$. Repeating the previous argument, for every $n$ large, we can find $y_n\in C_{\epsilon_n}(q_n)\subset C_{\epsilon_0}(q)$. Thus,
\[
\limsup_{C_{\epsilon_0}(q)\ni y\to y_0}\psi(y) \geq \limsup_{n\to\infty}\psi(y_n)\geq \psi(y_0).
\]
 Since $\psi$ is upper semicontinuous, it actually follows that $\limsup_{C_{\epsilon_0}(q)\ni y\to y_0}\psi(y)=\psi(y_0)$. Hence, by Proposition~\ref{Prop:CantorComb}, $\mathcal S(h)$ contains a Cantor comb, against ~(1). Therefore, ~(1) implies ~(2).

Next, if (3)  holds, it follows from Lemma~\ref{Lem:cara-contact} and the second equivalent condition in Proposition~\ref{Prop:CantorComb} that (1) holds.
\end{proof}

\begin{example}\label{Ex:Cantor-good}
	We construct a semigroup such that its Koenigs function has  a Cantor set of  discontinuities which is not a Cantor comb of discontinuities.  Let $\mathcal C$ be the ternary Cantor set in $[0,1]$.
	Define a function $\psi$ on $(0,1)$ as follows.
	If $(a,b)\subset (0,1)$ is one of the complementary intervals
	of $\mathcal C$ and $x\in (a,b)$, then we set $\psi(x)=\sin((b-x)^{-1}(x-a)^{-1})$. If $x\in (0,1)\cap C$, then we set $\psi(x)=1$. It is easy to see that $\psi$ is upper semicontinuous and its set of discontinuities is $\mathcal C\setminus \{0,1\}$.
	
	Consider a hyperbolic semigroup whose
	Koenigs domain is $\Om_\psi$.
	To each point  $c\in \mathcal C \setminus \{0,1\}$ there corresponds a prime end whose impression is $[-1,1]+ic$. Hence the Koenigs function of the semigroup has a Cantor set of discontinuities, but there are no Cantor combs.
\end{example}

Observe that by Theorem~\ref{Thm:density-in-strip},  Corollary~\ref{Cor:complement}, and Theorem~\ref{Thm:Cantor-no-approx},  for the domain $\Om$ of the previous example, $\mathcal E_\infty(\Omega)$ is weak-star dense in $H^\infty(\Omega)$.

\section{Boundary fixed points, unbounded discontinuities, and contact arcs}\label{Sec:bum}

The aim of this section is to study the interplay between boundary fixed points, a type of discontinuities that we call ``unbounded'' and maximal contact arcs, also in relation with the defining function of the Koenigs' domain.

\subsection{Boundary fixed points}

As we already recalled, if $(\varphi_t)_{t\geq 0}$ is a semigroup in $\D$, then for every $\sigma\in\partial \D$ and for every $t\geq 0$, the non-tangential limit $\varphi^\ast_t(\sigma):=\angle\lim_{z\to\sigma}\varphi_t(z)$
exists.

\begin{definition}
Let $(\varphi_t)_{t\geq 0}$ be a semigroup in $\D$ and $\sigma\in \partial \D$. Then $\sigma$ is a {\sl boundary fixed point} of $(\varphi_t)_{t\geq 0}$ if $\varphi_t^\ast(\sigma)=\sigma$ for every $t\geq 0$.
\end{definition}

It can be proved that a point $\sigma$ is a boundary fixed point of $(\varphi_t)_{t\geq 0}$ if and only if there exists $t_0>0$ such that $\varphi_{t_0}^\ast(\sigma)=\sigma$. Moreover, if $(\varphi_t)_{t\geq 0}$ is non-elliptic, the Denjoy-Wolff point is a boundary fixed point (see {\sl e.g.}, \cite[Chapter 11]{BCDbook} for details).

In the sequel we will use some facts, that we are going to collect in the following lemmas. We start with the following (see, {\sl e.g.} \cite[Prop.~13.6.1]{BCDbook})

\begin{lemma}\label{lem:h-bfp-simple} Let $(\varphi_t)_{t\geq 0}$ be a non-elliptic semigroup in $\D$ with Koenigs  function  $h$ and Denjoy-Wolff point $\tau$. Let $\sigma\in\partial\D$. Then $\sigma$ is a boundary fixed point of $(\varphi_t)_{t\geq 0}$ if and only if
\begin{equation}\label{Eq:lim-h-at-bfp-simple}
h^\ast(\sigma)=\angle\lim_{z\to \sigma} h(z)=\infty.
\end{equation}
Moreover, if $\sigma\neq\tau$, then $\lim_{z\to\sigma}h(z)=\infty$.
\end{lemma}

Next, (see, {\sl e.g.} \cite[Theorem~13.5.5 and Theorem~13.6.6]{BCDbook}):

\begin{lemma}\label{lem:h-bfp} Let $(\varphi_t)_{t\geq 0}$ be a non-elliptic semigroup in $\D$ with Koenigs  function  $h$ and Denjoy-Wolff point $\tau$. Let $\sigma\in\partial\D\setminus\{\tau\}$. Then $\sigma$ is a boundary fixed point if and only if
\begin{equation}\label{Eq:lim-h-at-bfp}
\lim_{z\to \sigma}\Re h(z)=-\infty, \quad -\infty<a:=\liminf_{z\to \sigma}\Im h(z)\leq b:=\limsup_{z\to \sigma}\Im h(z)<+\infty.
\end{equation}
Moreover, if  $a<b$, then $\{w\in\C: a<\Im w<b\}\subset \Omega$ and  $\{w\in\C: a-\delta<\Im w<b\}\not\subset \Omega$, $\{w\in\C: a<\Im w<b+\delta\}\not\subset \Omega$ for all $\delta>0$. Also, $\lim_{t\to+\infty}h^{-1}(\gamma(t))=\sigma$ for every continuous curve $\gamma:[0,+\infty)\to\Omega$ such that $\lim_{t\to+\infty}\Im \gamma(t)\in (a,b)$
and $\lim_{t\to+\infty}\Re \gamma(t)=-\infty$.

While, if $a=b$, then $\{w\in\C: a-\delta<\Im w\leq a\}\not\subset \Omega$ and $\{w\in\C: a\leq \Im w<a+\delta\}\not\subset \Omega$ for all $\delta>0$.
\end{lemma}

The previous lemma allows us to give the following definition:

\begin{definition}
Let $(\varphi_t)_{t\geq 0}$ be a non-elliptic semigroup in $\D$ with Koenigs function $h$ and Denjoy-Wolff point $\tau\in\partial D$. Let $\sigma\in \partial\D\setminus\{\tau\}$ be a boundary fixed point for $(\varphi_t)_{t\geq 0}$. Let
\begin{equation}\label{Eq:defino-a-por-sigma}
a_-(\sigma):=\liminf_{z\to \sigma}\Im h(z), \quad a_+(\sigma):=\limsup_{z\to \sigma}\Im h(z).
\end{equation}
If $a_-(\sigma)<a_+(\sigma)$, the point $\sigma$ is a {\sl boundary regular fixed point}. We denote by   $\mathcal B\mathcal F_R$  the set of  boundary regular fixed points $\sigma\neq\tau$. If $a_-(\sigma)=a_+(\sigma)$, the point $\sigma$ is a {\sl super-repelling boundary fixed point} and we denote by  ${\mathcal B\mathcal F}_N$   the set of super-repelling boundary fixed points.
\end{definition}

\begin{remark}
It is known (see, {\sl e.g.}, \cite[Chapters 12 and 13]{BCDbook}) that $\sigma\in \mathcal B\mathcal F_R$  if and only if $\lim_{z\to \sigma}\varphi_t(z)=\sigma$ and $\lim_{z\to \sigma}|\varphi'_t(z)|<+\infty$ for every $t>0$, while $\sigma\in \mathcal B\mathcal F_N$  if and only  $\lim_{z\to \sigma}\varphi_t(z)=\sigma$ and $\lim_{z\to \sigma}|\varphi'_t(z)|=+\infty$ for every $t>0$.
\end{remark}

\begin{proposition}\label{Prop:char-bfp}
Let $(\varphi_t)_{t\geq 0}$ be a non-elliptic semigroup in $\D$ with Denjoy-Wolff point $\tau\in\partial\D$. Let $h$ be its Koenigs function, $\Omega=h(\D)$  the associated Koenigs  domain. Let $\psi:I\to \R$ be the defining function of $\Omega$.
\begin{enumerate}
\item Let $\Theta_\infty$  be the (possibly empty) union of the open unbounded connected components of $\psi^{-1}(-\infty)$.
Then $\lim_{n\to\infty}h^{-1}(w_n)=\tau$ for every sequence $\{w_n\}\subset\Omega$ such that  $\lim_{n\to \infty}\Re w_n=-\infty$ and $\lim_{n\to \infty}\Im w_n=y_0$ for some $y_0\in \Theta_\infty$.
\item Let $I_R$ be the set of all bounded open connected components of $\psi^{-1}(-\infty)$. There is a one-to-one correspondence between $\mathcal B\mathcal F_R$ and $I_R$   given by
\[
\mathcal B\mathcal F_R\ni\sigma\mapsto (a_-(\sigma), a_+(\sigma))\subset\R.
\]
\item Assume that $\hbox{Int}(\overline\Omega)=\Omega$. Let $J_\infty$ be the union of the open connected components of $\psi^{-1}(-\infty)$ and let $I_N$ be the set of points $y_0\in \overline{I}\setminus \overline{J_\infty}$ such that either $\lim_{y\to y_0^-}\psi(y)=-\infty$ or $\lim_{y\to y_0^+}\psi(y)=-\infty$. There is a one-to-one correspondence between $\mathcal B\mathcal F_N$ and $I_N$ given by
\[
\mathcal B\mathcal F_N\ni\sigma \mapsto a_-(\sigma).
\]
\end{enumerate}
\end{proposition}
\begin{proof}
	In the sequel we will use the following classical result about boundary behaviour of univalent functions (see, {\sl e.g.}, \cite[Prop. 3.3.5]{BCDbook}):
	
	\medskip
	$(\star)$ Let $\gamma:(-\infty,+\infty)\to \Omega$ be an injective continuous curve such that $\lim_{t\to \pm \infty}\gamma(t)=\infty$. Let $U$ be a connected component of $\C_\infty\setminus \overline{\gamma((-\infty,+\infty))}^\infty$. If $U\subset\Omega$ then there exists  $\xi\in\partial \D$ such that $\lim_{t\to\pm \infty}h^{-1}(\gamma(t))=\xi$.
	\medskip
	
	(1) The result follows from the behaviour of $h$ on the so-called ``parabolic petals'' (see, {\sl e.g.}, \cite[Chapter 13]{BCDbook}). However, we give here a direct proof. Assume $\Theta_\infty\neq \emptyset$. Note that $\Theta_\infty$ has one or two connected components. We assume $(-\infty, c)$ is one of such component for some $c\in\R$ (the other case is similar and we omit it).  Let $\{w_n\}\subset \Omega$ be such that $\lim_{n\to\infty}\Re w_n=-\infty$ and $\lim_{n\to\infty}\Im w_n=y_0$ for some $y_0<c$. Let $y_1\in (y_0,c)$. We can assume that $\Im w_n<y_1$ for all $n$. Let $\gamma(t):=t+iy_1$, $t\in \R$. Since $U:=\{w\in\C: \Im w<y_1\}\subset\Omega$, it follows by $(\star)$ that $\Gamma:=\overline{h^{-1}(\gamma((-\infty,+\infty))}$ is a Jordan curve such that $\Gamma\cap\partial\D=\{\tau\}$. Note also that $h^{-1}(U)$ is the bounded connected component of $\C\setminus \Gamma$. Since $\{h^{-1}(w_n)\}\subset h^{-1}(U)$ and it converges to the boundary of $\D$, it follows that $\{h^{-1}(w_n)\}$ converges to $\tau$.
	
	(2) follows at once from Lemma~\ref{lem:h-bfp}.
	
	(3) First, we show that if $\sigma\in \mathcal B\mathcal F_N$ then $a_-(\sigma)\in I_N$. Indeed,
	let $\alpha_1(t)=h(t\sigma)$ and let
	$M_1:=\Im \alpha_1([0,1))$. By  Lemma~\ref{lem:h-bfp},  $M_1\subset I$ is either an interval or a point such that $a_-(\sigma)\in\overline{M_1}$. If $M_1$ is an interval, since $\psi(\Im \alpha_1(t))<\Re \alpha_1(t)$ for every $t\in [0,1)$ and $\lim_{t\to 1^-}\Re \alpha_1(t)=-\infty$, it follows easily that  either $\lim_{y\to a_-(\sigma)^-}\psi(y)=-\infty$ or $\lim_{y\to a_-(\sigma)^+}\psi(y)=-\infty$ (or both). If $M_1$ is a point, then necessarily $M_1=\{a_-(\sigma)\}$ and $\alpha_1(t)\in \{w\in\C: \Im w=a_-(\sigma)\}$ for all $t\in [0,1)$---in particular, $a_-(\sigma)\in I$.
	Then choose any $\xi$ in $\D$ such that $\xi\overline{\sigma}\notin\R$. Put $\alpha_2(t)=h(t\sigma+(1-t)\xi)$. Since $h$ is injective, it is easy to see that $M_2:=\Im \alpha_2([0,1))$
	has to be an interval and cannot be a point. By repeating
	the above argument for the curve $\alpha_2$ in place of
	$\alpha_1$, we conclude that $a_-(\sigma)\in I_N$.

	Next, we show that $\mathcal B_N\ni\sigma \mapsto a_-(\sigma)\in I_N$ is injective. Assume by contradiction that there exist $\sigma_1, \sigma_2\in \mathcal B_N$, $\sigma_1\neq \sigma_2$, such that $a_-(\sigma_1)=a_-(\sigma_2)$. Let	$\gamma_j(t):=h(t\sigma_j)$,
	$t\in [0,1)$ and $j=1,2$. Let $\Gamma$ be the closure in $\C_\infty$ of $\gamma_1([0,1))\cup\gamma_2([0,1))$. Then by Lemma~\ref{lem:h-bfp}, $\Gamma$ is a Jordan curve in $\C_\infty$. By the same lemma, it follows that there exists $R\in\R$ such that $\Gamma\cap\C$ is contained in $\{w\in\C: \Re w<R\}$. Let $U$ be the connected component of $\C_\infty\setminus\Gamma$		which is contained in $\{w\in\C: \Re w<R\}$.
	
	Since we are assuming that $\sigma_1\neq\sigma_2$, by $(\star)$ we have $U\not\subset\Omega$. We show that, if this is the case, then $\hbox{Int}(\overline\Omega)\neq \Omega$, reaching a contradiction.
	
	Indeed, assume $w\in U\setminus\Omega$. We have that $w-t\in U\setminus\Omega$ for all $t>0$, since  $\Omega$ is starlike at infinity and $\partial U=\Gamma\cap \C\subset\Omega$.
	
	Taking into account  Lemma~\ref{lem:h-bfp}, we see that for every $\epsilon>0$ there exists $R_\epsilon\in \R$ so that
	\[
	U\cap\{w\in \C: \Re w<R_\epsilon\}\subset \{w\in \C: |\Im w-a_-(\sigma_1)|<\epsilon\}.
	\]
	Hence, if $\Im w\neq a_-(\sigma_1)$, it follows that there exists $t_0>0$ such that $w-t_0\in \Gamma\cap\C\subset\Omega$, which is a contradiction. Therefore, $\Im w=a_-(\sigma_1)$. In particular $\psi(a_-(\sigma_1))>-\infty$. Also, since $\Re w+ i a_-(\sigma_1)-t\in U$ for all $t>0$, we have that  either $\Im \gamma_1(t)>a_-(\sigma_1)$ and $\Im \gamma_2(t)<a_-(\sigma_1)$ or $\Im \gamma_1(t)<a_-(\sigma_1)$ and $\Im \gamma_2(t)>a_-(\sigma_1)$, for $t$ close to $1$. In both cases, $\lim_{y\to a_-(\sigma_1)}\psi(y)=-\infty<\psi(a_-(\sigma_1))$. Thus, by Theorem~\ref{Thm:Cantor-no-approx} we have $\hbox{Int}(\overline\Omega)\neq \Omega$, contradicting our hypothesis. Hence the map is injective.
	
	Finally, we prove that $\mathcal B_N\ni\sigma \mapsto a_-(\sigma)\in I_N$ is surjective. Let $a\in I_N$. We assume $\lim_{y\to a^+}\psi(y)=-\infty$ (the other case is similar). Therefore, for every $n=1,2,\dots$ there exists $\epsilon_n>0$ such that  $\psi(y)<-n$ for all $y\in I$ such that
	$0<y-a<\epsilon_n$.
	It follows that there exists a  continuous injective curve
	$\theta:[1,+\infty)\to \Omega$ such that $\lim_{t\to +\infty}\Re \theta(t)=-\infty$,
	$\Im \theta$ strictly decreases
	and $\lim_{t\to +\infty}\Im \theta(t)=a$. Hence (see, {\sl e.g.}, \cite[Prop. 3.3.3]{BCDbook}), there exists $\sigma\in\partial\D$ such that $\lim_{t\to+\infty}h^{-1}(\theta(t))=\sigma$ and $h^\ast(\sigma)=\infty$. Therefore, $\sigma$ is a boundary fixed point of $(\varphi_t)_{t\geq 0}$.
	
	Now, let us prove that $\sigma\neq \tau$. Indeed, if this were the case, we can define the injective continuous curve $\eta:(-\infty,+\infty)\to \Omega$ by setting $\eta(t):=\theta(t)$ for $t\geq 1$ and $\eta(t):=\theta(1)-(t-1)$ for $t<1$. By construction, and since we are assuming $\sigma=\tau$, it follows that $L:=\overline{h^{-1}(\theta((-\infty,+\infty))}$ is a Jordan curve such that $L\setminus\{\tau\}\subset\D$ and $L\cap\partial \D=\{\tau\}$. Let $U\subset \D$ be the bounded connected component of $\C\setminus L$.
	Then $h\big|_U$ extends to a continuous function from $\overline U$ to $\C_\infty$.
	Thus $h(U)\subset\Omega$ is one of the connected components of $\C\setminus \overline{\eta(-\infty,+\infty)}$. But then $a$ belongs to the closure of one open unbounded connected component of $\psi^{-1}(-\infty)$, against the hypothesis $a\in I_N$. Thus, $\sigma\neq \tau$.
	
	Finally,
	$\lim_{t\to +\infty}\Im h(h^{-1}(\theta(t)))=a$, and $\lim_{t\to+\infty}h^{-1}(\theta(t))=\sigma$. Therefore,
	\[
	\liminf_{z\to\sigma}\Im h(z)\leq a\leq \limsup_{z\to\sigma}\Im h(z).
	\]
	Thus, by Lemma~\ref{lem:h-bfp} and since $a\in I_N$, we have $\lim_{z\to\sigma}\Im h(z)=a$.
	This implies that $\sigma\in\mathcal B\mathcal F_N$, $a=a_-(\sigma)$, and we are done.
\end{proof}


\subsection{Unbounded discontinuities}\label{Subsect:bdd-disc}

\begin{definition}
Let $h$ be the Koenigs  function of a non-elliptic semigroup in $\D$. We say that $\sigma\in \mathcal S(h)$ is {\sl bounded}, provided there exists $R\in\R$ such that $\{w\in\C: \Re w<R\}\cap I(\hat{h}(\underline{\sigma}))=\emptyset$. If $\sigma\in \mathcal S(h)$ is not bounded, we say it is {\sl unbounded}. We let $\mathcal S_u(h)$ denote the set of all unbounded discontinuities.
\end{definition}

Except for the Denjoy-Wolff point, unbounded discontinuities can be deduced by analytic properties of the semigroup:

\begin{proposition}
Let $(\varphi_t)_{t\geq 0}$ be a non-elliptic semigroup in $\D$ with Denjoy-Wolff point $\tau\in\partial\D$. Let $h$ be its Koenigs function, and $\Omega=h(\D)$. Let $\sigma\in\partial\D\setminus\{\tau\}$. Then the following are equivalent:
\begin{enumerate}
\item   $\lim_{t\to+\infty}\varphi_t^\ast(\sigma)=\tau$, and for every $t_0>0$ there exists a sequence $\{w_n\}\subset\D$ converging to $\sigma$ such that $\lim_{n\to\infty}\varphi_{t_0}(w_n)=\sigma$.
\item  $\sigma\in\mathcal S_u(h)$.
\end{enumerate}
\end{proposition}
\begin{proof}
(1) implies (2). It is clear that $\sigma$ is not a boundary fixed point.
Moreover, $\varphi_t$ does not extend  continuously at $\sigma$  for $t>0$ by hypothesis. Hence, as we already remarked, it follows that $h$ is not continuous at $\sigma$, thus $\sigma\in\mathcal S(h)$. The cluster set of $h$ at $\sigma$ is the impression of $\hat{h}(\underline{\sigma})$---which is either a segment or a half-line parallel to the real axis (see Lemma~\ref{Lem:discont-Koenigs}).  We have to prove that $\sigma$ is an unbounded discontinuity, that is, $I(\hat{h}(\underline{\sigma}))$ is a half-line.

 Suppose by contradiction that $K:=I(\hat{h}(\underline{\sigma}))=\{w\in\C: a\leq \Re w\leq b, \Im w=y\}$ for some $-\infty< a<b=\Re h^\ast(\sigma)<+\infty$. Fix $t>b-a$. Recall that   $\varphi_t(z)=h^{-1}(h(z)+t)$ for all $z\in\D$. Let $\{w_n\}\subset\D$ be the sequence converging to $\sigma$ such that $\lim_{n\to \infty}\varphi_t(w_n)=\sigma$ given by hypothesis (1). Up to subsequences, we can assume that $\{h(w_n)\}$ converges to some $p\in K$. Let $\xi_n:=h(w_n)+t$. Hence, $\{\xi_n\}$ converges to $p+t$. Since the identity map $\sf{id}_\D$ extends as a homeomorphism between $\widehat{\D}$ (the Carath\'eodory closure of $\D$) and $\overline{\D}$, it follows that $\{\varphi_t(w_n)\}$ converges to $\sigma$ if and only if $\{h^{-1}(\xi_n)\}$ converges to $\underline{\sigma}\in\partial_C\D$ in the Carath\'eodory topology. Moreover, taking into account that $\hat{h}$ is a homeomorphism between $\widehat{\D}$ and $\widehat{\Omega}$, it follows that $\{h^{-1}(\xi_n)\}$ converges to $\underline{\sigma}\in\partial_C\D$ if and only if $\{\xi_n\}$ converges to $\hat{h}(\underline{\sigma})$ in the Carath\`eodory topology. Let $(C_n)$, $n\geq 0$ be a null-chain representing $\hat{h}(\underline{\sigma})$ and let $V_n$, $n\geq 1$ denotes the interior part of $C_n$. Fix $s\in (b,t)$. Since $\bigcap_{n\geq 1} \overline{V_n}=K$ and $V_{n+1}\subset V_n$ it follows that there exists $m$ such that $\overline{V_m}\subset \{w\in\C: \Re w<s\}$. But this implies that $\{\xi_n\}$ is not eventually contained in $V_m$, and hence (see, {\sl e.g.}, \cite[Remark~4.2.2]{BCDbook}), $\{\xi_n\}$ does not converge to $\hat{h}(\underline{\sigma})$ in the Carath\'eodory topology, contradiction. Therefore (1) implies (2).

(2) implies (1). Recall that, by the so-called ``boundary Denjoy-Wolff Theorem'' (see, {\sl e.g.}, \cite[Theorem~14.1.1]{BCDbook}), for every $\sigma\in\partial \D$ which is not a boundary fixed point of $(\varphi_t)_{t\geq 0}$, the curve $[0,+\infty)\ni t\mapsto \phi_t^\ast(\sigma)$ is a continuous injective curve converging to $\tau$.  Thus $\lim_{t\to+\infty}\varphi_t^\ast(\sigma)=\tau$. Finally, since $I(\hat{h}(\underline{\sigma}))$ is a half-line,  using again that $\varphi_t(z)=h^{-1}(h(z)+t)$ for all $z\in\D$ and $t\geq 0$, that $I(\hat{h}(\underline{\sigma}))$ is the cluster set of $h$ at $\sigma$ and a similar argument as before, we have the conclusion.
\end{proof}

Now we  relate unbounded discontinuities to the defining function of the Koenigs domain.

\begin{proposition}\label{Prop:char-bdd-dis}
Let $(\varphi_t)_{t\geq 0}$ be a non-elliptic semigroup in $\D$ with Denjoy-Wolff point $\tau\in\partial\D$. Let $h$ be its Koenigs function, $\Omega=h(\D)$  the associated Koenigs  domain . Let $\psi:I\to \R$ be the defining function of $\Omega$. Let $D_U$ be the set of points $y_0\in \overline{I}$ such that either
\[
\liminf_{y\to y_0^-}\psi(y)=-\infty, \quad -\infty<\limsup_{y\to y_0^-}\psi(y)<+\infty,
\]
or
\[
\liminf_{y\to y_0^+}\psi(y)=-\infty, \quad -\infty<\limsup_{y\to y_0^+}\psi(y)<+\infty.
\]
Then the map  $\mathcal U:\mathcal S_u(h)\setminus\{\tau\}\to D_U$ defined by
$\mathcal U(\sigma)= \Im h^\ast(\sigma)$ is surjective. Moreover, for each $y_0\in I\cap D_U$ the fibre $\mathcal U^{-1}(y_0)$   contains at most two points, while, if $y_0\in\partial I\cap D_U$ the fibre  $\mathcal U^{-1}(y_0)$ contains only one point.

Finally,  assume $\hbox{Int}(\overline\Omega)=\Omega$ and let $y_0\in I\cap D_U$. Then, there exist $\sigma_1\neq\sigma_2$ such that $\mathcal U^{-1}(y_0)=\{\sigma_1, \sigma_2\}$ if and only if  either $\sigma_1\in\mathcal S_u(h)\cap(\mathcal S^-(h)\setminus \mathcal S^+(h))$ and  $\sigma_2\in\mathcal S_u(h)\cap (\mathcal  S^+(h)\setminus \mathcal S^-(h))$ or $\sigma_1\in\mathcal S_u(h)\cap (\mathcal  S^+(h)\setminus \mathcal S^-(h))$ and  $\sigma_2\in \mathcal S_u(h)\cap(\mathcal  S^-(h)\setminus \mathcal S^+(h))$  and $\sigma_1,\sigma_2$ are the end points of a non-exceptional maximal contact arc of $(\varphi_t)_{t\geq 0}$.
\end{proposition}
\begin{proof}
Given $\sigma\in\mathcal S_u(h)$, it follows easily from
Lemma~\ref{Lem:discont-Koenigs-prel}
that $\Im h^\ast(\sigma)\in D_U$. Hence, the map $\mathcal U$ is well defined.

In order to see that  $\mathcal U$  is surjective, take $y_0\in D_U$. We can assume without loss of generality that $\liminf_{y\to y_0^-}\psi(y)=-\infty$ and $-\infty<M:=\limsup_{y\to y_0^-}\psi(y)<+\infty$. By definition of $\psi$  and since $\psi$ is upper semicontinuous, it is clear from the reverse implication of Lemma~\ref{Lem:discont-Koenigs}  that $L_{-\infty,M; y_0}$
is the impression of a prime end corresponding to a point $\sigma\in\mathcal S_u(h)\setminus\{\tau\}$ such that $\mathcal U(\sigma)=y_0$.

By Proposition~\ref{Lem:two-disc}, the fibre $\mathcal U^{-1}(y_0)$ contains at most two points for every $y_0\in D_U$ and only one if $y_0\in \partial I\cap D_U$,

Assume now $\hbox{Int}(\overline\Omega)=\Omega$. By Theorem~\ref{Thm:Cantor-no-approx},  every two non-exceptional  maximal contact arcs of $(\varphi_t)_{t\geq 0}$ have different final points. Thus, by Proposition~\ref{Lem:two-disc}, if $\sigma_1, \sigma_2\in\mathcal S_u(h)\setminus\{\tau\}$ are such that $\sigma_1\neq\sigma_2$ and $y_0:=\mathcal U(\sigma_1)=\mathcal U(\sigma_1)$ then $y_0\in I$ and, either $\sigma_1\in\mathcal S^+(h)\setminus \mathcal S^-(h)$ and $\sigma_2\in\mathcal S^-(h)\setminus \mathcal S^+(h)$ or $\sigma_1\in\mathcal S^-(h)\setminus \mathcal S^+(h)$ and $\sigma_2\in\mathcal S^+(h)\setminus \mathcal S^-(h)$. Moreover,
$\Re h^\ast(\sigma_1)\neq\Re h^\ast(\sigma_2)$.

We assume $x_1:=\Re h^\ast(\sigma_1)<x_2:=\Re h^\ast(\sigma_2)$ and $\sigma_1\in \mathcal S^-(h)$, $\sigma_2\in \mathcal S^+(h)$ (the other cases are similar).

We need to show that there exists a maximal contact arc for $(\varphi_t)_{t\geq 0}$ starting at $\sigma_1$ and ending at $\sigma_2$. The proof  is very similar to the one for the proof of the existence of the two maximal contact arcs in Proposition~\ref{Lem:two-disc}, so we only sketch it.

Since $\sigma_1\in \mathcal S^-(h)$, using (1.b) of Lemma~\ref{Lem:discont-Koenigs}, by Schwarz reflection principle, $h^{-1}$ restricted to $\Omega\cap \{w\in\C:\Im w<y_0\}$ extends holomorphically through $B:=\{w\in\C: x_1<\Re w<x_2, \Im w=y_0\}$. Hence, there exists an open arc $A\subset \partial \D$ such that $h$ extends holomorphically through $A$ and $h(A)=B$. By Lemma~\ref{Lem:cara-contact}, $A$ is a contact arc.

Since $\sigma_1\in \mathcal S^-(h)$ and $\sigma_2\in \mathcal S^+(h)$, properties (1.a) and (1.b) of Lemma~\ref{Lem:cara-contact} show that for every $x<x_1$ the point $x+iy_0$ is not accessible with injective continuous curves in $\Omega$ and thus cannot belong to  the image of any contact arc. Since $x_1+iy_0\in \overline{B}$, it follows then that the starting point $\xi_1$ of $A$ satisfies $h^\ast(\xi_1)=h^\ast(\sigma_1)$.

Now, arguing exactly as in the proof of Proposition~\ref{Lem:two-disc} one can show that $\xi_1=\sigma_1$.

In order to show that $\sigma_2$ is the final point of $A$, it is enough to prove that $\psi(y_0)=x_2$.  Indeed, by construction, $\psi(y_0)\geq x_2$. Since $\sigma_1\in \mathcal S^-(h)$ and $\sigma_2\in \mathcal S^+(h)$ (here we use (2.a) and (2.b) of Lemma~\ref{Lem:discont-Koenigs}), for every $\epsilon>0$ we can find  $\delta>0$ such that $\liminf_{t\to y}\psi(y)\leq x_2+\epsilon$ for all $|y-y_0|<\delta$. This clearly implies that $\tilde\psi(y_0)\leq x_2$ and hence, by Theorem~\ref{Thm:Cantor-no-approx}, $\psi(y_0)=\tilde\psi(y_0)=x_2$.

Finally, if $\sigma_1, \sigma_2$  are two unbounded discontinuities  which are the end points of a non-exceptional maximal contact arc $A$ of $(\varphi_t)_{t\geq 0}$ then clearly $\mathcal U(\sigma_1)=\mathcal U(\sigma_2)$ since $h^\ast(A)$ is contained in a line parallel to the real axis.
\end{proof}

\subsection{Boundary fixed points and unbounded discontinuities}\label{sub:all}

Let $(\varphi_t)_{t\geq 0}$ be a non-elliptic semigroup in $\D$, with Koenigs' function $h$, Koenigs' domain $\Omega$ and Denjoy-Wolff point $\tau\in\partial\D$. Let $\psi:I\to\R$ be the defining function of $\Omega$. In the previous sections, we introduced the following sets:
\begin{itemize}
\item $\mathcal B\mathcal F_R\subset\partial\D$, the set of all boundary regular fixed points of $(\varphi_t)_{t\geq 0}$ different from $\tau$,
\item $\mathcal B\mathcal F_N\subset\partial\D$, the set of all boundary super-repelling fixed points of $(\varphi_t)_{t\geq 0}$,
\item $\mathcal S^+(h), \mathcal S^-(h)\subset\partial\D\setminus\{\tau\}$, the sets of ``upper''/``lower'' discontinuities,
\item $\mathcal S_u(h)\subset\partial\D$, the set of all unbounded discontinuities,
\item $I_R\subset I$, the set of all bounded open connected components of $\psi^{-1}(-\infty)$,
\item $J_\infty\subset I$, the union of all open connected components of $\psi^{-1}(-\infty)$,
\item $I_N\subset\overline{I}$, the set of all points $y_0\in \overline{I}\setminus \overline{J_\infty}$ such that either $\lim_{y\to y_0^+}\psi(y)=-\infty$ or  $\lim_{y\to y_0^-}\psi(y)=-\infty$ (or both),
\item $D_U\subset\overline{I}$, the set of all points $y_0\in\overline{I}$ such that either
\[
\liminf_{y\to y_0^+}\psi(y)=-\infty<\limsup_{y\to y_0^+}\psi(y)<+\infty,
\]
or
\[
\liminf_{y\to y_0^-}\psi(y)=-\infty<\limsup_{y\to y_0^-}\psi(y)<+\infty.
\]
\end{itemize}

Assuming $\hbox{Int}(\overline\Omega)=\Omega$, we proved that there is a one-to-one correspondence between $\mathcal B\mathcal F_R$ and $I_R$; a one-to-one correspondence between $\mathcal B\mathcal F_N$ and $I_N$; and a surjective map from $\mathcal S_u(h)$ to $D_U$, so that, if the fibre above $y_0$ contains two points $\sigma_1,\sigma_2$, then $y_0\in I$ and there is a maximal contact arc starting from $\sigma_1$ and ending at $\sigma_2$.

 Let us now explore what happens at those points of intersection between $D_U$ and $I_N$ or those points of $D_U$ which belong to the closure of an element of $I_R$.

Note that, by definition, if $(a,b)\in I_R$ then $x\not\in D_U$ for every $x\in (a,b)$. However, it might happen that $a$ or $b$ belongs to $D_U$.

\begin{proposition}\label{Prop:fixed-vs-unbd}
Let $(\varphi_t)_{t\geq 0}$ be a non-elliptic semigroup in $\D$ with Denjoy-Wolff point $\tau\in\partial\D$. Let $h$ be its Koenigs function, $\Omega=h(\D)$  the associated Koenigs  domain  and assume that $\hbox{Int}(\overline\Omega)=\Omega$. Let $\psi:I\to \R$ be the defining function of $\Omega$.
\begin{enumerate}
\item $y_0\in D_U\cap I_N$ if and only  there exist a unique $\sigma\in \mathcal B\mathcal F_N$ and a unique $\xi\in \mathcal S_u(h)\setminus\{\tau\}$ such that $\Im h^\ast(\sigma)=\Im h^\ast(\xi)=y_0$.
\item  Let $(y_0,b)\in I_R$. Then $y_0\in D_U$ if and only if there exist a unique $\sigma\in \mathcal B\mathcal F_R$ and a unique $\xi\in \mathcal S_u(h)\cap \mathcal S^-(h)$ such that $\liminf_{z\to \sigma}\Im h^\ast(z)=\Im h^\ast(\xi)=y_0$.
\item  Let $(a,y_0)\in I_R$. Then $y_0\in D_U$ if and only if there exist a unique $\sigma\in \mathcal B\mathcal F_R$ and a unique $\xi\in \mathcal S_u(h)\cap \mathcal S^+(h)$  such that $\limsup_{z\to \sigma}\Im h^\ast(z)=\Im h^\ast(\xi)=y_0$.
\end{enumerate}
Moreover, in the hypothesis of (1), (2) or (3), $y_0\in I$ and there exists a maximal contact arc $A$ starting from $\sigma$ and ending at $\xi$ such that $\Im h^\ast(A)=y_0$.
\end{proposition}
\begin{proof}
(1) Assume $y_0\in D_U\cap I_N$. Then, by Proposition~\ref{Prop:char-bfp} there exists a unique $\sigma\in \mathcal B\mathcal F_N$ such that
$\Im h^\ast(\sigma)=\lim_{z\to \sigma}\Im h(z)=y_0$. By Proposition~\ref{Prop:char-bdd-dis}, either there exist a unique $\xi\in \mathcal S_u(h)\setminus\{\tau\}$ such that $\Im h^\ast(\xi)=y_0$ or there are $\xi^+,\xi^-\in \mathcal S_u(h)\setminus\{\tau\}$ such that $\Im h^\ast(\xi^\pm)=y_0$ and  $\xi^+\in S^+(h)$,   $\xi^-\in S^-(h)$. By (1.a) and (1.b) of Lemma~\ref{Lem:discont-Koenigs}, the existence of such $\xi^+$ and $\xi^-$
implies that there exist two sequences $\{z_n^+\}$ and $\{z_n^-\}$ such that $\Im z_n^+>y_0$, $\Im z_n^+<y_0$, $\{z_n^\pm\}$ converges to $h^\ast(\xi^\pm)$ and $z_n^\pm-t\in\C\setminus\Omega$ for all $t\geq 0$. Since  the continuous curve $[0,1)\ni t\mapsto h(t\sigma)$
in $\Omega$
has the property that $\lim_{t\to 1^-}\Re h(t\sigma)=-\infty$ and $\lim_{t\to 1^-}\Im h(t\sigma)=y_0$, this immediately leads to a contradiction.

Therefore, there exists a unique $\xi\in \mathcal S_u(h)\setminus\{\tau\}$ such that $\Im h^\ast(\xi)=y_0$. By the same argument as before, it follows that either $\xi\in\mathcal S^-(h)\setminus \mathcal S^+(h)$ and $\Im h(t\sigma)>y_0$ for all $t$ close to $1$
or $\xi\in\mathcal S^+(h)\setminus \mathcal S^-(h)$ and $\Im h(t\sigma)<y_0$ for all $t$ close to $1$. In both cases, $y_0\in I$.

The converse implication of (1) is straightforward from Proposition~\ref{Prop:char-bfp} and Proposition~\ref{Prop:char-bdd-dis}.

(2) If $(y_0,b)\in I_R$ then $\psi(y)=-\infty$ for all $y\in (y_0,b)$. Hence, one can argue as in the proof of (1) showing the existence of a unique $\sigma\in\mathcal B\mathcal F_R$ and  a unique $\xi\in S^-(h)$ such that $\liminf_{z\to \sigma}\Im h^\ast(z)=\Im h^\ast(\xi)=y_0$, and that $y_0\in I$. The reverse implication follows again by Proposition~\ref{Prop:char-bfp} and Proposition~\ref{Prop:char-bdd-dis}.

(3) is similar to (2).

 We are left to show that there exists a maximal contact arc starting from $\sigma$ and ending at $\xi$. Since the argument is quite similar to the one for proving the existence of the two maximal contact arcs in Proposition~\ref{Lem:two-disc}, we omit it and leave details to the reader.
\end{proof}

\subsection{Discontinuities and the Denjoy-Wolff point}

For the Denjoy-Wolff point, recalling the notation introduced in \eqref{Eq:L-con-subindex} and setting $L_y:=\{w\in\C: \Im w=y\}$ for $y\in\R$, we have (see \cite[Thm.11.1.4]{BCDbook}),

\begin{lemma}\label{Lem:DW-discont}
Let $h$ be the Koenigs  function of a non-elliptic semigroup in $\D$ with Denjoy-Wolff point $\tau$ and let $\Omega=h(\D)$ the Koenigs  domain. If $\tau\in  \mathcal S(h)$  then
\begin{enumerate}
\item  either there exists $y\in \R$ such that  $\Omega\subset \{w\in \C: \Im w>y\}$ or $\Omega\subset \{w\in \C: \Im w<y\}$ and $I(\hat{h}(\underline{\tau}))=\overline{L_{a,+\infty;y}}^\infty$,
where $a\in [-\infty,\infty)$,
\item or, there exist real numbers $y_1<y_2$ such that  $\Omega\subset \{w\in \C:y_1< \Im w<y_2\}$  and
$ I(\hat{h}(\underline{\tau}))	
=\overline{L_{a_1,+\infty;y_1}}^\infty\cup\overline{L_{a_2,+\infty;y_2}}^\infty$,
where $a_1,a_2\in [-\infty,\infty)$.
\end{enumerate}
Moreover, in the first case, there are no $\sigma\in\partial\D$ such that $h^\ast(\sigma)\in L_y$, while in the second case there are no $\sigma\in\partial\D$ such that $h^\ast(\sigma)\in L_{y_1}\cup L_{y_2}$.
\end{lemma}

Such lemma suggests the following definition:

\begin{definition}
Let $(\varphi_t)_{t\geq 0}$ be a non-elliptic semigroup with Denjoy-Wolff point $\tau\in\partial \D$ and Koenigs function $h$. Suppose $\tau\in\mathcal S_u(h)$. If $I(\hat{h}(\underline{\tau}))= L_{y_1}\cup L_{y_2}\cup\{\infty\}$, for some $-\infty<y_1<y_2<+\infty$, we say that $\tau$ is a {\sl double} unbounded discontinuity.  In all other cases, we say that $\tau$ is a {\sl simple} unbounded discontinuity.
\end{definition}

\begin{remark}\label{rem:disco-tau-double-beahv-psi}
By Lemma~\ref{Lem:DW-discont} it follows that
if $\tau\in  \mathcal S(h)$ then the semigroup $(\varphi_t)_{t\geq 0}$ is either hyperbolic or parabolic with positive hyperbolic step.

In particular, $\tau\in\mathcal S_u(h)$ if and only if, given $\psi:(a_0,a_1)\to\R$ the defining function of the associated Koenigs domain, $-\infty\leq a_0<a_1\leq +\infty$, then, either $a_0>-\infty$ and
\[
\liminf_{y\to a_0^+}\psi(y)=-\infty, \quad \limsup_{y\to a_0^+}\psi(y)= +\infty,
\]
or, $a_1<+\infty$ and
\[
\liminf_{y\to a_1^-}\psi(y)=-\infty, \quad \limsup_{y\to a_1^-}\psi(y) +\infty,
\]
Moreover,  $\tau\in  \mathcal S(h)$ is a double unbounded discontinuity if and only if  $(\varphi_t)_{t\geq 0}$ is  hyperbolic (that is, $-\infty<a_0<a_1<+\infty$) and both the two previous conditions hold.
\end{remark}

 For future aims, we need the following proposition (recall the definition of $a_\pm(\sigma)$ from \eqref{Eq:defino-a-por-sigma}):

\begin{proposition}\label{Prop:tau-disc-unbd}
Let $(\varphi_t)_{t\geq 0}$ be a non-elliptic semigroup in $\D$ with Denjoy-Wolff point $\tau\in\partial\D$. Let $h$ be the Koenigs function and $\Omega=h(\D)$  the  Koenigs  domain. Assume that  the defining function $\psi$ of $\Omega$ is defined in $(a_0,a_1)$ for some $-\infty<a_0<a_1< +\infty$. Then the following are equivalent:
\begin{enumerate}
\item $\tau$ is a simple unbounded discontinuity,
\item one and only one of the following two conditions holds:
\begin{itemize}
\item[(2.a)]   $\liminf_{y\to a_0^+}\psi(y)=-\infty$,  $\limsup_{y\to a_0^+}\psi(y)= +\infty$ and, either  $\limsup_{y\to a_1^-}\psi(y)< +\infty$ or   $\liminf_{y\to a_1^-}\psi(y)> -\infty$   (or both).
\item[(2.b)]   $\liminf_{y\to a_1^-}\psi(y)=-\infty$,  $\limsup_{y\to a_1^-}\psi(y)= +\infty$ and, either  $\limsup_{y\to a_0^+}\psi(y)< +\infty$ or   $\liminf_{y\to a_0^+}\psi(y)> -\infty$   (or both).
\end{itemize}
\end{enumerate}
Moreover,  if (2.a) holds, then
\begin{itemize}
\item[(a1)] for every boundary fixed point $\sigma\in\partial\D\setminus\{\tau\}$ we have  $a_-(\sigma)>a_0$,
\item[(a2)] for every   $\xi\in\mathcal S_u(h)\setminus \{\tau\}$
we have $\Im h^\ast(\xi)\neq a_0$.
\end{itemize}
While,
 if (2.b) holds, then
\begin{itemize}
\item[(b1)] for every boundary fixed point $\sigma\in\partial\D\setminus\{\tau\}$ we have
$a_+(\sigma)<a_1$,
\item[(b2)] for every  $\xi\in\mathcal S_u(h)\setminus \{\tau\}$
we have $\Im h^\ast(\xi)\neq a_1$.
\end{itemize}
\end{proposition}

\begin{proof}
From Lemma~\ref{Lem:DW-discont} it follows immediately that  (1) and (2) are equivalent. Moreover, by the same token, if (2.a) ({\sl respectively (2.b)} holds then (a2) ({\sl respect.} (b2)) holds.

As for (a1),
if $\sigma$ is a boundary  fixed point different from $\tau$, and $a_0=a_-(\sigma)$, then for every $x\in\R$ there exists $\delta>0$ such that $\{w\in \C: \Re w>x-\delta, a_0<\Im w<a_0+\delta\}\subset\Omega$. This is immediate if $a_+(\sigma)>a_0$, and it follows from the fact that $\Omega$ is starlike at infinity and the continuous injective curve $[0,1)\ni t\mapsto h(t \sigma)$ has the property that $\lim_{t\to 1^-}\Re h(t\sigma)=-\infty$ and $\lim_{t\to 1^-}\Im h(t\sigma)=a_0$ in case $a_+(\sigma)=a_0$. But this implies that for every $x\in\R$ there exists $\zeta\in\partial\D$ such that $h^\ast(\zeta)=x+ia_0$ (the point $\zeta$ being the limit of $h^{-1}(x+i(a_0+s))$ for $s\to 0^+$), against Lemma~\ref{Lem:DW-discont}.
The same argument applies to deduce (b1) from (2.b).
\end{proof}

For our aims, we need also the following (see Proposition~\ref{Prop:char-bdd-dis} for the definition of the map $\mathcal U$):

\begin{lemma}\label{Prop:exceptional-maximal-disc}
Let $(\varphi_t)_{t\geq 0}$ be a non-elliptic semigroup in $\D$ with Denjoy-Wolff point $\tau\in\partial\D$. Let $h$ be the Koenigs function and $\Omega=h(\D)$  the  Koenigs  domain. Assume that  the defining function $\psi$ of $\Omega$ is defined in $(a_0,a_1)$ for some $-\infty<a_0<a_1\leq +\infty$. Then the following are equivalent:
\begin{enumerate}
\item There exists a maximal contact arc for $(\varphi_t)_{t\geq 0}$ with starting point $\tau$ and final point  $\sigma\in\mathcal S_u(h)$.
\item $a_1=+\infty$ and there exists
$a\in (a_0,+\infty)$ such that
$\psi(y)=-\infty$ for all $y>a$ and
\[
-\infty=\liminf_{y\to a^-}\psi(y)<\limsup_{y\to a^-}\psi(y)= \psi(a).
\]
\end{enumerate}
Moreover, if either (1) or (2)---and hence both---holds, then $\mathcal U(\sigma)=a$, $\mathcal U^{-1}(a)=\{\sigma\}$ and $h^\ast(\sigma)=\psi(a)+ia$.
\end{lemma}

\begin{proof}
(1) implies (2). Let $A\subset\partial\D$ be a maximal contact arc starting at $\tau$ and ending at $\sigma\in\mathcal S_u(h)$. If $\sigma=\tau$, then $(\varphi_t)_{t\geq 0}$ is a parabolic group (see \cite[Proposition~14.2.7]{BCDbook}), in particular, $\sigma\not\in\mathcal S(h)$. Hence, $\sigma\neq\tau$ and, since it is not a boundary fixed point of $(\varphi_t)_{t\geq 0}$, it follows that $h^\ast(\sigma)\in\C$. Since $x_0(A)=\tau$, a fixed point of $(\varphi_t)_{t\geq 0}$, by~Lemma~\ref{Lem:cara-contact} and \eqref{Eq:where-go-int-fin-max-arc}, there exists $a\geq a_0$ and $b\in (-\infty,+\infty]$ such that $h(A)=\{w\in \C: \Im w=a, \Re w<b\}$ and $\lim_{A\ni \xi\to\sigma}h(\xi)=h^\ast(\sigma)$. Thus $a=\Im h^\ast(\sigma)$ and $b=\Re h^\ast(\sigma)\in \R$. In particular, $\mathcal U(\sigma)=a$.

Moreover, we already recalled (see the discussion on maximal contact arcs) that
$h^\ast(\sigma)+t\in\Omega$ for all $t>0$.

Let $\Gamma_1$ be the closure of the image of the continuous injective curve given by $(0,+\infty)\ni t\mapsto h^{-1}(h^\ast(\sigma)+t)$. Since $\varphi_t(w)=h^{-1}(h(w)+t)$ for all $w\in\D$ and $t\geq 0$, it follows that $\Gamma_1$ is a Jordan arc with end points $\sigma$ and $\tau$ and $\Gamma_1\setminus\{\sigma,\tau\}\subset\D$. Hence, $\Gamma:=\overline{A}\cup\Gamma_1$ is a Jordan curve. Let $V\subset\D$ be the bounded connected component of $\C\setminus\Gamma$. Note that $V$ is a Jordan domain and hence $h$ extends as a homeomorphism to $\partial V$. Since by construction $h(\Gamma\setminus\{\sigma,\tau\})=\{w\in\C:\Im w=a\}\setminus\{b+ia\}$, it follows that $h(\partial V)=\{w\in\C:\Im w=a\}\cup\{\infty\}$. Thus, $h(V)=\{w\in\C: \Im w<a\}$ or $h(V)=\{w\in\C: \Im w>a\}$. But $h(V)\subset\Omega$ and $\Omega\subset \{w\in \C: \Im w>a_0\}$, therefore, $\{w\in\C: \Im w>a\}=h(V)\subset\Omega$. This implies that $\psi$ is defined on $(a_0,+\infty)$ and $\psi(y)=-\infty$ for all $y>a$.

This also implies that $\sigma\in\mathcal S^-(h)$ and there cannot be $\xi\in \mathcal S^+(h)$ such that $\Im h^\ast(\xi)=a$. Therefore, Proposition~\ref{Prop:char-bdd-dis} implies that $\mathcal U^{-1}(a)=\{\sigma\}$. Finally, Lemma~\ref{Lem:discont-Koenigs} shows that
\[
-\infty<\Re h^\ast(\sigma)=\limsup_{y\to a^-}\psi(y)\leq \psi(a).
\]
Since $\Re h^\ast(\sigma)+t\in\Omega$ for all $t>0$, we have actually $\limsup_{y\to a^-}\psi(y)=\psi(a)$ and hence $h^\ast(\sigma)=\psi(a)+ia$.

Thus (2) and the last part of the proposition are proved.

(2) implies (1). By Proposition~\ref{Prop:char-bdd-dis}, and arguing as before one can see that there exists a unique $\sigma\in \mathcal S_u(h)\setminus\{\tau\}$ such that $h^\ast(\sigma)=\psi(a)+ia$.

Next, the hypothesis on $\psi$ implies that $H:=\{w\in\C: \Im w>a\}\subset\Omega$, $\{w\in\C: \Im w=a, \Re w\leq\psi(a)\}\cap\Omega=\emptyset$ and $\{w\in\C: \Im w=a, \Re w>\psi(a)\}\subset\Omega$. Since for every sequence $\{w_n\}\subset H$ converging to any point of $\{w\in\C: \Im w=a, \Re w\leq\psi(a)\}\cup\{\infty\}$ the cluster set of $h^{-1}(w_n)$ is contained in $\partial\D$, by Schwarz reflection principle it follows that $h^{-1}$ extends holomorphically through $\partial_\infty H$. Let $V:=h^{-1}(H)$.  Then  $V\subset\D$ is a Jordan domain and $\partial V$ is the union of an open arc $A\subset\partial\D$ and a Jordan arc $\Gamma$ with the same end points as $A$ and such that $\Gamma$ minus its end points is contained in $\D$. By construction, $h^\ast(A)=\{w\in\C: \Im w=a, \Re w<\psi(a)\}$. Therefore, $A$ is a contact arc for $(\varphi_t)_{t\geq 0}$ (see Lemma~\ref{Lem:cara-contact}). The order on $A$ is given by the direction of the flow $(h|_{\partial H})^{-1}(w+t)$, $t>0$, $w\in\C$ with $\Im w=a, \Re w<\psi(a)$. Hence, if $\zeta\in\partial A$ is such that $h^\ast(\zeta)=\psi(a)+ia$, it follows that $\zeta$ is the final point of $A$. Using an argument similar to the one in  the proof of Proposition~\ref{Lem:two-disc} one can easily show that $\zeta=\sigma$.

We are left to prove that $\tau$ is the initial point of $A$.
Denote by $\tilde{h}$ the homeomorphic extension of $h|_V$ to $\overline{V}$. By construction,  $\tilde{h}(\overline{V}\setminus\{x_0(A)\})=\overline{H}$. Since $\tilde{h}$ is a homeomorphism from $\overline{V}$ onto $\overline{H}^\infty$, it follows that $\tilde{h}(x_0(A))=\infty$. Thus $h^\ast(x_0(A))=\lim_{A\ni\xi\to x_0(A)}h(\xi)=\infty$.  Therefore, $x_0(A)=\lim_{t\to+\infty}h^{-1}(\gamma(t))$ for any continuous curve $\gamma:[0,+\infty)\to H$ such that $\lim_{t\to+\infty}\gamma(t)=\infty$ (see, {\sl e.g.}, \cite[Proposition~3.3.3]{BCDbook}).  Fix $w_0\in H$ and let $\gamma(t):=w_0+t$. Hence
\[
\lim_{t\to+\infty}h^{-1}(\gamma(t))=\lim_{t\to+\infty}\varphi_t(h^{-1}(w_0))=\tau,
\]
and therefore $\tau=x_0(A)$. Thus (1) holds.
\end{proof}

\section{Koenigs  domains $\Omega$ such that  $\hbox{Int}(\overline\Omega)=\Omega$  and $\C\setminus\overline{\Omega}$ has at most two connected components}\label{Sec:char2}

 In this section we characterize Koenigs  domains whose complementary in $\C$ is connected or doubly connected (and contained in a strip).

\begin{remark}\label{Rem:easy-contain-plane}
Suppose $\Omega$ is the Koenigs domain of a non-elliptic semigroup $(\varphi_t)_{t\geq 0}$. Let $\psi:I\to [-\infty,+\infty)$ be the defining function of $\Omega$. If $\Omega\subset H$, where $H$ is a half-plane not parallel to $\{w\in \C: \Im w>0\}$, then $\C\setminus\overline{\Omega}$ is connected.
This follows from the fact that for any $p\in \C\setminus\overline{\Omega}$ there exists $t_0\geq 0$ such that $p-t_0\in \C\setminus\overline{H}$ and $p-t\in \C\setminus\overline{\Omega}$ for all $t\geq 0$.

Moreover, it is clear that in such a case $(\varphi_t)_{t\geq 0}$ has neither boundary fixed points other than the Denjoy-Wolff point nor unbounded discontinuities. Also $\liminf_{y\to y_0}\psi(y)>-\infty$ for all $y_0\in I$.
\end{remark}

\begin{theorem}\label{Thm:connect-half-plane}
Let $(\varphi_t)_{t\geq 0}$ be a non-elliptic semigroup in $\D$ with Denjoy-Wolff point $\tau\in\partial\D$. Let $h$ be the Koenigs  function, $\Omega=h(\D)$ the Koenigs
domain  and let $\psi:I\to \R$ be the defining function of $\Omega$. Suppose $\hbox{Int}(\overline\Omega)=\Omega$ and $\Omega\subset \{w\in\C: \Im w>0\}$.  Let $I_\infty$ denotes the unbounded connected component of $\psi^{-1}(-\infty)$ (possibly empty).  Then the following are equivalent:
\begin{enumerate}
\item $\C\setminus \overline{\Omega}$ is connected.
\item $(\varphi_t)_{t\geq 0}$ has no boundary fixed points other than $\tau$ and $\mathcal S(h)$ contains only bounded discontinuities except, at most,  one unbounded discontinuity $\sigma\in\partial\D\setminus\{\tau\}$ which is the final point of a maximal contact arc starting at $\tau$.
\item  $\liminf_{y\to y_0}\psi(y)>-\infty$
for every $y_0\in \overline{I}\setminus \overline{I_\infty}$.
\end{enumerate}
\end{theorem}
\begin{proof}
(2) implies (3). Let $W\subset I$ be the set of points $y$ such that either $\liminf_{y\to y_0^+}\psi(y)=-\infty$ or $\liminf_{y\to y_0^-}\psi(y)=-\infty$. Using the subsets of $\overline{I}$ we introduced so far (see the beginning of Subsection~\ref{sub:all}), we have that
\[
W\setminus\overline{I_\infty}=(I_N\cap I) \cup I_R \cup (\mathcal D_U\cap I).
\]
Since, by hypothesis, $(\varphi_t)_{t\geq 0}$ has no boundary fixed points other than $\tau$, Proposition~\ref{Prop:char-bfp} implies that $I_R=I_N=\emptyset$.  Also, by hypothesis, $\mathcal S(h)$ contains no unbounded discontinuity except, at most, one unbounded discontinuity $\sigma\in\partial\D\setminus\{\tau\}$ which is the final point of a maximal contact arc starting at $\tau$. Therefore, Proposition~\ref{Prop:char-bdd-dis} (together with Lemma~\ref{Prop:exceptional-maximal-disc} in case of the existence of such a  $\sigma$) implies that $\mathcal D_U\cap I=\emptyset$. Thus, $W\setminus\overline{I_\infty}=\emptyset$, and (3) follows.

(3) implies (2).  Indeed, if either $I_\infty=\emptyset$ or
$I_\infty=[a,+\infty)$, where $a>0$,
it follows from Proposition~\ref{Prop:char-bfp} and Proposition~\ref{Prop:char-bdd-dis} that (2) holds. Similarly, if
$I_\infty=(a,+\infty)$ for some $a>0$ and $\liminf_{y\to a^-}\psi(y)>-\infty$.

Thus, we assume that (3) holds, that $I_\infty=(a,+\infty)$ for some $a>0$, $\psi(a)>-\infty$ and $\liminf_{y\to a^-}\psi(y)=-\infty$.
Notice that $\limsup_{y\to a^-}\psi(y)=\limsup_{y\to a}\psi(y)\le \psi(a)$.
If $M:=\limsup_{y\to a^-}\psi(y)<\psi(a)$,
then the interval $(M,\psi(a)] + ia$ is contained in $\hbox{Int}(\overline\Omega)\setminus\Omega$,
which contradicts our hypothesis. Therefore $\limsup_{y\to a^-}\psi(y)=\psi(a)$.
By Lemma~\ref{Prop:exceptional-maximal-disc}, (2) holds.

%

(1) implies (3). Arguing by contradiction we assume that there exists  $y_0\in \overline{I}\setminus \overline{I_\infty}$ such that  $\liminf_{I\ni y\to y_0}\psi(y)=-\infty$.

By the same argument as in the proof of Lemma~\ref{Lem:meaning-connect-strip}, $L_{y_0}\subset\overline{\Omega}$.

Since by hypothesis $\{w\in\C: \Im w<0\}\subset\C\setminus\overline{\Omega}$ and we are assuming that $\C\setminus\overline{\Omega}$ is connected, it follows that $\{w\in\C: \Im w>y_0\}\subset\overline\Omega$. But $\{w\in\C: \Im w>y_0\}$ is open, thus $\{w\in\C: \Im w>y_0\}\subset\hbox{Int}(\overline{\Omega})=\Omega$, by hypothesis. Therefore, $(y_0,+\infty)\subset I$ and $\psi(y)=-\infty$ for all $y>y_0$. In other words,  $y_0\in\overline{I_\infty}$, a contradiction. Thus (3) holds.

(3) implies (1). Notice that
$\Pi=\{w\in\C: \Im w<0\}$ is contained in $\C\setminus \overline{\Omega}$.
The condition (3) implies that $\inf_{y\in J} \psi(y) >-\infty$ for any compact interval $J\subset \overline I\setminus \overline {I_\infty}$.
It follows that any point in $\C\setminus \overline{\Omega}$ can be connected with a point in $\Pi$ in $\C\setminus \overline{\Omega}$ by a union of a horizontal
interval (going to the left) and a vertical one (going down). Therefore
$\C\setminus \overline{\Omega}$ is connected.

\end{proof}

As a direct corollary of Lemma~\ref{Prop:exceptional-maximal-disc} and Theorem~\ref{Thm:connect-half-plane}, we have:
\begin{corollary}\label{Cor:connex-in-strip}
Let $(\varphi_t)_{t\geq 0}$ be a non-elliptic semigroup in $\D$ with Denjoy-Wolff point $\tau\in\partial\D$. Let $h$ be the Koenigs  function, $\Omega=h(\D)$ the Koenigs  domain  and let $\psi:I\to \R$ be the defining function of $\Omega$. Suppose $\hbox{Int}(\overline\Omega)=\Omega$ and $\Omega\subset \mathbb S$, for some strip $\mathbb S$.   Then the following are equivalent:
\begin{enumerate}
\item $\C\setminus \overline{\Omega}$ is connected.
\item $(\varphi_t)_{t\geq 0}$ has no boundary fixed points other than $\tau$ and $\mathcal S(h)$ contains only bounded discontinuities.
\item  $\liminf_{I\ni y\to y_0}\psi(y)>-\infty$ for every $y_0\in \overline{I}$.
\end{enumerate}
\end{corollary}

Now we examine the case $\Omega$ is contained in a strip and $\C\setminus\overline{\Omega}$ has two connected components.

\begin{theorem}\label{Thm:connect-strip}
Let $(\varphi_t)_{t\geq 0}$ be a non-elliptic semigroup in $\D$ with Denjoy-Wolff point $\tau\in\partial\D$. Let $h$ be the Koenigs  function, $\Omega=h(\D)$ the Koenigs  domain  and let $\psi:I\to \R$ be the defining function of $\Omega$. Suppose $\hbox{Int}(\overline\Omega)=\Omega$ and $\Omega\subset \mathbb S$, for some strip $\mathbb S$.   Then the following are equivalent:
\begin{enumerate}
\item $\C\setminus \overline{\Omega}$ has two connected components,

\item there exist $a\leq b$ such that $[a,b]\subset \overline{I}$,  $\liminf_{y\to y_0}\psi(y)=-\infty$ for all $y_0\in [a,b]$ and $\liminf_{I\ni y\to y_1}\psi(y)>-\infty$ for all $y_1\in\overline{I}\setminus [a,b]$,

\item $(\varphi_t)_{t\geq 0}$ satisfies one, and only one, of the following:
\begin{itemize}
\item [(3.1)]  $(\varphi_t)_{t\geq 0}$ has  one and only one boundary fixed point $\sigma\in\partial\D\setminus\{\tau\}$ and $\mathcal S_u(h)=\emptyset$,
\item [(3.2)] $(\varphi_t)_{t\geq 0}$ has  one and only one boundary fixed point $\sigma\in\partial\D\setminus\{\tau\}$, $\mathcal S_u(h)=\{\xi\}$ with $\xi\neq \tau$, and there exists a maximal contact arc starting from $\sigma$ and ending at $\xi$,
\item [(3.3)] $(\varphi_t)_{t\geq 0}$ has  one and only one boundary fixed point $\sigma\in\partial\D\setminus\{\tau\}$ which is a boundary regular fixed point, $\mathcal S_u(h)=\{\xi_1, \xi_2\}$ with $\xi_j\neq \tau$ and $\xi_1\ne \xi_2$,
and  for $j=1,2$,
there exists a maximal contact arc starting from $\sigma$ and ending at $\xi_j$,
\item[(3.4)]  $(\varphi_t)_{t\geq 0}$ has no boundary fixed points other than $\tau$ and $\mathcal S_u(h)=\{\xi\}$, with $\xi\neq \tau$,
\item[(3.5)]  $(\varphi_t)_{t\geq 0}$ has no boundary fixed points other than $\tau$ and $\mathcal S_u(h)=\{\xi_1, \xi_2\}$, with $\xi_j\neq \tau$, $j=1,2$
	and $\xi_1\ne \xi_2$
and there exists a maximal contact arc  starting from $\xi_1$ and ending at $\xi_2$,
\item[(3.6)]  $(\varphi_t)_{t\geq 0}$ has no boundary fixed points other than $\tau$, $\mathcal S_u(h)=\{\tau\}$ and $\tau$ is a simple unbounded discontinuity.
\end{itemize}
\end{enumerate}
\end{theorem}

\begin{proof}
(1) is equivalent to (2) by  Lemma~\ref{Lem:meaning-connect-strip}.

(2) implies (3). We consider three cases:

{\sl Case 1: $a<b$}. Since by hypothesis $\hbox{Int}(\overline\Omega)=\Omega$  and $\liminf_{y\to y_0}\psi(y)=-\infty$ for all $y_0\in [a,b]$, it follows from Theorem~\ref{Thm:Cantor-no-approx}  that $\psi(y)=-\infty$ for all $y\in (a,b)$. By Proposition~\ref{Prop:char-bfp} there exists a unique boundary  fixed point $\sigma\neq \tau$ such $a_-(\sigma)=a$ and $a_+(\sigma)=b$ so that $\sigma$ is a boundary regular fixed point. Since $\liminf_{I\ni y\to y_1}\psi(y)>-\infty$ for all $y_1\in\overline{I}\setminus [a,b]$, the same Proposition~\ref{Prop:char-bfp} implies that there are no other boundary fixed points different from $\sigma$ and $\tau$. By Proposition~\ref{Prop:tau-disc-unbd} and since $\liminf_{I\ni y\to y_1}\psi(y)>-\infty$ for all $y_1\in\overline{I}\setminus [a,b]$, $\tau\not\in\mathcal S_u(h)$. Finally, Proposition~\ref{Prop:char-bdd-dis} and Proposition~\ref{Prop:fixed-vs-unbd} show at once that either (3.1) or (3.2) or (3.3) holds.

{\sl Case 2: $a=b\not\in\partial I$}. By Proposition~\ref{Prop:tau-disc-unbd}, $\tau\not\in \mathcal S_u(h)$.

Since $a\in I$ and $\psi$ is upper semicontinuous, it follows that  $\limsup_{y\to a}\psi(y)\leq \psi(a)<+\infty$.

If either $\lim_{y\to a^-}\psi(y)=-\infty$ or  $\lim_{y\to a^+}\psi(y)=-\infty$ and since $\liminf_{I\ni y\to y_1}\psi(y)>-\infty$ for all $y_1\in\overline{I}\setminus \{a\}$, it follows from Proposition~\ref{Prop:char-bfp} that there exists a unique boundary fixed point $\sigma$ different from $\tau$. Moreover, $\Im h^\ast(\sigma)=a$. Proposition~\ref{Prop:char-bdd-dis} and Proposition~\ref{Prop:fixed-vs-unbd} show at once that either (3.1) or (3.2) holds.

If  $\limsup_{y\to a^\pm}\psi(y)>-\infty$, then by Proposition~\ref{Prop:char-bfp}, there are no boundary fixed points other than $\tau$. Proposition~\ref{Prop:char-bdd-dis} shows that there exist one or two unbounded discontinuities,
and Proposition~\ref{Prop:fixed-vs-unbd} implies that either (3.4) or (3.5) holds.

{\sl Case 3: $a=b\in\partial I$}. We can assume that $I=(a, c)$ for some $c>a$.

If $\limsup_{y\to a^+}\psi(y)=+\infty$, it follows from Proposition~\ref{Prop:tau-disc-unbd}, Proposition~\ref{Prop:char-bfp} and Proposition~\ref{Prop:char-bdd-dis} that $\tau$ is a simple unbounded  discontinuity and there are no other boundary fixed points nor unbounded discontinuities, thus (3.6) holds.

If  $\limsup_{y\to a^+}\psi(y)<+\infty$, then $\tau\not\in\mathcal S_u(h)$ and we can argue as in Case 2 depending on whether $\lim_{y\to a^+}\psi(y)=-\infty$ or not, and we see that in this case  either (3.1)  or (3.4) holds. Note here that in this case, (3.2) and (3.5) cannot hold because  Proposition~\ref{Prop:fixed-vs-unbd} asserts that in such cases $a\in I$.

All possible cases are covered by the previous arguments, and hence (2) implies (3).

(3) implies (2). It follows at once examining the various cases and applying
Proposition~\ref{Prop:char-bfp}, Proposition~\ref{Prop:char-bdd-dis}, Proposition~\ref{Prop:fixed-vs-unbd} and Proposition~\ref{Prop:tau-disc-unbd}.
\end{proof}

\section{Proof of Theorem~\ref{Thm:main-hyp}, Theorem~\ref{Thm:main-para-pos} and Theorem~\ref{Thm:main-para-zero}}\label{Sec:proofs}

The proofs of Theorem~\ref{Thm:main-hyp}, Theorem~\ref{Thm:main-para-pos} and Theorem~\ref{Thm:main-para-zero} follow by gathering together the results we proved in the previous sections.

\begin{proof}[Proof of Theorem~\ref{Thm:main-hyp}]
The semigroup $(\varphi_t)_{t\geq 0}$ is hyperbolic if and only if $\Omega$ is contained in a strip $\mathbb S$. Hence, by Theorem~\ref{Thm:density-in-strip} and Corollary~\ref{Cor:complement}, (1) and (2) are equivalent. Then (2) is equivalent to (3) and (4) by Theorem~\ref{Thm:Cantor-no-approx}, Corollary~\ref{Cor:connex-in-strip} and Theorem~\ref{Thm:connect-strip}. The last assertion follows from (1) and Lemma~\ref{Lem:approx-weak-approx-p}.
\end{proof}

\begin{proof}[Proof of Theorem~\ref{Thm:main-para-pos}]
The semigroup $(\varphi_t)_{t\geq 0}$ is parabolic of positive hyperbolic step if and only if $\Omega$ is contained in a half-plane parallel to the real axes but not contained in a strip.
Hence, by Theorem~\ref{Thm:density-in-half-plane} and Corollary~\ref{Cor:complement}, (1) and (2) are equivalent. Then (2) is equivalent to (3) and (4) by Theorem~\ref{Thm:Cantor-no-approx} and Theorem~\ref{Thm:connect-half-plane}. The last assertion follows from (1) and Lemma~\ref{Lem:approx-weak-approx-p}.
\end{proof}

\begin{proof}[Proof of Theorem~\ref{Thm:main-para-zero}]
The semigroup $(\varphi_t)_{t\geq 0}$ is parabolic of zero hyperbolic step if and only if $\Omega$ is not contained in any half-plane parallel to the real axes. By Proposition~\ref{Prop:easy}.(4), if there is not a half-plane $H$ such that $\Omega\subset H$, then $\Lambda_\infty(\Omega)=\{0\}$, and the first assertion follows.

If there is a half-plane $H$ such that $\Omega\subset H$,  by Theorem~\ref{Thm:density-in-half-plane} and Corollary~\ref{Cor:complement}, (1) and (2) are equivalent.  Since  $H$ cannot be parallel to the real axis, by Remark~\ref{Rem:easy-contain-plane} and Theorem~\ref{Thm:Cantor-no-approx}  we have the equivalence among (2), (3) and (4).
The last assertion follows from (1) and Lemma~\ref{Lem:approx-weak-approx-p}.
\end{proof}

\section{Frequencies  for $p\in[1,\infty)$}\label{Sec:p-freq}

Let $\Omega$ be the Koenigs domain of a non-elliptic semigroup. The aim of this section is to discuss completeness of $\mathcal E_p(\Omega)$ for $p\in[1,\infty)$, giving the proof of Theorem~\ref{Thm:main-density-in-log} and discuss some  necessary conditions for completeness of linear combinations of exponentials.

\subsection{Logarithmic starlike domains at infinity with frequencies and the proof of Theorem~\ref{Thm:main-density-in-log}} We start with the following proposition:

\begin{proposition}\label{Prop:connect-log}
Let $(\varphi_t)_{t\geq 0}$ be a non-elliptic semigroup in $\D$ with Denjoy-Wolff point $\tau\in\partial\D$. Let $h$ be the Koenigs  function, $\Omega=h(\D)$ the Koenigs  domain  and let $\psi:I\to \R$ be the defining function of $\Omega$. Suppose $\hbox{Int}(\overline\Omega)=\Omega$. Assume  $\Omega$ is contained in the rotation of a logarithmic starlike at infinity domain.  Let $I_\infty$ denotes the unbounded connected component of $\psi^{-1}(-\infty)$ (possibly empty).  Then the following are equivalent.
\begin{enumerate}
\item $\C\setminus \overline{\Omega}$ is connected.
\item $(\varphi_t)_{t\geq 0}$ has no boundary fixed points other than $\tau$ and $\mathcal S(h)$ contains only bounded discontinuities except, at most,  one unbounded discontinuity $\sigma\in\partial\D\setminus\{\tau\}$ which is the final point of a maximal contact arc starting at $\tau$.
\item  $\liminf_{I\ni y\to y_0}\psi(y)>-\infty$ for every $y_0\in \overline{I}\setminus \overline{I_\infty}$.
\end{enumerate}
\end{proposition}
\begin{proof}
Observe that if $\Omega$ is contained in $D$, the rotation of a logarithmic starlike at infinity domain, then there is $v\in\C$, $\Re v\leq 0$ and $|v|=1$, such that  for all $w\in D$ there exists $t>0$ so that  $w+tv\in \C\setminus\overline D$. If $v$ is not pure imaginary, then it follows that for all $w\in D$ there exists $t'>0$ such that $w-t'\in \C\setminus\overline D$. Hence, arguing as in Remark~\ref{Rem:easy-contain-plane} we have the result.

If $v$ is pure imaginary, then $D$ is the $\pi/2$ rotation of a logarithmic starlike at infinity domain. Hence, there is a half-plane $H$ parallel to the real axis which is contained in $D$ and such that $\partial H\cap\partial D\neq\emptyset$. We can assume $H=\{w\in\C:\Im w>0\}$ and $x_0\in \R\cap \partial D$ is the maximum of the points in $\partial D\cap\R$.

Since $\Omega$ is starlike at infinity, then
\[
\Omega\subset \{w\in \C:\Im w>0\}\cup \{w\in \C: \Re w>x_0\}.
\]
Then the same argument as in the proof of Theorem~\ref{Thm:connect-half-plane} gives the result.
\end{proof}

With this result at hand, we can give the

\begin{proof}[Proof of Theorem~\ref{Thm:main-density-in-log}]
It follows at once by Theorem~\ref{Thm:density-in-log}, Theorem~\ref{Thm:Cantor-no-approx} and Proposition~\ref{Prop:connect-log}.
\end{proof}

Now, we provide examples of logarithmic starlike at infinity domains with $p$-frequencies, which are not contained in any half-plane:

\begin{proposition}\label{example-log-complete}
Let $\eta:\C_+\to\C$ be defined by $\eta(w):=w-(\log(w+3))^a$, where $a\in (0,1)$ is a constant.  Then $\eta$ is univalent on $\C_+$, and  the domain $\Om_0=\eta(\C_+)$ is a  logarithmic starlike at infinity domain with $p$-frequencies, which is not contained in any half-plane.
\end{proposition}

\begin{proof}	It is easy to see that  $|\eta'-1|<a<1$ on $\C_+$, hence $\Re \eta'>0$ on $\C_+$. Since $\C_+$ is convex, it follows by the Noshiro-Warschawski's theorem (see, {\sl e.g.} \cite[Theorem~3.1.3]{BCDbook}) that $\eta$ is univalent on $\C_+$.
	
	Moreover, the boundary curve of the domain $\eta(\C_+)$ is parametrized  as $\R\ni t \mapsto \alpha(t)+i\beta(t)$, where $\alpha(t):=\Re \eta(it)$ and $\beta(t):=\Im \eta(it)$. Since $\beta'(t)=\Re \eta'(it)>0$ for all $t$  and $\beta(t)$ tends to $\pm\infty$ when $t\to\pm\infty$, it follows that $\beta$ is an increasing diffeomorphism of $\R$ to $\R$. 

	Hence the function  $\psi=\alpha\circ\beta^{-1}$ is differentiable, and $\Om_0=\eta(\C_+)$ has the form $\{x>\psi(y): y\in \R\}$.  Therefore, $\Omega_0$ is a starlike at infinity domain and clearly it is not contained in any half-plane, since $\psi(y)\to-\infty$ as $y\to \pm\infty$.

Since $\beta'(t)=\Re \eta'(it)\in (1-a,1+a)$ for all $t$, it follows that $\psi$ satisfies (1) of Definition~\ref{Def:log-star-dom}. By the same token and since $\beta(0)=0$, we have for all $t$
\[
|\beta^{-1}(t)|=\left| \int_0^{t} \frac{ds}{\beta'(\beta^{-1}(s))}\right|=\left|\int_0^{t} \frac{ds}{\Re \eta'(i\beta^{-1}(s))}\right|\leq \frac{|t|}{1-a}.
\]
Now, there exists some $C>0$ and $R>1$ such that $\al(t)\ge -C\big(\log|t|\big)^a$ for $|t|>R$.
Hence  the previous displayed formula implies
	\[
	\psi(t)=\al(\beta^{-1}(t))\ge -C\big(\log\frac{|t|}{1-a}\big)^a, \quad |t|>R, 
	\]
that is, $\psi$ satisfies (2) of Definition~\ref{Def:log-star-dom}. Therefore, $\Omega_0$ is a  logarithmic starlike at infinity domain.

	Finally, let  $\lam\le 0$. By the conformal invariance of the Hardy spaces,  $e^{\lam z}\in H^p(\Om_0)$ if and only if $e^{\lam \eta(z)}\in H^p(\C_+)$. Take any $b\in (a,1)$. It is easy to see that there exists some $R>0$ 
	such that 
	\[
	p\lam \Re\eta(w)\le p|\lam| \,\Re \big(\log (w+3)\big)^a 
	\le \log \Re (w+3)^b
	, \quad \Re w \ge 0, |w|>R. 
	\]	
	Hence for some $C>0$,
	\[
	p\lam \Re\eta(w)\le C + \log \Re (w+3)^b, \quad w\in \C_+,
	\]		
	which implies that $e^{p\lam \Re\eta(w)}\le e^C \cdot  \Re (w+3)^b$
	in $\C_+$. Since the function $\Re (w+3)^b$ is harmonic in $\C_+$,
	$|e^{\lam \eta(w)}|^p$ has a harmonic majorant in $\C_+$, which proves that $e^{\lam z}\in H^p(\Om_0)$. Thus, $\Omega_0$ has $p$-frequencies, and we are done.	
\end{proof}

As a direct application of  Theorem~\ref{Thm:main-density-in-log} and Proposition~\ref{example-log-complete}, we have:
	
\begin{corollary}
Let $\psi$ be any continuous function on $\R$ such that $\psi(y)>C_1-\big(\log(|y|+3)\big)^a$ for all $y$, where $C_1$ and $a\in (0,1)$ are constants. Then $\Omega:=\{x+iy: x>\psi(y)\}$ is a Koenigs domain of a non-elliptic semigroup, and $\mathcal E_p(\Omega)$ is dense in $H^p(\Omega)$  for all $p\in[1,\infty)$.  		
\end{corollary}

\subsection{Necessary conditions for completeness of $p$-frequencies}	

The aim of this section is to give some necessary conditions for completeness of frequencies in $H^p(\Omega)$. We first show that a necessary condition is that the frequencies are not reduced to an interval:

\begin{proposition}
	\label{lem-finite-interval-frequencies}
Suppose $\Om$ is a Koenigs domain of a non-elliptic semigroup such that $\C_++c\subset \Om$ for some
$c\in \R$. Suppose that for some $p_0\in [1,+\infty)$ and $M>0$,
\[	
\Lambda_{p_0}(\Om)\subset [-M,0].
\]
Then $\mathcal E_p(\Omega)$ is not dense in $H^p(\Omega)$ for every $p\in[1,\infty)$.
\end{proposition}
\begin{proof}

Notice that by Proposition~\ref{Prop:easy}.(1), the condition $\Lambda_{p_0}(\Om)\subset [-M,0]$ implies that
either $\Lambda_{p_0}(\Om)=[-N,0]$ for some $N \ge 0$ or
$\Lambda_{p_0}(\Om)=(-N,0]$ for some $N > 0$. Then, by Proposition~\ref{Prop:easy}.(2), for any $q\in [1,\infty)$,
$\Lambda_q(\Om)$ is also a finite interval.

We show that, if $\Lambda_p(\Omega)$ is a finite interval  then $\mathcal E_p(\Omega)$ is not dense in $H^p(\Omega)$.

We may assume that $c=0$. Arguing by contradiction, we assume that $\mathcal E_p(\Omega)$ is dense in $H^p(\Omega)$.

Let $h:\D\to\Omega$ be a Riemann map. We claim that $h^{-1}$ extends continuously and injectively on $\overline{\C_+}^\infty$. Indeed, if $ir\in \Omega$ for some $r\in\R$ then clearly $h^{-1}$ extends continuously at $ir$. If $ir\in\partial\Omega$, then (see, {\sl e.g.}, \cite[Proposition~3.3.3]{BCDbook}) there exists a point $\sigma\in \partial\D$ such that $\lim_{t\to 0^+}h^{-1}(t+ir)=\sigma$ and $h^\ast(\sigma)=ir$. Now, let $\{z_n\}\subset\C_+$ be a sequence converging to $ir$. We can assume that $\{\Re z_n\}$ is strictly decreasing. Let $\Gamma_n$ be the segment joining $z_n$ to $z_{n+1}$. Hence, $\Gamma=\cup_n \Gamma_n$ is a Jordan arc contained in $\C_+$ (since it is convex) converging to $ir$. It is easy to show (see, {\sl e.g.}, \cite[Proposition~3.3.5]{BCDbook}) that the preimage of $\Gamma$ via $h^{-1}$ converges to $\sigma$. Thus $h^{-1}$ extends univocaly at every $ir$, $r\in \R$. It is also easy to show (by the same token) that there exists $\tau\in\partial\D$ such that $\lim_{\C_+\ni z\to \infty}h^{-1}(z)=\tau$. We set $h^{-1}(\infty)=\tau$. In order to show that $h^{-1}$ is continuous on $\overline{\C_+}$, it is enough to show that if $\{ir_n\}$ is a sequence converging to $ir$ (with $r\in\R\cup\{\infty\}$)  then $\{h^{-1}(ir_n)\}$ converges to $h^{-1}(ir)$. Indeed, let $\Gamma_n$ be a Jordan arc in $\overline{\C_+}^\infty$ such that $ir_n, ir\in \Gamma_n$  and $\Gamma_n\setminus\{ir_n, ir\}\subset\C_+$. We can take $\Gamma_n$ in such a way that its spherical diameter tends to $0$ as $n\to\infty$. Thus, if $\{h^{-1}(ir_n)\}$ does not converge to $h^{-1}(ir)$, there is a subsequence $\{n_k\}$ such that $h^{-1}(\Gamma_{n_k})\cap \D$ has diameter bounded from below by a positive constant. By this violates the no Koebe arcs theorem (see, {\sl e.g.}, \cite[Theorem~3.2.4]{BCDbook}). Thus $h^{-1}$ extends continuously on $\overline{\C_+}$. Finally, it is easy to see that $h^{-1}$ is injective on $\overline{\C_+}$ (see, {\sl e.g.}, \cite[Proposition 3.3.4]{BCDbook}).

Therefore, $G:=h^{-1}(\C_+)$ is a Jordan domain in $\D$. By Mergelyan's approximation theorem, polynomials are dense in the algebra of holomorphic functions which are continuous up to $\overline{G}$, since this last space is dense in $H^p(G)$, it follows that $H^p(\D)$ is dense in $H^p(G)$ and thus  $H^p(\Omega)$ is dense in $H^p(\C_+)$.

By Lemma~\ref{Lem:dense-complete}, the density of $H^p(\Omega)$  in $H^p(\C_+)$ implies that $\mathcal E_p(\Omega)$ is dense in $H^p(\C_+)$.

So to finish the proof it suffices to check that the last statement is false. To this end, choose a nonzero $C^\infty$ function $s$ on $\R_+$, which vanishes on $[0,M]$. Let $\mathcal H^2(\C_+)$ be the Hardy-Smirnov space of the semiplane $\C_+$. By the Paley-Wiener theorem, $(2\pi)^{-1/2}\, \mathcal{L}$ is an isometric isomorphism  of $L^2(\R_+, dx)$ onto $\mathcal H^2(\C_+)$ (here $\mathcal{L}$ is the Laplace transform).
	Set $u=\mathcal{L}s$. Then $u\in \mathcal H^2(\C_+)$,
	and, moreover, $(z+1) u\in H^\infty (\C_+)$.
	The function $u$ gives rise to
	a bounded linear functional on $H^p(\C_+)$, defined by
	\[
	\tilde u(v) =\int_{i\R} u(-z) v(z) |dz|, \quad v\in H^p(\C_+).
	\]
	By the Fourier inversion formula,
	\[
	\tilde u(e^{\lambda z}) =2\pi s(-\lambda), \quad \lambda\leq 0.
	\]
	Since $s$ is nonzero,
	$\tilde u$ is a nonzero functional on $H^p(\C_+)$. On the other hand,
	$s(-\lambda)=0$ for all $\lambda \in \Lambda_p(\Om)$, which implies that $\tilde u$ vanishes on exponentials $e^{\lambda z}$ for
	$\lambda \in \Lambda_p(\Om)$. Hence the space of linear combinations of these exponentials is not dense in $H^p(\C_+)$.
\end{proof}

\begin{example}
Let $\eta(w)=w-\log(w+3)$. The arguments given in Proposition~\ref{example-log-complete}  show that  $\eta$ is univalent on $\C_+$ and  the domain $\Om_0=\eta(\C_+)$ has the form $\Om_0=\{x+iy\in\C:x>\psi_0(y)\}$. 

We claim that $\Lambda_{1}(\Om_0)= (-1,0]$. Hence, by Proposition~\ref{lem-finite-interval-frequencies}, we have that $\mathcal E_p(\Omega_0)$ is not dense in $H^p(\Omega_0)$, for all $p\geq 1$.

In order to prove the claim, let $e^{\lambda z}\in H^1(\Om_0)$, $\lambda\in\C$. Since $\bigcup_{t\geq 0}(\Om_0-t)=\C$, it follows that the semigroup associated to the Koenigs domain $\Om_0$ is parabolic. Hence,  $\Re \lambda\leq 0$ by Theorem~\ref{Betsakos}.

Now, $|e^{\lam\eta(w)}|\asymp |e^{\lambda w}| |w|^{|\lambda|}$ as
$w\in \partial \C_+$ goes to infinity. Taking into account that the harmonic measure
of $\C^+$ is given by $d\omega_{\C_+}(w)=(|w|^2+1)^{-1}/\pi$,
one gets that $e^{\lam z}\in H^1(\Om_0)$ if and only if $e^{\lam\eta(w)}\asymp |e^{\lambda w}||w|^{|\lambda|}\in L^1(d\omega_{\C_+})$. From this it follows easily that if $e^{\lam z}\in H^1(\Om_0)$ then $\Im \lambda=0$ and $|\lambda|<1$. On the other hand, if  $\lambda\in (-1,0]$, the estimate
\[
\lam \Re \eta(w)\le |\lam| \log|w+3|
\]
implies that
\[
|e^{\lam\eta(w)}|\le |w+3|^{|\lam|}\le K\Re (w+3)^{|\lam|}
\]
in $\C_+$. Hence $\Re (w+3)^{|\lam|}$ is a harmonic majorant of $|e^{\lam\eta}|$ in $\C_+$ and $e^{\lambda z}\in H^1(\Om_0)$.

Note that $\Om_0$ is not a logarithmic starlike at infinity domain. 
\end{example}

Next, we see that the existence of two non-exceptional maximal contact arcs with the same final point is an obstruction to $p$-completeness of frequencies (and not only to $\infty$-weak-star completeness):

\begin{proposition}\label{contact arc}
Assume $(\varphi_t)_{t\geq 0}$ is a non-elliptic semigroup with Denjoy-Wolff point $\tau\in\partial\D$. If there exist two contact arcs $A, B$ with final point $\sigma=x_1(A)=x_1(B)\in\partial\D\setminus\{\tau\}$ then $(\varphi_t)_{t\geq 0}$ has no $p$-complete frequencies for every $p\in[1,\infty)$.
\end{proposition}
\begin{proof}
Let $h$ be the  Koenigs function of $(\varphi_t)_{t\geq 0}$. By Lemma~\ref{Lem:cara-contact}, we can find $\epsilon>0$, $z_0\in \C$ and two non-empty open arcs $A'\subset A$ and $B'\subset B$ such that $h$ extends holomorphically through $A'$ and $B'$,
\[
\Omega\cap \{\zeta\in\C: |\zeta-z_0|<\epsilon\}=\{\zeta\in\C: |\zeta-z_0|<\epsilon\}\setminus L,
\]
where
\[
L:=\{w\in\C: |\Re w-\Re z_0|<\epsilon, \Im w=\Im z_0\},
\]
and $h(A')=h(B')=L$.

We claim that $h^{-1}\not\in\overline{\mathcal E(\Omega)}^{H^p(\Omega)}$. Since $h^{-1}\in H^\infty(\Omega)\subset H^p(\Omega)$, this proves the statement.

The claim, using conformal invariance,  is equivalent to $z\not\in\overline{\hbox{span}\{e^{\lambda h(z)}:\lambda\in \Lambda_p(\Omega)\}}^{H^p(\D)}$.

Indeed, for every $\lambda\in \Lambda_p(\Omega)$, it follows that  $e^{\lambda h(A')}=e^{\lambda h(B')}$. Hence, any function $f$ in $\mathcal E_p(\Omega)$ extends holomorphically through $A'$ and $B'$ and $f(A')=f(B')$.

Now, if  $\{P_n\}\subset \mathcal E_p(\Omega)$  converges in $H^p(\D)$ to $z$, since the sequence of boundary values $\{P_n^\ast\}$ converges in $L^p(\partial\D)$ to the boundary values of $z$, it follows that $\{P_n^\ast(e^{i\theta})\}$  converges to $e^{i\theta}$ for almost every $\theta\in[0,2\pi)$.

Recall that $P_n$ extends holomorphically through $A'\cup B'$ and let $A''\subset A'$ be the set of full measure such that $\{P_n(e^{i\theta})\}$  converges to $e^{i\theta}$ and similarly define $B''\subset B'$.

Let $V$ be the (bounded) Jordan domain whose boundary is given by $A'$ and the segment joining the two end points of $A'$. By the Riesz-Privalov theorem, applied to $h|_V$ (see, {\sl e.g.}, \cite[Theorem~6.8]{Pom}), we see that $h(A'')$ has full measure in $h(A')=L$. Similarly, $h(B'')$ has full measure in $h(B')=L$. Therefore, there exist   $q\in A''$ and $p\in B''$ such that $P_n(q)=P_n(p)$ for all $n$, but $\{P_n(p)\}$ converges to $p$ and $\{P_n(q)\}$ converges to $q$, a contradiction.
Therefore, $z\not\in\overline{\hbox{span}\{e^{\lambda h(z)}:\lambda\in \Lambda_p(\Omega)\}}^{H^p(\D)}$, and we are done.
\end{proof}

\end{document}